\RequirePackage{ifthen}
\newboolean{MPA}
\setboolean{MPA}{false}

\ifthenelse {\boolean{MPA}}
{
\documentclass{svjour3}
\smartqed
\usepackage[margin=1.3in]{geometry}
} {
\documentclass[11pt]{article}
\usepackage[margin=1in]{geometry}
}

\usepackage[utf8]{inputenc}
\usepackage{amsmath}
\usepackage{graphicx}
\usepackage{cite}
\usepackage{amsfonts}
\usepackage{booktabs}
\usepackage[colorlinks,citecolor=blue,linkcolor=blue,urlcolor=blue,bookmarks=false]{hyperref}

\makeatletter
\let\cl@chapter\undefined
\makeatletter

\usepackage[capitalize,noabbrev]{cleveref}
\crefname{question}{Question}{Questions}
\crefname{observation}{Observation}{Observations}

\usepackage[title]{appendix}
\usepackage[normalem]{ulem}

\ifthenelse {\boolean{MPA}}
{

\newenvironment{prf}[1][]
{\begin{proof}}
{\qed \end{proof}}

\journalname{Mathematical Programming A}

} {

\usepackage{amsthm}
\newtheorem{theorem}{Theorem}
\newtheorem{lemma}{Lemma}

\newtheorem{definition}{Definition}
\newtheorem{example}{Example}

\newtheorem{corollary}{Corollary}
\newtheorem{question}{Question}

\newenvironment{prf}[1][]
{\begin{proof}}
{\end{proof}}

}

\newtheorem{observation}{Observation}

\usepackage{enumitem}

\newcommand{\abs}[1]{\left\lvert#1\right\rvert}

\newcommand{\pare}[1]{\left(#1\right)}
\newcommand{\pares}[1]{(#1)}
\newcommand{\bra}[1]{\left\{#1\right\}}
\newcommand{\norm}[1]{\left\lVert#1\right\rVert}
\newcommand{\norms}[1]{\lVert#1\rVert}

\DeclareMathOperator\intr{int}
\DeclareMathOperator\obj{obj}
\DeclareMathOperator\argmax{argmax}
\DeclareMathOperator\argmin{argmin}
\DeclareMathOperator\OPT{OPT}

\usepackage{color}
\newcommand{\R}{\mathbb R}

\newcommand{\E}{\mathbb E}
\renewcommand{\Pr}{\mathbb P}
\newcommand{\Ex}{\mathbb E}

\newcommand{\ball}{B^m}
\newcommand{\sphere}{S^{m-1}}

\newcommand{\stt}{\textnormal{s.t.}}

\newcommand\inner[2]{\langle #1, #2 \rangle}

\begin{document}

\title{$k$-median: exact recovery in the extended stochastic ball model\thanks{\textbf{Funding: } This work is supported by ONR grant N00014-19-1-2322. Any opinions, findings, and conclusions or recommendations expressed in this material are those of the authors and do not necessarily reflect the views of the Office of Naval Research.}}

\ifthenelse {\boolean{MPA}}
{
\titlerunning{$k$-median: exact recovery in the extended stochastic ball model}
\authorrunning{Alberto Del Pia, Mingchen Ma}

\author{Alberto Del Pia \and Mingchen Ma}
\institute{Alberto~Del~Pia \at
              Department of Industrial and Systems Engineering 
              \& Wisconsin Institute for Discovery \\
              University of Wisconsin-Madison, Madison, WI, USA \\
              \email{delpia@wisc.edu}
              \and
              Mingchen~Ma \at
              Department of Computer Sciences \\
              University of Wisconsin-Madison, Madison, WI, USA \\
              \email{mma54@wisc.edu}}
}
{
\author{Alberto Del Pia
\thanks{Department of Industrial and Systems Engineering \& Wisconsin Institute for Discovery,
             University of Wisconsin-Madison, Madison, WI, USA.
             E-mail: {\tt delpia@wisc.edu}.}
\and
Mingchen Ma
\thanks{Department of Computer Sciences,
             University of Wisconsin-Madison, Madison, WI, USA.
             E-mail: {\tt mma54@wisc.edu}.}
             }
}

\date{July 19, 2022}

\maketitle


\begin{abstract}
    We study exact recovery conditions for the linear programming relaxation of the $k$-median problem in the stochastic ball model (SBM).
    In Awasthi et al.~(2015),
    the authors give a tight result for the $k$-median LP in the SBM, saying that exact recovery can be achieved as long as the balls are pairwise disjoint. 
    We give a counterexample to their result, thereby showing that the $k$-median LP is not tight in low dimension. 
    Instead, we give a near optimal result showing that the $k$-median LP in the SBM is tight in high dimension. 
    We also show that, if the probability measure satisfies some concentration assumptions, then the $k$-median LP in the SBM is tight in every dimension. 
    Furthermore, we propose a new model of data called extended stochastic ball model (ESBM), which significantly generalizes the well-known SBM.
    We then show that exact recovery can still be achieved in the ESBM.
    \ifthenelse {\boolean{MPA}}
{
} {}
\end{abstract}

\ifthenelse {\boolean{MPA}}
{}{
}

\section{Introduction}

Clustering problems form a fundamental class of problems in data science with a wide range of applications in computational biology, social science, and engineering. 
Although clustering problems are often NP-hard in general, recent results in the literature show that we may be able to solve these problems efficiently if the data exhibits a good structure.
More specifically, we may be able to solve these problems in polynomial time if the problem data is generated according to some reasonable model of data.
These models of data are defined in such a way that there is a ground-truth that reveals which cluster a data point comes from. 
In this way, for each instance of the clustering problem generated according to such model of data, it is clear which optimal solution our algorithm should return.
If the algorithm returns the correct solution, we say that the algorithm ``achieves exact recovery''.
Examples of models of data include the stochastic block model and the stochastic ball model.

One of the most successful types of algorithms to achieve exact recovery in polynomial time are convex relaxation techniques, including linear programming (LP) relaxations and semidefinite programming (SDP) relaxations.
When these algorithms achieve exact recovery, the optimal solution to the convex relaxation is an integer vector which is also the optimal solution to the underlying integer programming problem which models the clustering problem. Recently, much work has been done to understand the phenomenon of exact recovery for convex relaxation methods, with diverse clustering problems and models.
Recent LP relaxations that achieve exact recovery for clustering problems include \cite{NELLORE2015165,awasthi2015relax,delpia2020linear,derosa2020ratiocut}, while some SDP relaxations that achieve exact recovery are \cite{Agarwal2017,Ames2014,AAmes2014,awasthi2015relax,10.1214/17-AOS1545,DBLP:conf/colt/FeiC18,7440870,IguMixPetVil17,li2020when,Ling2020,7298436,chen2014improved,pirinen2019exact,li2021convex}.

In this paper we study the \emph{$k$-median problem,} which is one of the most well-known and studied clustering problems.
We are given a set $P$ of $n$ different points in a metric space $(X,d)$ and a positive integer $k \le n$, and our goal is to partition these $n$ points into $k$ different sets $A_1,A_2,\dots, A_k$, also known as clusters.
Each cluster $A_i$ has a center $a_i \in P,$
which satisfies $\sum_{p \in A_i}d(a_i,p)=\min \{ \sum_{p \in A_i}d(q,p) \mid q \in A_i\},$ and each point in $P$ is assigned to the cluster with the closest center. 
Formally, the \emph{$k$-median problem} is defined as the following optimization problem:
\begin{align*}
\min \ & \sum_{p \in P} \min_{i \in [k]} d(p,a_i) \\
\stt \ & a_1,\dots,a_k \in P.
\end{align*}
The $k$-median problem is NP-hard even in some very restrictive settings, like the Euclidean $k$-median problem on the plane \cite{MegSup84},
and only few very special cases of the $k$-median problem are known to be solvable in polynomial time, like the $k$-median problem on trees \cite{KarHak79,Tam96}.
Several papers study approximation algorithms for the $k$-median problem, including \cite{CHARIKAR2002129,LIN1992245,Koll2007,arya2004local,arora1998polynomial,charikar1999improved}.


The model of data that we consider in this paper, and that is arguably the one used the most in the study of the $k$-median problem, is the stochastic ball model (SBM), formally introduced in \cref{def SBM}.
In the SBM, we consider $k$ probability measures, each one supported on a unit ball in $\R^m$, and $n$ data points are sampled from each of them.
In this paper we study the effectiveness of the LP relaxation to achieve exact recovery.
The main goal of this paper is then to seek for the minimum pairwise distance $\Delta$ between the ball centers which is needed for the LP relaxation to achieve exact recovery with high probability when the number of the input data points $n$ is large enough. 
To the best of our knowledge, the only known result in this direction is Theorem~7 in \cite{awasthi2015relax} (or Theorem~6 in the conference version of the paper \cite{awasthi2015relaxc}).
Unfortunately, as will be discussed later, this result is false.

The SBM has also been used as a model of data for other closely related clustering problems, including \emph{$k$-means} and \emph{$k$-medoids clustering}.
In \cref{tab tab} we summarize the known exact recovery results for clustering problems in the SBM, including some of our results that will discuss later.
For more details about the results in the table, including the additional assumptions required, we refer the reader to the corresponding cited paper.
We remark that the problem considered in \cite{NELLORE2015165} differs from the $k$-median defined in this paper because in the objective function the sum of the \emph{squared} distances is considered.



\begin{table}[htbp]
	\centering
	\begin{small}
	\begin{tabular}{llll}
		\toprule  
		Problem & Method & Sufficient Condition & Reference \\ 
		\bottomrule
		\toprule  
		$k$-means / $k$-median & Thresholding & $\Delta > 4$ & Simple Algorithm \\
		\midrule
		$k$-means & SDP & $\Delta>2\sqrt{2}(1+1/\sqrt{m})$ & Theorem 3 in \cite{awasthi2015relax} \\
		 & SDP & $\Delta>2+k^2/m$ & Theorem 9 in \cite{IguMixPetVil17} \\
		 & SDP & $\Delta>2+O(\sqrt{k/m})$ & Corollary 2 in \cite{li2020when}\\
		 & SDP & $\Delta>O(\sqrt{\log n /m})$ & Corollary of Theorem 3 in \cite{DBLP:conf/colt/FeiC18} \\
		 \midrule
		 $k$-means & LP & $\Delta>4$ & Theorem 9 in \cite{awasthi2015relax} \\
		 & LP & $\Delta>1+\sqrt{3}$ & Theorem 4 in \cite{derosa2020ratiocut} \\
		 \midrule
		 $k$-median & LP & $\Delta>3.75$ & Theorem 6 in \cite{NELLORE2015165} \\
		 & LP & \sout{$\Delta>2$} & \sout{$\text{Theorem 7 in~\cite{awasthi2015relax}}$} \\
		 & LP & $\Delta>3.29$ & \cref{th main3}  \\
		 & LP & $\Delta>2+O(\sqrt{k\log m/m})$ & \cref{th main4} \\
		\bottomrule  
	\end{tabular}
	\caption{Exact recovery results for clustering problems in the SBM.}
	\label{tab tab}
	\end{small}
\end{table}

\subsection{Our contribution}
In \cite{awasthi2015relax}, the authors study the $k$-median problem in the SBM. 
In the model of data considered in the paper, there are $k$ unit balls and $n$ points are sampled from each ball.
The probability measures, supported on each ball, are translations of each other.
Moreover, each probability measure is invariant under rotations centered in the ball center and every neighborhood of each ball center has positive probability measure.
In Theorem 7 in \cite{awasthi2015relax}, the authors claim that, if the unit balls are pairwise disjoint, then the LP-relaxation of the $k$-median problem achieves exact recovery with high probability.
Unfortunately this result is false.
In \cref{ex counter} in \cref{app counter}, we present an example in $\R^2$ where the balls are pairwise disjoint and the probability measures satisfy the assumption of Theorem~7 in \cite{awasthi2015relax}, but when $n$ is large enough, with high probability the LP relaxation does not achieve exact recovery.
Our example implies that to achieve exact recovery, a significant distance between the ball centers is needed.
In \cref{sec problem} we also point out the key problem in the proof of Theorem~7 in \cite{awasthi2015relax}.
Furthermore, we notice that the techniques used in \cite{awasthi2015relax} highly depend on the assumptions that we draw the same number of points from each ball, and that the balls have the same radius and the same probability measure. 
These observations naturally lead to two questions, which are at the heart of this paper.
\begin{question}
\label{Q1}
What is the minimum pairwise distance $\Delta$ between the ball centers which guarantees that the $k$-median LP relaxation in the SBM achieves exact recovery with high probability?
\end{question}

\begin{question}
\label{Q2}
If we relax some of the assumptions in the model of data, 
will exact recovery still happen for the $k$-median LP relaxation?
\end{question}

In this paper, we provide the first answers to \cref{Q1} and \cref{Q2}.
We propose a more general version of the SBM called ESBM, which is a natural model for \cref{Q2} formally defined in \cref{def ESBM}.
In the ESBM, the number of points drawn from each ball can be different, the balls can have different radii and different probability measures.
We study exact recovery for the $k$-median problem in the ESBM. 
Informally, we obtain the following results, where we denote by $c_i$ the center and by $r_i$ the radius of ball $i$.
\begin{itemize}
    \item \cref{th main2}: In the ESBM, if for every $i \neq j$ we have $d(c_i,c_j)>(1+\beta)R+\max\{r_i,r_j\}+O(\sqrt{k \log m /m})$, then the $k$-median LP achieves exact recovery with high probability. 
    Here, $R:=\max_{i \in [k]}r_i$ and $\beta$ is a parameter that measures the difference between the numbers of points sampled from the balls.
    \item \cref{th main3}: In the SBM, if $\Delta>3.29$, then the $k$-median LP achieves exact recovery with high probability. 
    \item \cref{th main4}: 
    In the SBM, if $\Delta>2+O(\sqrt{k\log m/m})$, then the $k$-median LP achieves exact recovery with high probability. 
    \item \cref{th main5}: In the SBM, if $\Delta>2$ and the density function decreases as we increase the distance from the center, then the $k$-median LP achieves exact recovery with high probability.
\end{itemize}
We remark that exact recovery can only be considered when the balls are pairwise disjoint.
Moreover, we need to assume that $d(c_i,c_j)>2\max\{r_i,r_j\}$ for every $i \neq j$, otherwise the ground-truth solution may not be optimal to the $k$-median problem.
In particular, in the SBM we need to have $\Delta >2$.

For the ESBM, \cref{th main2} provides sufficient conditions for exact recovery.
For the SBM, \cref{th main3,th main4} provide the condition $\Delta>\min\{3.29,2+O(\sqrt{k \log m/m})\}$ to guarantee exact recovery. 
This result implies that the $k$-median LP is tight in high dimension.
Furthermore, \cref{th main5} implies that, if we add strong assumptions on the probability measures, then $\Delta>2$ also guarantees exact recovery. 


The rest of the paper is organized as follows.
In \cref{sec IP} we introduce the integer programming formulation \eqref{pr IP} of the $k$-median problem and the corresponding linear programming relaxation \eqref{pr LP}. 
We then provide deterministic necessary and sufficient conditions which guarantee that a feasible solution to \eqref{pr IP} is optimal to \eqref{pr LP} (\cref{th deterministic}).
In \cref{sec Model} we introduce the definition of SBM, ESBM, and exact recovery. 
In \cref{sec Probabilistic condition}, 
we introduce a very general sufficient condition
which ensures that exact recovery happens with high probability (\cref{th prob}).
In \cref{sec truth ESBM}, we will present our main theorems for exact recovery (\cref{th one dim gen,th one dim 1,th main2,th main3,th main4,th main5}). 
Finally, in \cref{sec num}, we perform numerical experiments to illustrate the empirical performance of \eqref{pr LP} under the SBM and the ESBM.


We conclude this section with what we believe is an interesting open question. 
As we already mentioned, in the SBM, exact recovery can only be considered when the balls are pairwise disjoint, otherwise the ground-truth solution may not be optimal to the $k$-median problem. 
In this case, we can set aside the concept of exact recovery and focus instead simply on seeking an optimal solution to the $k$-median problem.
A natural question is whether, in this scenario, we are still able to find an optimal solution to the $k$-median problem by simply solving the LP relaxation.
Interesting models of data that can be considered for this question are the SBM and ESBM with intersecting balls, as well as subgaussian mixture models (SGMMs), where data points are drawn from a mixture of $k$ subgaussian distributions and certain overlaps are allowed.
In fact, SBM can be viewed as a special case of SGMMs. 
Previous work such as \cite{DBLP:conf/colt/FeiC18,10.1093/imaiai/iax001} show that SDP relaxations still have desirable theoretical guarantees for clustering data points under SGMMs. 
On the contrary, to the best of our knowledge, there is no theoretical understanding of the performance of LP relaxations under these more general models of data, where certain overlaps are allowed.

\section{The $k$-median problem via linear programming}
\label{sec IP}

The $k$-median problem can be formulated as an integer linear program as follows.
\begin{align}
	\label{pr IP}
	\tag{IP}
	\begin{split}
	\min \ & \sum_{p,q \in P}d(p,q)z_{pq} \\
	\stt \ & \sum_{p \in P}z_{pq} = 1 \quad \forall q \in P \\
	     \ & z_{pq} \le y_p \quad \forall p,q \in P \\
	     \ & \sum_{p \in P} y_p =k\\
	     \ & y_p,z_{pq} \in \{0,1\} \quad \forall p,q \in P.
	\end{split}
\end{align}
Here, $y_p = 1$ if and only if $p$ is a center, and $z_{pq} = 1$ if and only if $p$ is the center of $q.$ 
The first constraint says that each point is assigned to exactly one center. The second constraint says that $z_{pq} = 1$ can happen only if $p$ is a center. The third constraint says that there are exactly $k$ centers.
It is simple to check that an optimal solution to \eqref{pr IP} provides an optimal solution to the $k$-median problem.

In this paper we consider the linear programming  relaxation of \eqref{pr IP} obtained from \eqref{pr IP} by replacing the constraints $y_p,z_{pq} \in \{0,1\}$ with $y_p,z_{pq} \ge 0$.
Such a linear program, which is given below, has been used in other works in the literature including \cite{CHARIKAR2002129}.
\begin{align}
	\label{pr LP}
	\tag{LP}
	\begin{split}
	\min \ & \sum_{p,q \in P}d(p,q)z_{pq} \\
	\stt \ & \sum_{p \in P}z_{pq} = 1 \quad \forall q \in P \\
	     \ & z_{pq} \le y_p \quad \forall p,q \in P \\
	     \ & \sum_{p \in P} y_p =k\\
	     \ & y_p,z_{pq} \ge 0 \quad \forall p,q \in P.
	\end{split}
\end{align}
The linear program \eqref{pr LP} is called a linear programming relaxation of \eqref{pr IP} because each feasible solution to \eqref{pr IP} is also feasible to \eqref{pr LP}.
The main advantage of \eqref{pr LP} over \eqref{pr IP} is that the first can be solved in polynomial time, while the second is NP-hard.

\subsection{Conditions for the integrality of \eqref{pr LP}}

Let $(\bar y, \bar z)$ be a feasible solution to \eqref{pr IP}.
The main goals of this section are twofold.
First, we provide necessary and sufficient conditions for $(\bar y, \bar z)$ to be  an optimal solution to \eqref{pr LP}.
Second, we give sufficient conditions for $(\bar y, \bar z)$ to be the unique optimal solution to \eqref{pr LP}.
In particular, under these sufficient conditions the $k$-median problem is polynomially solvable.

We start by writing down the the dual linear program of \eqref{pr LP}. 
To do so, we associate 
the dual variables $\alpha_q$ $\forall q \in P$, to the first block of constraints,
the dual variables $\beta_{pq}$ $\forall p,q \in P$, to the second block of constraints,
and the dual variable $\omega$ to the single constraint $\sum_{p \in P} y_p =k$.
We obtain the dual linear program
\begin{align}
    \label{pr LP dual}
    \tag{DLP}
    \begin{split}
    \max \ & \sum_{q \in P}\alpha_q-k\omega \\
    \stt \ & \alpha_q \le \beta_{pq}+d(p,q) \quad \forall p,q \in P \\
         \ & \sum_{q \in P}\beta_{pq} \le \omega \quad \forall p \in P \\
         \ & \beta_{pq} \ge 0 \quad \forall p,q \in P.
    \end{split}
\end{align}

It is simple to see that \eqref{pr LP} always has a finite optimum, thus by the Strong Duality Theorem, so does \eqref{pr LP dual}. In particular, \eqref{pr LP dual} is always feasible.

Let $(y,z)$ be a feasible solution to \eqref{pr LP}, and let  $(\alpha,\beta,\omega)$ be a feasible solution to \eqref{pr LP dual}.
The Complementary Slackness Theorem (see, e.g., Theorem 4.5 in \cite{BerTsi97}), says that the vector $(y, z)$ is optimal to \eqref{pr LP} and $(\alpha,\beta,\omega)$ is optimal to \eqref{pr LP dual} if and only if 
\begin{align}
    & \beta_{pq} \pare{z_{pq}-y_p}=0 && \forall p,q \in P \label{eq cs1} \\
    & z_{pq} \pare{\alpha_q-\beta_{pq}-d(p,q)}=0 && \forall p,q \in P  \label{eq cs2}\\ 
    & y_p \pare{\sum_{q \in P}\beta_{pq}-\omega}=0 && \forall p \in P  \label{eq cs3}.
\end{align}

Now let $(\bar y, \bar z)$ be a feasible solution to \eqref{pr IP}.
Clearly, the vector $(\bar y, \bar z)$ is feasible to \eqref{pr LP}.
Furthermore, let $(\alpha,\beta,\omega)$ be a feasible solution to \eqref{pr LP dual}.
From complementary slackness, the vector $(\bar y, \bar z)$ is optimal to \eqref{pr LP} and $(\alpha,\beta,\omega)$ is optimal to \eqref{pr LP dual} if and only if 
\begin{align}
    & \beta_{pq}=0 && \forall p,q \in P \text{ such that }\bar y_p = 1, \bar z_{pq}=0 \label{eq cs1 int}\\ 
    & \beta_{pq}=\alpha_{q}-d(p,q) && \forall p,q \in P \text{ such that }\bar z_{pq} = 1 \label{eq cs2 int}\\
    & \sum_{q \in P}\beta_{pq}=\omega && \forall p \in P \text{ such that } \bar y_p = 1.\label{eq cs3 int}
\end{align}

Next, we provide an interpretation of the dual variables.
We can interpret $\alpha_q$ as the maximum distance a point $q$ can ``see''. 
We can then interpret $\beta_{pq}$ as the ``contribution'' from $q$ to $p$. 
The above conditions \eqref{eq cs1 int}--\eqref{eq cs3 int}, together with \eqref{pr LP dual} feasibility, can then be interpreted as follows.
When $q$ is not assigned to a center $p$, 
condition \eqref{eq cs1 int} says that $q$ does not contribute to $p$, and
the first constraint in \eqref{pr LP dual} implies that $q$ cannot see $p$.
Vice versa, when $q$ is assigned to a center $p$, 
condition \eqref{eq cs2 int} and the third constraint in \eqref{pr LP dual}, imply that $q$ can see $p$, and that $q$ contributes to $p$.
Hence, a center $p$ is seen exactly by the points in its cluster, which are also the points that contribute to $p$.
Finally, condition \eqref{eq cs3 int} says that the centers of the clusters all get the same contribution $\omega$. 

In the remainder of the paper, we denote by $t_+$ the positive part of a number $t$, i.e., $t_+ := \max\{t,0\}$.
We obtain the following observation regarding \eqref{pr LP dual}.
\begin{observation}
\label{obs dual modify}
Suppose $(\alpha,\beta,\omega)$ is a feasible solution to \eqref{pr LP dual}. 
For each $p,q \in P,$ let $\beta'_{pq}:=(\alpha_q-d(p,q))_+.$ 
Then $(\alpha,\beta',\omega)$ is a feasible solution to \eqref{pr LP dual} with the same objective value.
\end{observation}
In particular, \cref{obs dual modify} implies that there is always an optimal solution to \eqref{pr LP dual} where $\beta_{pq}=(\alpha_q-d(p,q))_+$.
Next, we define the contribution function.
\begin{definition}[Contribution function]
\label{def contribution}
Given $\alpha \in \R^P$, the \emph{contribution function} $C^{\alpha}(z): \R^m \to \R$ is defined by
\begin{align*}
    C^\alpha(z):= \sum_{q \in P}(\alpha_q - d(z,q))_+.
\end{align*}
\end{definition}
According to \cref{obs dual modify}, the contribution function can be seen as the contribution that a point $p \in P$ gets from all points in $P$. 
We are now ready to present our main deterministic result.

\begin{theorem}
\label{th deterministic}
Let $(\bar y, \bar z)$ be a feasible solution to \eqref{pr IP}.
Let $a_i$, $i \in [k]$, be the $k$ points in $P$ such that $\bar y_{a_i} = 1.$
For every $i \in [k]$, let $A_i := \{q \in P \mid \bar z_{a_i q} = 1\}$.
Then $(\bar y, \bar z)$ is optimal to \eqref{pr LP} 
if and only if there exists $\alpha \in \R^P$ such that 
\begin{align}
    & C^\alpha(a_1)=\dots=C^\alpha(a_k) \label{eq Th1 a}\\
    & C^\alpha(q) \le C^\alpha(a_1) && \forall q \in P \setminus \{a_i\}_{i \in [k]} \label{eq Th1 b} \\
    & \alpha_q \ge d(a_i,q) && \forall i \in [k], \ \forall q \in A_i\label{eq Th1 c} \\
    & \alpha_q \le d(a_i,q) && \forall i \in [k], \ \forall q \in P \setminus A_i. \label{eq Th1 d}
\end{align}
Furthermore, if there exists $\alpha \in \R^P$ such that \eqref{eq Th1 a}, \eqref{eq Th1 c} hold, and \eqref{eq Th1 b}, \eqref{eq Th1 d} are satisfied strictly, then $(\bar y, \bar z)$ is the unique optimal solution to \eqref{pr LP}.
\end{theorem}

\begin{prf}
In the first part of the proof we show the `if and only if' in the statement. 
After that, we will show the `uniqueness'.

First, we show the implication from left to right.
Assume that $(\bar y, \bar z)$ is an optimal solution to \eqref{pr LP}.
Then by Strong Duality \eqref{pr LP dual} also has an optimal solution, which we denote by $(\alpha,\beta,\omega)$.
For each $p,q \in P$, let $\beta'_{pq} := (\alpha_q-d(p,q))_+.$
According to \cref{obs dual modify}, $(\alpha,\beta',\omega)$ is also optimal to \eqref{pr LP dual}.
Complementary slackness implies that $(\bar x, \bar y)$ and $(\alpha,\beta',\omega)$ satisfy the complementary slackness conditions \eqref{eq cs1 int}--\eqref{eq cs3 int}.
Note that for every $p \in P$, we have $\sum_{q\in P}\beta'_{pq}  = \sum_{q\in P} (\alpha_q-d(p,q))_+ = C^\alpha(p).$ 
Constraints \eqref{eq Th1 a} are then implied by \eqref{eq cs3 int}, since $C^\alpha(a_i)=\omega$ for every $i \in [k].$
Constraints \eqref{eq Th1 b} are implied by \eqref{eq cs3 int} and the second constraint in \eqref{pr LP dual}. 
Constraints \eqref{eq Th1 c} are implied by \eqref{eq cs2 int} and the third constraint in \eqref{pr LP dual}.
Finally, constraints \eqref{eq Th1 d} are implied by \eqref{eq cs1 int} and the first constraint in \eqref{pr LP dual}.

Next, we show the implication from right to left.
Let $\alpha \in \R^P$ such that \eqref{eq Th1 a}--\eqref{eq Th1 d} are satisfied. 
For every $p,q \in P$, we define $\beta_{pq}:=(\alpha_q-d(p,q))_+$ and we let $\omega:= C^\alpha(a_1).$ 
From \eqref{eq Th1 b}, we know that $(\alpha,\beta,\omega)$ is feasible to \eqref{pr LP dual}. 
We can then check that $(\bar y, \bar z)$ and $(\alpha,\beta,\omega)$ satisfy the complementary slackness conditions \eqref{eq cs1 int}, \eqref{eq cs2 int}, and \eqref{eq cs3 int} due to \eqref{eq Th1 d}, \eqref{eq Th1 c}, and \eqref{eq Th1 a}, respectively. 
We conclude that $(\bar x, \bar y)$ is optimal to \eqref{pr LP}.

\smallskip

To show the `uniqueness' part of the statement, we continue the previous proof (of the implication from right to left) with the additional assumption that \eqref{eq Th1 b}, \eqref{eq Th1 d} are satisfied strictly.

From complementary slackness we also obtain that $(\alpha,\beta,\omega)$ is an optimal solution to \eqref{pr LP dual}.
Let $(y',z')$ be a feasible solution to \eqref{pr LP}. 
Applying complementary slackness to $(y',z')$ and $(\alpha,\beta,\omega)$, we obtain that $(y',z')$ is an optimal solution to \eqref{pr LP} if and only if these two vectors satisfy conditions \eqref{eq cs1}--\eqref{eq cs3}.
Thus, to prove that $(\bar y, \bar z)$ is the unique optimal solution to \eqref{pr LP}, we only need to show that if $(y',z')$ and $(\alpha,\beta,\omega)$ satisfy \eqref{eq cs1}--\eqref{eq cs3}, then $(y',z')=(\bar y, \bar z)$.

Since for every $p \in P \setminus \{a_i\}_{i \in [k]}$, we have $C^\alpha(p)=\sum_{q \in P}\beta_{pq}<\omega$, \eqref{eq cs3} implies that $y'_p=0$ for every $p \in P \setminus \{a_i\}_{i \in [k]}$. 
From the primal constraints $z'_{pq} \le y'_p$ $\forall p,q \in P$, we obtain $z'_{pq}=0$ $\forall p \in P \setminus \{a_i\}_{i \in [k]}$, $\forall q \in P$.
Since for every $i \in [k]$ and for every $q \in P \setminus A_i$, we have $\alpha_q < d(a_i,q)$,
we know from \eqref{eq cs2} that 
$z'_{a_i q}=0$ for every $i \in [k]$ and for every $q \in P \setminus A_i$. 
From the primal constraint $\sum_{p \in P} z'_{pq}=1$ $\forall q \in P$ we then obtain $z'_{a_i q}=1$ for every $i \in [k]$ and for every $q \in A_i$.  
Primal constraints $z'_{pq} \le y'_p$ $\forall p,q \in P$ and $\sum_{p \in P}y'_p=k$ imply $y'_{a_i}=1$ for every $i \in [k]$. 
We have thereby shown $(y',z')=(\bar y, \bar z)$.
\end{prf}

We remark that deterministic sufficient conditions which guarantee that an integer solution to \eqref{pr IP} is an optimal solution to \eqref{pr LP} have also been presented in \cite{NELLORE2015165,awasthi2015relax}. 
The main difference with respect to these known results is that \cref{th deterministic} provides necessary and sufficient conditions.
In this paper, we do not only use \cref{th deterministic} to prove that  \eqref{pr LP} can achieve exact recovery, but we also use it to construct examples where \eqref{pr LP} does not achieve exact recovery.

\section{Models of data and exact recovery}
\label{sec Model}

In \cref{sec IP}, we considered the $k$-median problem in a deterministic setting.
In the remainder of the paper we will instead consider a probabilistic setting.
Furthermore, our discussion of the $k$-median problem so far is very general, as it applies to any given input consisting of $n$ points in a metric space.
In the remainder of the paper, we will only consider the Euclidean space.
Thus we use $d(\cdot,\cdot)$ to denote the Euclidean distance and we use $\norm{\cdot}$ to denote the Euclidean norm. 
We also denote by $\ball_r(c)$ the closed ball of radius $r$ and center $c$ in $\R^m$ and by $\sphere_r(c)$ the sphere of radius $r$ and center $c$ in $\R^m$.
In this paper, unless otherwise stated, we always assume that the radius $r$ of balls is positive, i.e., $r \in \R_+$, where $\R_+ := \{x \in \R \mid x > 0\}$.
On the other hand we allow the radius of spheres to be nonnegative, i.e., $r \in \{x \in \R \mid x \ge 0\}$.
In particular, $\sphere_0(c)$ is the set containing only the vector $c$.

In this paper we will consider two models of data for the $k$-median problem, which are called the \emph{stochastic ball model} and the \emph{extended stochastic ball model}.
Before defining these two models of data, we first introduce our notation for basic probability theory, in particular, our notation follows~\cite{durrett2010probability}. 
Let $(\mu,\Omega,\mathcal{F})$ be a \emph{probability space}, where $\Omega$ is a set of ``outcomes'', $\mathcal{F}$ is a set of ``events'', and $\mu$ is a probability measure.  
The set $\mathcal{F}$ is a $\sigma$-algebra on $\Omega$, and in this paper we will always let $\mathcal{F}$ be the $\sigma$-algebra generated by $\Omega$.
Therefore we will refer to the probability space $(\mu,\Omega,\mathcal{F})$ by simply writing $(\mu,\Omega)$. 
If $A \in \mathcal{F}$ is a event, we use $\bar A$ to denote its complementary event.
We say $X$ is an \emph{$m$-dimensional random vector} if $X$ is a measurable map from $(\Omega,\mathcal{F})$ to $(\R^m,\mathcal{R}^m)$, where $\mathcal{R}^m$ is the $\sigma$-algebra generated by $\R^m$. 
If $m = 1$, we call $X$ a \emph{random variable}. 
In particular, if $(\mu,\Omega,\mathcal{F})$ is a probability space, $\Omega\subseteq \R^m$ and $X$ is the identity map, we say that $X$ is a random vector \emph{drawn according to $\mu.$}
If $X$ is a random variable, we define its \emph{expected value} 
to be $\Ex X=\int_\Omega X(x)d\mu(x)$.

We are now ready to define the \emph{stochastic ball model}.

\begin{definition}
[Stochastic ball model (SBM)]
\label{def SBM}
For every $i \in [k]$, let $(\mu,\ball_1(0))$ be a probability space.
For each $i \in [k]$, draw $n$ i.i.d. random vectors $v_\ell^{(i)}$, for $\ell \in [n]$, according to $\mu$.
The points from cluster $i$ are then taken to be $x^{(i)}_\ell := c_i + v_\ell^{(i)}$, for $\ell \in [n]$.
\end{definition}
Variants of the SBM have been considered in the literature, with different assumptions on the properties that the probability space $(\mu,\ball_1(0))$ should satisfy. 
We refer the reader for example to \cite{IguMixPetVil17}.

In this paper we will also consider a more general model of data, which we call the \emph{extended stochastic ball model}. 
The extended stochastic ball model is more general than the SBM in the following ways: (i) we do not require the balls to have the same radius, (ii) we do not require the probability measure on the balls to coincide, and (iii) we allow to draw different numbers of data points from different balls.

\begin{definition}[Extended stochastic ball model (ESBM)]
\label{def ESBM}
For every $i \in [k]$, let $(\mu_i,\ball_{r_i}(c_i))$ be a probability space.
For each $i \in [k]$, let $\beta_i \ge 1$ and draw $n_i := \beta_i n$ i.i.d. random vectors $x^{(i)}_\ell$, for $\ell \in [n_i]$, according to $\mu_i$.
The points from cluster $i$ are then taken to be  $x^{(i)}_\ell$, for $\ell \in [n_i]$.
\end{definition}

In this paper we will consider three different assumptions on the probability spaces of the form $(\mu_i,\ball_{r_i}(c_i))$ that we consider, namely: 
\begin{enumerate}[label=(a\arabic*)]
\item
\label{ass rotation}
The probability measure $\mu_i$ is invariant under rotations centered in $c_i$;
\item
\label{ass open}
Every open subset of $\ball_{r_i}(c_i)$ containing $c_i$ has positive probability measure;
\item
\label{ass dim}
Every subset of $\ball_{r_i}(c_i)$ with 
zero Lebesgue measure
has zero probability measure.
\end{enumerate}

In this paper we will see that in the ESBM, the linear program \eqref{pr LP} can perform very well in solving the $k$-median problem.
To formalize this notion we define next the concept of \emph{exact recovery}.
\begin{definition}[Exact recovery]
\label{def rt}
We say that \eqref{pr LP} \emph{achieves exact recovery} if it has a unique optimal solution,  such solution is also feasible 
(thus optimal) to \eqref{pr IP}, and it assigns each point to the ball from which it is drawn.
\end{definition}


The reader might wonder why in the definition of the ESBM we assume that $n_i=\beta_i n$ for $i \in [k]$, effectively requiring the $n_i$ to be of the same order.
In \cref{ex diff order} in \cref{app same order} we show that this assumption is needed in order to obtain exact recovery.

\section{Sufficient conditions for exact recovery in the ESBM}
\label{sec Probabilistic condition}

In this section, we introduce general sufficient conditions which guarantee that \eqref{pr LP} achieves exact recovery with high probability in the ESBM.
To state our results we fist introduce the contribution function in the ESBM.


In the original definition (\cref{def contribution}), we assumed that $\alpha$ is a vector in $\R^P$.
When we will consider the contribution function in the ESBM, we will always assume that for every $i \in [k]$ there exists $\alpha'_i \in \R$ such that $\forall \ell \in [n_i]$ we have $\alpha_{x^{(i)}_\ell}=\alpha'_i$. 
For ease of notation, we then define the contribution function in the ESBM.
\begin{definition}[Contribution function in the ESBM]
\label{def contribution ESBM}
Given $\alpha \in \R^k$, the \emph{contribution function in the ESBM}
$C^{\alpha}(z): \R^m \to \R$ is defined by
\begin{align*}
    C^\alpha(z):=\sum_{i \in [k]}\sum_{\ell \in [n_i]}(\alpha_i-d(z,x^{(i)}_\ell))_+.
\end{align*}
\end{definition}
Clearly, given $\alpha \in \R^P$ and $\alpha' \in \R^k$ such that $\forall i \in [k], \ell \in [n_i]$ we have $\alpha_{x^{(i)}_\ell}=\alpha'_i$, the two definitions are equivalent, i.e., $C^\alpha(z) = C^{\alpha'}(z)$ for every $z \in \R^m$. 
Since each $x^{(i)}_\ell$ is a random vector drawn according to $\mu_i$, we define $\bar \Omega:=\prod_{\ell \in [n_1]} \ball_{r_1}(0) \times \dots \times \prod_{\ell \in [n_k]} \ball_{r_k}(c_k)$ and we let $\bar \mu$ be the corresponding joint probability measure for 
$x^{(1)}_1, \dots, x^{(1)}_{n_1}, \dots, x^{(k)}_1, \dots, x^{(k)}_{n_k}$.
Then for every $z \in \R^m$ and for every $\alpha \in \R^k$, $C^\alpha(z)$ is a random variable on the probability space $(\bar \mu, \bar \Omega).$

Next, we define the function $G^\alpha(z)$, which plays a fundamental role in our sufficient conditions.

\begin{definition}
\label{def G}
Given $\alpha \in \R^k$, in the ESBM we define the function
$G^{\alpha}(z): \R^m \to \R$ as
\begin{align*}
    G^\alpha(z):=\frac{1}{n} \Ex C^\alpha(z).
\end{align*}
\end{definition}


\begin{observation} 
\label{obs G}
In the ESBM, we obtain
\begin{align*}
    G^\alpha(z) 
    & = \sum_{i \in [k]} \beta_i\int_{x \in \ball_{r_i}(c_i)}(\alpha_i-d(z,x))_+d\mu_i(x) \\
    & = \sum_{i \in [k]} \beta_i\int_{\ball_{\alpha_i}(z) \cap \ball_{r_i}(c_i)} (\alpha_i-d(z,x)) d\mu_i(x).
\end{align*} 
\end{observation}

\begin{prf}
The expected value of the contribution function is 
\begin{align*}
    \Ex C^\alpha(z)=\sum_{i \in [k]} n_i\int_{x \in \ball_{r_i}(c_i)}(\alpha_i-d(z,x))_+d\mu_i(x).
\end{align*}
Using $n_i=\beta_i n$, for $i \in [k]$, we obtain
\begin{align*}
    G^\alpha(z) 
    = \frac{1}{n}\Ex C^\alpha(z)
    \ & = \sum_{i \in [k]}\beta_i\int_{x \in \ball_{r_i}(c_i)}(\alpha_i-d(z,x))_+d\mu_i(x)\\
    \ & = \sum_{i \in [k]}\beta_i\int_{\ball_{\alpha_i}(z) \cap \ball_{r_i}(c_i)} (\alpha_i-d(z,x)) d\mu_i(x),
\end{align*}
where the last equality holds because $\alpha_i-d(z,x) \ge 0$ if and only if $x \in \ball_{r_i}(c_i) \cap \ball_{\alpha_i}(z)$.
\end{prf}

\begin{observation}
\label{obs G property}
In the ESBM, the function from $\R^{k+m}$ to $\R$ defined by $(\alpha,z) \mapsto G^\alpha(z)$ is continuous.
\end{observation}

\begin{prf}
To prove this observation, it suffices to show that for every compact set $B \subseteq \R^{k+m}$, the function from $B$ to $\R$ defined by $(\alpha,z) \mapsto G^\alpha(z)$ is continuous.
Therefore, let $B \subseteq \R^{k+m}$ be an arbitrary compact set.
From \cref{obs G}, $G^\alpha(z)$ can be written in the form
\begin{align*}
    G^\alpha(z)=\sum_{i \in [k]}\beta_i\int_{x \in \ball_{r_i}(c_i)}(\alpha_i-d(z,x))_+d\mu_i(x).
\end{align*}
Hence, it suffices to show that, for every $i \in [k]$, the function from $B$ to $\R$ defined by $(\alpha,z) \mapsto \int_{x \in \ball_{r_i}(c_i)}(\alpha_i-d(z,x))_+d\mu_i(x)$ is continuous. 

We know that the function from $B \times \ball_{r_i}(c_i)$ to $\R$ defined by $(\alpha,z,x) \mapsto (\alpha_i-d(z,x))_+$ is continuous.
Since $B \times \ball_{r_i}(c_i)$ is a compact set, the Heine–Cantor theorem implies that $(\alpha_i-d(z,x))_+$ is uniformly continuous over $B \times \ball_{r_i}(c_i)$. 
This implies that for every $\epsilon>0$, there is some $\delta>0$, such that for every $x \in \ball_{r_i}(c_i)$ and for every $(\alpha^1,z^1),(\alpha^2,z^2) \in B$, when $\norm{(\alpha^1,z^1)-(\alpha^2,z^2)}<\delta$, we have $\abs{(\alpha^1_i-d(z^1,x))_+-(\alpha^2_i-d(z^2,x))_+}<\epsilon$.
We obtain that
\begin{align*}
    & \abs{\int_{x \in \ball_{r_i}(c_i)}(\alpha^1_i-d(z^1,x))_+-(\alpha^2_i-d(z^2,x))_+d\mu_i(x)} \\
   \le \ & \int_{x \in \ball_{r_i}(c_i)}\abs{(\alpha^1_i-d(z^1,x))_+-(\alpha^2_i-d(z^2,x))_+}d\mu_i(x) \\
   < \ & \epsilon \Pr(x \in \ball_{r_i}(c_i)) \le \epsilon.
\end{align*}
This concludes the proof that, for every $i \in [k]$, the function from $B$ to $\R$ defined by $(\alpha,z) \mapsto \int_{x \in \ball_{r_i}(c_i)}(\alpha_i-d(z,x))_+d\mu_i(x)$ is continuous.
\end{prf}


For ease of notation, given $k$ balls $\ball_{r_i}(c_i) \subseteq \R^m$, for $i \in [k]$, throughout the paper we denote by 
\begin{align*}
D_i := \min \{d(c_i,c_j)-r_i \mid j \in [k], \ j \neq i\}.
\end{align*}
We also give the following definition in order to simplify the language in this paper.
\begin{definition}
Let $(\mu(n),\Omega(n), \mathcal{F}(n))$ be a probability space, which depends on a parameter $n$, and let $A_n \in \mathcal{F}(n)$ be an event which depends on $n$. 
We say that $A_n$ happens \emph{with high probability}, if for every $\delta \in (0,1)$, there exists $N>0$ such that when $n>N$, $\Pr (A_n)>1-\delta$.
\end{definition}
Note that, when we say with high probability, we always mean with respect to the parameter called $n$ in the probability space.
In this paper we use several times the well-known fact that, if a constant number of events happen with high probability, then they also happen together with high probability. 



We are now ready to state the main result of this section.

\begin{theorem}
\label{th prob}
Consider the ESBM.
For every $i \in [k]$, assume that the probability space $(\mu_i,\ball_{r_i}(c_i))$ satisfies \ref{ass rotation}, \ref{ass open}, \ref{ass dim}.
For every $i \in [k]$, denote by $E_i := \Ex d(x,c_i)$, where $x$ is a random vector drawn according to $\mu_i$.
Assume that there exists some $\gamma \in \R$ that satisfies $\max_{i \in [k]}\beta_i(r_i-E_i)<\gamma<\min_{i \in [k]}\beta_i(D_i-E_i)$.
For every $i \in [k]$, let $\alpha_i := E_i+\frac{\gamma}{\beta_i}$ and assume that $c_i$ is the unique point that achieves $\max \{G^{\alpha}(z) \mid z \in \ball_{r_i}(c_i)\}$.
Then \eqref{pr LP} achieves exact recovery with high probability. 
\end{theorem}

Next, we present a corollary of \cref{th prob} for the ESBM with some special structure.

\begin{corollary}
\label{cor prob}
Consider the ESBM.
For every $i \in [k]$, assume that the probability space $(\mu_i,\ball_{r_i}(c_i))$ satisfies \ref{ass rotation}, \ref{ass open}, \ref{ass dim}.  
For every $i \in [k]$, assume $n_i=n$, $r_i=1$, and denote by $E_i := \Ex d(x,c_i)$,  where $x$ is a random vector drawn according to $\mu_i$.
We further assume $E_1 = \cdots = E_k$. 
Assume that there exists some $\alpha' \in \R$ that satisfies $1<\alpha'<\min_{i\neq j}d(c_i,c_j)-1$.
For every $i \in [k]$, let $\alpha_i:=\alpha'$ and assume that $c_i$ is the unique point that achieves $\max \{G^{\alpha}(z) \mid z \in \ball_{r_i}(c_i)\}$.
Then \eqref{pr LP} achieves exact recovery with high probability.
\end{corollary}

In this paper we often use the concept of median.
Let $P$ be a finite set of points in $\R^m$. 
We say that $x_* \in P$ is a \emph{median} of $P$ if $x_* \in \argmin\{\sum_{s \in P}d(x,s) \mid x \in P \}.$

Next, we give an overview of the proof of \cref{th prob}. We first study points that are drawn from a single ball. 
The key observation is that when $n$ is large enough, the median of the points drawn from a single ball is very close to the ball center.
This allows us to characterize the solution corresponding to the ground-truth. 
Then, using Hoeffding's inequality, we prove that with high probability, when we add any small perturbation to $\alpha$, the medians still get most contribution, which in turn implies \eqref{eq Th1 b}.
The condition $\max_{i \in [k]}\beta_i(r_i-E_i)<\gamma<\min_{i \in [k]}\beta_i(D_i-E_i)$ guarantees that with high probability, when we add a very small perturbation to $\alpha$, the resulting $\alpha$ satisfies \eqref{eq Th1 c} and \eqref{eq Th1 d}. 
Finally, using Hoeffding's inequality, we can guarantee that with high probability the choice of $\alpha$ that satisfies \eqref{eq Th1 a} is very close to the parameter $\alpha$ in the statement. 
As a consequence, \eqref{pr LP} achieves exact recovery with high probability according to \cref{th deterministic}.

A careful reader may find that, if we assume that all $E_i$ are the same and all $\beta_i$ equal one, then our proof is similar to Steps 2--4 in the proof of Theorem~7 in \cite{awasthi2015relax}. 
However, we point out here an important difference. 
In Step~2, the authors show that, after adding a small perturbation to $\alpha$, the points that get most contribution in expectation will be the ball centers. 
Then in Step~3, they show that with high probability a special choice of $\alpha$ can be seen as the $\alpha$ in Step~2 plus a small perturbation. 
Finally in Step~4, they use Hoeffding's inequality to show that with high probability the $\alpha$ in Step~3 can make the median in each ball obtain most contribution. 
However, we notice that since Step~4 is conditioned on Step~3, the probability spaces considered in Step~3 and Step~4 are different, so in Step~4 the sequence of random variables considered in Hoeffding's inequality are not independent, and Hoeffding's inequality cannot be used directly.
In our proof this problem is not present.

In \cref{sec single ball,sec several balls} we prove some lemmas that will be used in the proof of \cref{th prob}, which is given in \cref{sec th prob proof}.
Then, in \cref{sec cor prob proof}, we prove \cref{cor prob}.


In this paper we use the standard notation $[a,b]$ for closed segments and $(a,b)$ for open segments in $\R$. 
Throughout the paper this notation is used only when these segments are nonempty.
Therefore, each time we write $[a,b]$ or $(a,b)$ we are also implicitly assuming $a \le b$ and $a < b$, respectively.

\subsection{Lemmas about a single ball}
\label{sec single ball}

In this section we present some lemmas that consider a single probability space of the form $(\mu,\ball_r(0))$.

\begin{lemma}
\label{lm one median}
Let $(\mu,\ball_r(0))$ be a probability space that satisfies \ref{ass dim}. 
Let $x_1,\dots,x_n$ be random vectors drawn i.i.d. according to $\mu$, where $n \ge 3$, and let $\mathcal{M}$ the set of medians of $\{x_\ell\}_{\ell \in [n]}$. 
Then $|\mathcal{M}|=1$ with probability one. 
\end{lemma}

\begin{prf}
Since $\mathcal{M}$ is always non empty, in order to show that $|\mathcal{M}|=1$ with probability one, it suffices to show that we have $|\mathcal{M}| \ge 2$ with probability zero. 

Let $\bar x_1,\dots,\bar x_{n-1} \in \ball_r(0)$.
Then we have 
\begin{align*}
    \Pr(|\mathcal{M}| \ge 2) 
    = \int_{\ball_r(0)} \cdots \int_{\ball_r(0)}\Pr(|\mathcal{M}| \ge 2 \mid x_1=\bar x_1,\dots,x_{n-1}=\bar x_{n-1})d\mu(\bar x_1)\cdots d\mu(\bar x_{n-1}).
\end{align*}
Hence, to prove the lemma it suffices to show that for every $\bar x_1,\dots,\bar x_{n-1} \in \R^m$ we have
\begin{align}
\label{eq lets do it}
    \Pr(|\mathcal{M}| \ge 2 \mid x_1=\bar x_1,\dots,x_{n-1}=\bar x_{n-1})=0.
\end{align}
From \ref{ass dim}, we know that $x_1,\dots,x_{n-1}$ are different points with probability one. 
So it is sufficient to show that \eqref{eq lets do it} holds when $\bar x_1,\dots,\bar x_{n-1}$ are all different. 

Note that $|\mathcal{M}| \ge 2$ implies that there exist $u,v \in [n]$ with $u \neq v$ such that $\sum_{\ell \in [n]} d(x_u,x_\ell)=\sum_{\ell \in [n]}d(x_v,x_\ell).$
So we have
\begin{align*}
    & \Pr(|\mathcal{M}| \ge 2 \mid x_1=\bar x_1,\dots,x_{n-1}=\bar x_{n-1}) \\
    \le \ & \sum_{u,v \in [n], u \neq v} \Pr\pare{\sum_{\ell \in [n]} d(x_u,x_\ell)=\sum_{\ell \in [n]}d(x_v,x_\ell) \mid x_1=\bar x_1,\dots,x_{n-1}=\bar x_{n-1}}.
\end{align*}
Thus, to prove the lemma, it suffices to show that, for every $\bar x_1,\dots,\bar x_{n-1} \in \R^m$ all different, and for every $u,v \in [n]$ with $u \neq v$, we have
\begin{align*}
    \Pr\pare{\sum_{\ell \in [n]} d(x_u,x_\ell)
    =\sum_{\ell \in [n]}d(x_v,x_\ell) \mid x_1=\bar x_1,\dots,x_{n-1}=\bar x_{n-1}}=0.
\end{align*}

Notice that the above event only depends on the choice of $x_{n}$, since $x_1,\dots,x_{n-1}$ are fixed to $\bar x_1,\dots,\bar x_{n-1}$ respectively. 
Thus we define 
\begin{align*}
S := \bra{ x_n \in \R^m \mid \sum_{\ell \in [n]} d(x_u,x_\ell)=\sum_{\ell \in [n]}d(x_v,x_\ell), x_1 = \bar x_1,\dots,x_{n-1} = \bar x_{n-1}}.
\end{align*}
To prove the lemma, it suffices to show that the Lebesgue measure of $S$ is zero.
In fact, \ref{ass dim} then implies that 
$S$ has zero probability measure.
Hence, in the remainder of the proof we show that the Lebesgue measure of $S$ is zero.

We consider separately two cases.
In the first case we assume $u \neq n$ and $v\neq n$.
Then 
\begin{align*}
S = \bra{ x_n \in \R^m \mid 
d(\bar x_u,x_n)-d(\bar x_v,x_n)
=\sum_{\ell \in [n] \setminus \{n\}} d(\bar x_v, \bar x_\ell)-\sum_{\ell \in [n] \setminus \{n\}} d(\bar x_u,\bar x_\ell) }.
\end{align*}
We define the function $f : \R^m \to \R$ defined by
\begin{align*}
f(x_n) := 
d(\bar x_u,x_n)-d(\bar x_v,x_n)
- \sum_{\ell \in [n] \setminus \{n\}} d(\bar x_v, \bar x_\ell) + \sum_{\ell \in [n] \setminus \{n\}} d(\bar x_u,\bar x_\ell).
\end{align*}
Note that $S$ is the zero set of $f$.
The function $f(x_n)$ is a real analytic function on the connected open domain $\R^m \setminus \{\bar x_u,\bar x_v\}$ since the distance function can be written as a composition of exponential functions, logarithms and polynomials.
Furthermore, $f(x_n)$ is not identically zero, since it increases as $x_n$ moves on the segment from $\bar x_v$ to $\bar x_u$.
From Proposition~1 in \cite{Mit20}, we obtain that $S$ has zero Lebesgue measure.

In the second case we assume $u=n$ and $v \neq n$.
Then 
\begin{align*}
S = \bra{ x_n \in \R^m \mid 
\sum_{\ell \in [n]\setminus \{n\} \setminus \{v\}} d(x_n,\bar x_\ell) 
=\sum_{\ell \in [n] \setminus \{n\}} d(\bar x_v,\bar x_\ell) }.
\end{align*} 
We define the function $f : \R^m \to \R$ defined by
\begin{align*}
f(x_n) := 
\sum_{\ell \in [n]\setminus \{n\} \setminus \{v\}} d(x_n,\bar x_\ell) 
-\sum_{\ell \in [n] \setminus \{n\}} d(\bar x_v,\bar x_\ell).
\end{align*}
Also in this case $S$ is the zero set of $f$.
As in the previous case, the function $f(x_n)$ is a real analytic function on the connected open domain $\R^m \setminus \{\bar x_1,\dots,\bar x_{v-1},\bar x_{v+1},\dots,\bar x_{n-1}\}$.
Furthermore, it is not identically zero, as it increases as the norm of $x_n$ goes to infinity.
Again from Proposition~1 in \cite{Mit20}, we obtain that $S$ has zero Lebesgue measure.
So in both cases we have shown that the Lebesgue measure of $S$ is zero.
\end{prf}

The next two lemmas state that, under some assumptions on the probability space $(\mu,\ball_r(0))$, the vector $z=0$ is the unique point that achieves $\min \{\E d(z,y) \mid z \in \ball_r(0)\}$, where $y$ be a random vector drawn according to $\mu$.
In \cref{lm ball center m = 1} we consider the case $m=1$ and in \cref{lm ball center m >= 2} we study the case $m \ge 2$.

\begin{lemma}
\label{lm ball center m = 1}
Let $(\mu,B^1_r(0))$ be a probability space that satisfies \ref{ass rotation}, \ref{ass open}. 
Let $y$ be a random vector drawn according to $\mu$.
Then $z=0$ is the unique point that achieves $\min \{\E d(z,y) \mid z \in B^1_r(0)\}$.
\end{lemma}

\begin{prf}
We show that for every $z \neq 0$, we have $\E d(z,y)>\E d(0,y)$. 
Let $z \in [-r,r] \setminus \{ 0\}$.
Then we have 
\begin{align*}
    \E d(z,y)= \int_{-r}^z (z-y)d\mu(y)+ \int_z^r(y-z)d\mu(y).
\end{align*}
Without loss of generality, we assume that $z>0$. 
We then have
\begin{align*}
 \E d(z,y)- \E d(0,y) = \int_{-r}^{-z}zd\mu(y)+\int_{-z}^0zd\mu(y)+\int_0^z (z-2y) d\mu(y)+\int_z^r-zd\mu(y).
\end{align*}
Since $\mu$ satisfies \ref{ass rotation}, we have $\int_{-r}^{-z}zd\mu(y) = \int_z^r zd\mu(y)$ and $\int_{-z}^0zd\mu(y) = \int_0^z z d\mu(y)$.
So we obtain
\begin{align*}
 \E d(z,y)- \E d(0,y) = \int_{-z}^0zd\mu(y)+\int_0^z (z-2y) d\mu(y) = 2 \int_0^z(z-y)d\mu(y)>0,
\end{align*}
where the inequality holds due to \ref{ass open}.
\end{prf}

\begin{lemma}
\label{lm ball center m >= 2}
Let $(\mu,\ball_r(0))$ be a probability space with $m \ge 2$ that satisfies \ref{ass rotation}.
Let $y$ be a random vector drawn according to $\mu$.
Then $z=0$ is the unique point that achieves $\min \{\E d(z,y) \mid z \in \ball_r(0)\}$.
\end{lemma}

\begin{prf}
Note that we can write any $z \in \ball_r(0)$ as $z=tv$, for a unit vector $v$ and a scalar $t \in [0,r]$.
Since $\mu$ is invariant under rotations centered in the origin, to prove the lemma it suffices to show that for any fixed unit vector $v$, $t=0$ is the unique point that achieves $\min \{\E d(tv,y) \mid t \in [0,r]\}$.
To prove the lemma it is sufficient to show that 
\begin{align}
\label{eq L7 todo}
    \frac{\partial}{\partial t} \E d(tv,y)
    >0 \qquad \forall t \in (0,r).
\end{align}
In fact, we notice that $\E d(tv,y)$ is a continuous function in $t \in [0,r]$, since for every $\epsilon>0$ and for every $t,t' \in [0,r]$ with $\abs{t-t'}<\epsilon$, we have 
\begin{align*}
\abs{\E d(tv,y) - \E d(t'v,y)}
= \abs{\E( d(tv,y)-d(t'v,y))} 
\le \abs{tv-t'v} 
= \abs{t-t'} 
< \epsilon. 
\end{align*}
Hence, if \eqref{eq L7 todo} holds, then by the Newton-Leibniz formula, we have
\begin{align*}
    \E d(sv,y) - \E d(0,y) = \int_0^s  \frac{\partial}{\partial t} \E d(tv,y)dt>0 \qquad \forall s>0.
\end{align*}
Thus, in the remainder of the proof we show \eqref{eq L7 todo}.

We know that
\begin{align*}
    \E d(tv,y) = \int_{\ball_r(0)} d(tv,y) d\mu (y),
\end{align*}
thus we obtain
\begin{align}
\label{eq to be evaluated}
    \frac{\partial}{\partial t} \E d(tv,y)
    = \frac{\partial}{\partial t} \int_{\ball_r(0)} d(tv,y) d\mu (y)
    = \int_{\ball_r(0)} \frac{\partial}{\partial t} d(tv,y) d\mu (y)
    = \int_{\ball_r(0)} \frac{\inner{tv-y}{v}}{d(tv,y)} d\mu (y),
\end{align}
where $\inner{\cdot}{\cdot}$ denotes the scalar product.
In the remainder of the proof, for $s \ge 0$, we denote by $\mu^s$ the uniform probability measure with support $\sphere_s(0)$.
Since $\mu$ is invariant under rotations centered in the origin, we know that a vector $y$ with $\norm{y}=s$, $s \in [0,r]$, is drawn according to $\mu^s$.

We evaluate \eqref{eq to be evaluated} in a fixed $\bar t \in (0,r)$.
Let $\hat{\mu}$ be the probability measure of the random variable $\norm{y}$ and let $x$ be a random vector drawn according to $\mu^s$. 
We have 
\begin{align}
\label{eq L1}
    \frac{\partial}{\partial t} \E d(tv,y)\Bigr\rvert_{t=\bar t}
    = \int_{\ball_r(0)} \frac{\inner{\bar t v-y}{v}}{d(\bar t v,y)} d\mu (y)
    = \int_0^rd\hat{\mu}(s)\int_{\sphere_s(0)}\frac{\inner{\bar t v-x}{v}}{d(\bar t v,x)} d\mu^s(x).
\end{align}

Next, we study the inner integral in \eqref{eq L1} and consider two subcases.
In the first subcase we have $s \in [0,\bar t]$, and obtain
\begin{align*}
    \inner{\bar t v-x}{v}
    =\bar t-\inner{x}{v} 
    \ge \bar t-\norm{x} 
    \ge 0,
\end{align*}
where the chain of inequalities holds at equality if and only $x=\bar tv$. So we obtain that the inner integral in \eqref{eq L1} is strictly positive when $s \in (0,\bar t]$. 

In the second subcase we have $s \in (\bar t,r]$.
We define the random variable $\theta \in [0,\pi]$ to be the angle between $x$ and $v$ and we let $\Tilde{\mu}$ be its probability measure. 
We also define the random variable $\psi \in [0,\pi]$ to be the angle between $v$ and $\bar t v-x$, and the random variable $\phi \in [0,\pi)$ to be the angle between $x$ and $x-\bar t v$. 
See \cref{fig Lemma 1} for a depiction of the angles $\theta, \psi, \phi$.
Note that once $\theta$ is determined, since $\bar t$ is fixed, $\psi$ and $\phi$ are also determined. 
Therefore, we can consider the functions $\psi, \phi : [0,\pi] \to [0,\pi]$ that associate to each angle angle $\theta$, the corresponding angles $\psi(\theta)$ and $\phi(\theta)$. 
Then we know that for every $\theta \in [0,\pi]$, $\psi(\theta)=\pi-\phi(\theta)-\theta \le \pi-\theta$, and, when $\theta \in (0,\pi), \ \psi(\theta)<\pi-\theta.$
 \begin{figure}[htbp]
     \centering
     \includegraphics[scale=0.11]{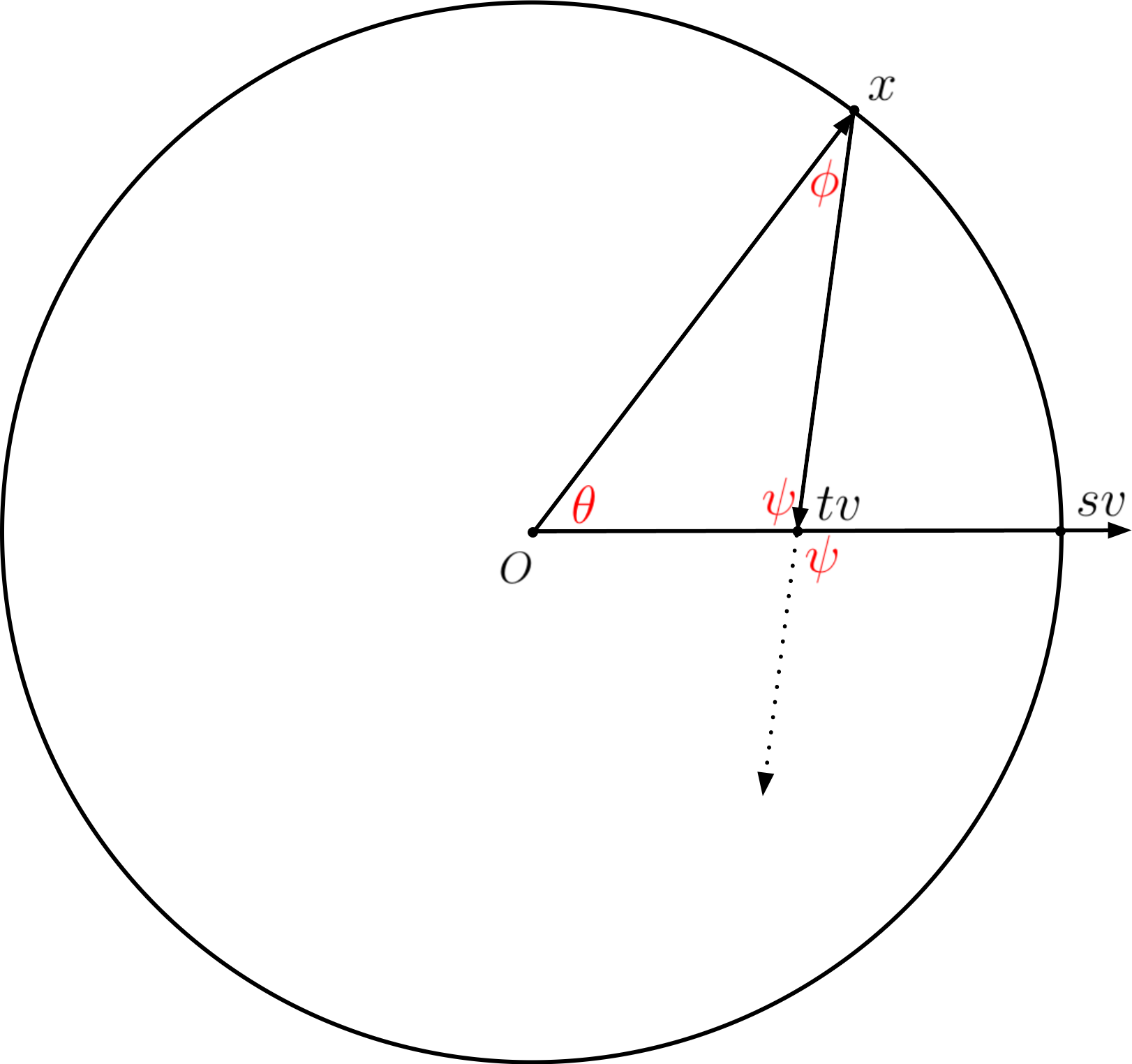}
     \caption{The angles $\theta, \psi, \phi$ in the proof of \cref{lm ball center m >= 2}. The dotted vector is $\bar t v - x$ applied to $\bar t v$.}
     \label{fig Lemma 1}
 \end{figure}
We then have
\begin{align*}
    \int_{\sphere_s(0)}\frac{\inner{\bar t v-x}{v}}{d(\bar t v,x)} d\mu^s(x)
    =\int_0^\pi\cos \psi(\theta)d\Tilde{\mu}(\theta)
    =\int_0^\frac{\pi}{2}\cos \psi(\theta)+\cos\psi(\pi-\theta)d\Tilde{\mu}(\theta) 
    > 0.
\end{align*}
In the above formula, the second equality use the fact that $\Tilde{\mu}$ is symmetric with respect to $\theta=\pi/2$, and the inequality follows because, when $m \ge 2$ and $\theta \in (0,\pi)$, we have $\psi(\theta) < \pi-\theta$ and $\psi(\pi-\theta) < \theta$, which implies $\cos \psi(\theta)+\cos\psi(\pi-\theta)> \cos (\pi-\theta) + \cos \theta= 0$, when $\theta \in (0,\pi/2)$. 
We obtain that the inner integral in \eqref{eq L1} is strictly positive when $s \in (\bar t,r]$. 
Thus we conclude that \eqref{eq L1} is positive.
\end{prf}

In the next lemma we make use of \cref{lm ball center m = 1,lm ball center m >= 2}.

\begin{lemma}
\label{lm improve center in one ball}
Let $(\mu,\ball_r(0))$ be a probability space that satisfies \ref{ass rotation}, \ref{ass open}.
Let $x_1,\dots,x_n$ be random vectors drawn i.i.d. according to $\mu$, and let $x_*$ be a median of $\{x_\ell\}_{\ell \in [n]}$.
Then $\forall \epsilon>0$, with high probability, we have $\norm{x_*}<\epsilon.$ 
\end{lemma}

\begin{prf}
Let $y$ be a random vector drawn according to $\mu$. 
We know from \cref{lm ball center m = 1,lm ball center m >= 2} that $x=0$ is the unique point that achieves $\min \{\E d(x,y) \mid x \in \ball_r(0)\}$.
Furthermore, since the function from $\ball_r(0)$ to $\R$ defined by $x \mapsto \E d(x,y)$ is continuous, we know that for every $\epsilon \in (0,r)$, there is some $\tau$ with $0<\tau<\epsilon<r$ and some $\xi>0$ such that for each $x \in \ball_r(0) \setminus \ball_{\epsilon}(0)$ and for each $x' \in \ball_{\tau}(0)$, we have $\E d(x,y)-\E d(x',y)>\xi.$ %
Let $x_{\min}:=\argmin \{ \norm{x} \mid x \in \{x_\ell\}_{\ell \in [n]} \}$ and notice that $x_* = \argmin \{ \sum_{\ell \in [n]}d(x,x_\ell)/n \mid x \in \{x_\ell\}_{\ell \in [n]} \}.$

We observe that to prove the lemma it suffices to show that $\norm{x_{\min}} \le \tau$, that $\sum_{\ell \in [n]}d(x_u,x_\ell)/n-\E d(x_u,y)\ge -\xi/2$ for every $u \in [n]$ with $x_u \in \ball_r(0) \setminus \ball_{\epsilon}(0)$, and that $\sum_{\ell \in [n]}d(x_v,x_\ell)/n-\E d(x_v,y) \le \xi/2$ for every $v \in [n]$ with $x_v \in \ball_\tau(0)$.
In fact, under these assumptions we obtain that for every $u \in [n]$ with $x_u \in \ball_r(0) \setminus \ball_{\epsilon}(0)$ and for every $v \in [n]$ with $x_v \in \ball_\tau(0)$, we have
\begin{align*}
    & \frac{\sum_{\ell \in [n]}d(x_{v},x_\ell)}{n}-\frac{\sum_{\ell \in [n]}d(x_u,x_\ell)}{n}= \\
    & = \pare{\frac{\sum_{\ell \in [n]}d(x_{v},x_\ell)}{n}-\E d(x_v,y)}-\pare{\frac{\sum_{\ell \in [n]}d(x_{u},x_\ell)}{n}-\E d(x_u,y)}-(\E d(x_u,y)-\E d(x_v,y)) \\
    & < \frac{\xi}{2}+\frac{\xi}{2}-\xi=0.
\end{align*}
Since $\norm{x_{\min}} \le \tau$, the above expression implies that $\norm{x_*} \le \epsilon$.

Inspired by the above observation, we define the following events.
We denote by $A$ the event that $\norm{x_*} \le \epsilon$ and we denote by $T$ the event that $\norm{x_{\min}} \le \tau$. 
For every $w \in [n]$, we denote by $M_w$ the event that at least one of the following events happens:
\begin{itemize}
    \item $x_w \in \ball_\tau(0)$ and $\sum_{\ell \in [n]}d(x_w,x_\ell)/n-\E d(x_w,y) \le \xi/2$;
    \item $x_w \in \ball_\epsilon(0)\setminus \ball_\tau(0)$;
    \item $x_w \in \ball_r(0) \setminus \ball_{\epsilon}(0)$ and $\sum_{\ell \in [n]}d(x_w,x_\ell)/n-\E d(x_w,y)\ge -\xi/2$.
\end{itemize}
In the remainder of the proof, we denote by $\bar E$ the complement of an event $E$.
We know that if $T$ and $M_w$, for all $w \in [n]$, are true then $A$ is true. 
So we get 
\begin{align}\label{eq pra}
    \Pr(A) \ge \Pr\pare{T \cap \bigcap_{w \in [n]} M_w} = 1-\Pr\pare{\bar T \cup \bigcup_{w \in [n]} \bar M_w} \ge 1-\Pr(\bar T)-\sum_{w \in [n]} \Pr(\bar M_w).
\end{align}
We next upper bound $\Pr(\bar T)$ and $\Pr(\bar M_w)$.

We define $p:=\Pr(y \in \ball_r(0) \setminus \ball_\tau(0))$ and obtain 
\begin{align}\label{eq prtb}
    \Pr(\bar T)=p^n.
\end{align}
Since $(\mu,\ball_r(0))$ satisfies \ref{ass open}, we know that $p < 1$.

For $w \in [n]$, we know that $\bar M_w$ is true if and only if at least one of the following event is true:
\begin{itemize}
    \item[$P_w$:] $x_w \in \ball_\tau(0)$ and $\sum_{\ell \in [n]}d(x_w,x_\ell)/n-\E d(x_w,y) > \xi/2$;
    \item[$Q_w$:] $x_w \in \ball_r(0) \setminus \ball_{\epsilon}(0)$ and $\sum_{\ell \in [n]}d(x_w,x_\ell)/n-\E d(x_w,y) < -\xi/2$.
\end{itemize}
Hence in the following we will upper bound separately $\Pr(P_w)$ and $\Pr(Q_w)$.

We start by analyzing $P_w$. 
For every $z \in \ball_\tau(0)$, we have 
\begin{align*}
    & \Pr\pare{\frac{\sum_{\ell \in [n]}d(z,x_\ell)}{n}-\E d(z,y) > \frac{\xi}{2} \mid x_w=z}
     =\Pr\pare{\frac{\sum_{\ell \neq w}d(z,x_\ell)}{n}-\E d(z,y) > \frac{\xi}{2} \mid x_w=z} \\
     & = \Pr \pare{\frac{\sum_{\ell \neq w}d(z,x_\ell)}{n}-\E d(z,y) > \frac{\xi}{2}} 
      = \Pr\pare{\sum_{\ell \neq w}d(z,x_\ell)-(n-1)\E d(z,y) > \frac{n\xi}{2}+\E d(z,y) } \\
     & \le \exp \pare{ -\frac{2(n\xi/2+\E d(z,y))^2}{(n-1)r^2} }
     \le \exp \pare{ -\frac{n\xi^2}{2r^2} }.
\end{align*}
Here, the second equality holds because $x_\ell$, for $\ell \in [n]$ are independent. 
In the first inequality, we use the Hoeffding's inequality and the fact that $d(z,x_\ell) \in [0,r]$.
The last inequality follows because $\E d(z,y) \ge 0$. So we get 
\begin{align*}
    \Pr(P_w)
     & = \int_{\ball_\tau(0)}\Pr\pare{\frac{\sum_{\ell \in [n]}d(z,x_\ell)}{n}-\E d(z,y) > \frac{\xi}{2} \mid x_w=z} d\mu(z) \\
     & \le \sup\bra{ \Pr\pare{\frac{\sum_{\ell \in [n]}d(z,x_\ell)}{n}-\E d(z,y) > \frac{\xi}{2} \mid x_w=z} \mid z \in \ball_\tau(0)} \le \exp \pare{ -\frac{n\xi^2}{2r^2} }.
\end{align*}

Next, we analyze in a similar way $Q_w$.
We have
\begin{align*}
    \Pr(Q_w) 
& = \int_{\ball_r(0) \setminus \ball_{\epsilon}(0)}\Pr\pare{\frac{\sum_{\ell \in [n]}d(z,x_\ell)}{n}-\E d(z,y) < - \frac{\xi}{2} \mid x_w=z} d\mu(z) \le \exp \pare{ -\frac{n\xi^2}{4r^2} },
\end{align*}
because for every $z \in \ball_r(0) \setminus \ball_{\epsilon}(0)$, we have
\begin{align*}
& \Pr \pare{\frac{\sum_{\ell \in [n]}d(z,x_\ell)}{n}-\E d(z,y) < -\frac{\xi}{2} \mid x_w=z} 
= \Pr \pare{\frac{\sum_{\ell \neq w}d(z,x_\ell)}{n}-\E d(z,y) < -\frac{\xi}{2} \mid x_w=z} \\
& = \Pr \pare{\frac{\sum_{\ell \neq w}d(z,x_\ell)}{n}-\E d(z,y) < -\frac{\xi}{2}} 
 =  \Pr \pare{\sum_{\ell \neq w}d(z,x_\ell)-(n-1)\E d(z,y) < -\pare{\frac{n\xi}{2}-\E d(z,y)}} \\
& \le \exp \pare{ -\frac{2(n\xi/2-\E d(z,y))^2}{(n-1)r^2} }
\le \exp \pare{ -\frac{2(n\xi/2-r)^2}{nr^2} }
\le \exp \pare{ -\frac{n\xi^2}{4r^2} },
\end{align*}
where the last inequality holds when $n>4r/((2-\sqrt{2})\xi)$. 

In the rest of the proof, we assume that $n>4r/((2-\sqrt{2})\xi)$.
Using the union bound, we have 
\begin{align*}
    \Pr(\bar M_w) \le \Pr(P_w) + \Pr(Q_w) 
    \le \exp \pare{ -\frac{n\xi^2}{2r^2} } +\exp \pare{ -\frac{n\xi^2}{4r^2} } \le 2\exp \pare{ -\frac{n\xi^2}{4r^2} }.
\end{align*}
Using \eqref{eq pra} and \eqref{eq prtb}, we obtain
\begin{align*}
    \Pr(A) 
    \ge 1-\Pr(\bar T)-\sum_{w \in [n]} \Pr(\bar M_w)
    \ge 1-p^n-2n\exp \pare{ -\frac{n\xi^2}{4r^2} }.
\end{align*}
The latter quantity goes to $1$ as $n$ goes to infinity because $p < 1$ and $p,\xi$ and $r$ are all parameters that do not depend on $n$. 
So with high probability we have $\norm{x_*} \le \epsilon$.
\end{prf}

In the next lemma we use \cref{lm improve center in one ball}.

\begin{lemma}
\label{lm opt}
Let $(\mu,\ball_r(0))$ be a probability space that satisfies \ref{ass rotation}, \ref{ass open}.
Let $x_1,\dots,x_n$ be random vectors drawn i.i.d. according to $\mu$, and let $x_*$ be a median of $\{x_\ell\}_{\ell \in [n]}$.
Let $E:=\E\norm{x}$, where $x$ is a random vector drawn according to $\mu$. 
Let $\OPT:=\sum_{\ell \in [n]}d(x_*,x_\ell).$ 
Then for each $\epsilon>0$, with high probability we have  $|\OPT / n - E|<\epsilon.$
\end{lemma}

\begin{prf}
Let $\epsilon>0$.
We apply \cref{lm improve center in one ball} and we know that with high probability, we have $\norm{x_*} < \epsilon/2$.  
This implies that, with high probability, we have $|d(x_\ell,x_*)-d(x_\ell,0)| < \epsilon/2$ for each $\ell \in [n].$
Summing the latter $n$ inequalities, we obtain that with high probability we have 
\begin{align}
\label{eq lm opt 1}
    \abs{\frac{\OPT}{n}-\frac{\sum_{\ell \in [n]}\norm{x_\ell}}{n}} < \frac{\epsilon}{2}.
\end{align}

On the other hand, according to Hoeffding's inequality, 
\begin{align*}
    \Pr \pare{ \abs{ \frac{\sum_{\ell \in [n]}\norm{x_\ell}}{n}-E}<\frac{\epsilon}{2} } 
    > 1-2\exp \pare{-\frac{n\epsilon^2}{2r^2}}.
\end{align*}
Since $\exp (-n\epsilon^2/(2r^2))$ goes to zero as $n$ goes to $+ \infty$, with high probability we have 
\begin{align}
\label{eq lm opt 2}
\abs{\frac{\sum_{\ell \in [n]}\norm{x_\ell}}{n} - E} < \frac \epsilon 2.
\end{align}

From \eqref{eq lm opt 1} and \eqref{eq lm opt 2}, with high probability we have
\begin{align*}
\abs{\frac \OPT n - E} 
\le
\abs{\frac{\OPT}{n}-\frac{\sum_{\ell \in [n]}\norm{x_\ell}}{n}}
+
\abs{\frac{\sum_{\ell \in [n]}\norm{x_\ell}}{n} - E}
< \frac \epsilon 2 + \frac \epsilon 2
= \epsilon.
\end{align*}
\end{prf}

\subsection{Lemmas about several balls}
\label{sec several balls}

While in \cref{sec single ball} we only considered one ball $\ball_r(0)$, in the three lemmas presented in this section we will consider $k$ balls $\ball_{r_i}(c_i)$, for $i \in [k]$.

\begin{lemma}
\label{lm geometry}
Let $\ball_{r_i}(c_i)$, for $i \in [k]$, be $k$ balls in $\R^m$.
For every $i\in [k]$, assume $r_i < D_i$
and let $[a_i,b_i] \subset (r_i,D_i)$.
Then there exist $\tau_i>0$, $\forall i \in [k]$, such that $\forall i,j \in [k]$, $\forall z \in \intr \ball_{\tau_i}(c_i)$, and $\forall \alpha_j \in [a_j,b_j]$, we have
\begin{align}
\label{eq lm geometry}
    \ball_{\alpha_j}(z) \cap \ball_{r_j}(c_j) =
\begin{cases}
\ball_{r_i}(c_i) & \text{if } j=i\\
\emptyset & \text{otherwise.}
\end{cases}
\end{align}
\end{lemma}

\begin{prf}
We first show the following claim, obtained from the statement of the lemma by fixing some $\alpha_i \in (r_i,D_i)$, for $i \in [k]$.
Let $\ball_{r_i}(c_i)$, for $i \in [k]$, be $k$ balls in $\R^m$, assume $r_i < D_i$, and let $\alpha_i \in (r_i,D_i)$.
Then, for every $i \in [k]$, there exists  $\tau_i(\alpha)>0$ such that $\forall z \in \intr \ball_{\tau_i(\alpha)}(c_i)$, we have \eqref{eq lm geometry}.

To prove the claim 
we choose, for every $i \in [k]$, $\tau_i(\alpha):=\min\{\alpha_i-r_i,D_1-\alpha_1,\dots,D_k-\alpha_k\} > 0.$
Let $z \in \intr \ball_{\tau_i(\alpha)}(c_i)$.
We first show that $\ball_{\alpha_i}(z) \cap \ball_{r_i}(c_i)=\ball_{r_i}(c_i).$
We only need to prove that for each $x \in \ball_{r_i}(c_i)$, we have $x \in \ball_{\alpha_i}(z)$, and this holds because
\begin{align*}
    d(x,z) \le d(x, c_i) + d(c_i, z) < r_i + \tau_i(\alpha) \le r_i + (\alpha_i - r_i) = \alpha_i.
\end{align*}
Next we show that we have $\ball_{\alpha_j}(z) \cap \ball_{r_j}(c_j)=\emptyset$ for $j \neq i$. 
We only need to prove that for each $x \in \ball_{r_j}(c_j)$, we have $x \not\in \ball_{\alpha_j}(z)$.
We have 
\begin{align*}
    d(x,c_i)
    \ge d(c_i,c_j) - d(x,c_j)
    & \ge d(c_i,c_j)-r_j
    \ge D_j,
\end{align*}
thus
\begin{align*}
    d(x,z) 
    & \ge d(x,c_i)-d(c_i,z) 
    \ge D_j-d(c_i,z) 
    > D_j-\tau_i (\alpha) 
    \ge \alpha_j,
\end{align*}
where the last inequality follows from the definition of $\tau_i(\alpha)$.
We have shown $d(x,z) > \alpha_j$, thus $x \not\in \ball_{\alpha_j}(z)$.
This concludes the proof of the claim.

To prove the lemma, we define the set $S:=\prod_{i \in [k]}[a_i,b_i]$ 
and take $\tau_i := \inf \{\tau_i(\alpha) \mid \alpha \in S\} = \min \{ \tau_i(\alpha) \mid \alpha \in S \} > 0$, where the equality follows from the extreme value theorem, since $S$ is compact and $\tau_i(\alpha)$ is a continuous function over $S$, for every $i \in [k]$.
\end{prf}

\begin{lemma}
\label{lm from balls to sum}
Consider the ESBM.
Let $\alpha \in \R^k$ and let $s_i \in \R^m$ for every $i \in [k]$.
Assume that we have
\begin{align}
    \label{eq from balls to sum}
    \ball_{\alpha_j}(s_i) \cap \ball_{r_j}(c_j) =
\begin{cases}
\ball_{r_i}(c_i) & \text{if } j=i\\
\emptyset & \text{otherwise}
\end{cases}
&& \forall i,j \in [k].
\end{align}
Then, we have
\begin{align*}
    & \alpha_i \ge d(s_i,x^{(i)}_\ell) && \forall i \in [k], \ \forall \ell \in [n_i] \\
    & \alpha_j < d(s_i,x^{(j)}_\ell) && \forall i,j \in [k], i \neq j, \ \forall \ell \in [n_j] \\
    & C^{\alpha}(s_i)
    =n_i \alpha_i- \sum_{\ell \in [n_i]}d(s_i,x^{(i)}_\ell)
    && \forall i \in [k].
\end{align*}
\end{lemma}

\begin{prf}
From \eqref{eq from balls to sum} with $j=i$ we obtain that for every $i \in [k]$ we have $\ball_{\alpha_i}(s_i) \cap \ball_{r_i}(c_i) = \ball_{r_i}(c_i)$ thus $\ball_{r_i}(c_i) \subseteq \ball_{\alpha_i}(s_i)$.
Since $x^{(i)}_\ell \in \ball_{r_i}(c_i)$ for every $i \in [k], \ell \in [n_i]$, we obtain
\begin{align*}
\alpha_i \ge d(s_i,x^{(i)}_\ell) && \forall i \in [k], \ \forall \ell \in [n_i].
\end{align*}
From \eqref{eq from balls to sum} with $i\neq j$, we obtain that for every $i,j \in [k]$ with $i \neq j$, we have $\ball_{\alpha_j}(s_i) \cap \ball_{r_j}(c_j) = \emptyset$.
Since $x^{(j)}_\ell \in \ball_{r_j}(c_j)$ for every $j \in [k], \ell \in [n_j]$, we obtain
\begin{align*}
\alpha_j < d(s_i,x^{(j)}_\ell) && \forall i,j \in [k], i \neq j, \ \forall \ell \in [n_j].
\end{align*}
We obtain that for every $i \in [k]$,
\begin{align*}
    C^{\alpha}(s_i)
    & =\sum_{j \in [k]}\sum_{\ell \in [n_{j}]}(\alpha_{j}-d(s_i,x^{(j)}_\ell))_+ \\
    & =\sum_{\ell \in [n_{i}]}(\alpha_i-d(s_i,x^{(i)}_\ell))_+ + \sum_{j \in [k], \ j \neq i} \sum_{\ell \in [n_{j}]}(\alpha_j-d(s_i,x^{(j)}_\ell))_+ \\
    & =\sum_{\ell \in [n_{i}]}(\alpha_i-d(s_i,x^{(i)}_\ell)) \\
    & =n_i \alpha_i - \sum_{\ell \in [n_i]} d(s_i, x^{(i)}_\ell).
\end{align*}
\end{prf}

\begin{lemma}
\label{lm concentration}
Consider the ESBM.
For every $i \in [k]$, assume that the probability space $(\mu_i,\ball_{r_i}(c_i))$ satisfies \ref{ass open}.
For every $i \in [k]$, assume $r_i < D_i$, let $\alpha_i \in (r_i,D_i)$, let $\tau_i > 0$, and assume that $c_i$ is the unique point that achieves $\max \{G^{\alpha}(z) \mid z \in \ball_{r_i}(c_i)\}$.
Then there exists $\xi>0$ such that with high probability, for every $\alpha' \in \R^k$ with $\norm{\alpha'-\alpha}_\infty \le \xi$ and for every $i \in [k]$, $\argmax \{C^{\alpha'}(z) \mid z \in \{x^{(i)}_\ell\}_{\ell \in [n_i]}\} \subseteq \intr \ball_{\tau_i}(c_i)$.
\end{lemma}

\begin{prf}
Since $G^\alpha(z)$ is continuous in $z$
according to \cref{obs G property}, 
and $\ball_{r_i}(c_i) \setminus \intr \ball_{\tau_i}(c_i)$ is compact,
we know that for every $i \in [k]$, $\max\{G^\alpha(z) \mid z \in \ball_{r_i}(c_i) \setminus \intr \ball_{\tau_i}(c_i) \}$ is achieved.
Since, by assumption, for every $i \in [k]$, $c_i$ is the unique point that achieves $\max \{G^{\alpha}(z) \mid z \in \ball_{r_i}(c_i)\}$, we obtain that $G^\alpha(c_i)-\max\{G^\alpha(z) \mid z \in \ball_{r_i}(c_i) \setminus \intr \ball_{\tau_i}(c_i) \}>0, \forall i \in [k]$. 
Let 
$$
L:= \min_{i \in [k]} \bra{G^\alpha(c_i)-\max\{G^\alpha(z) \mid z \in \ball_{r_i}(c_i) \setminus \intr \ball_{\tau_i}(c_i) \}} > 0.
$$

Since for every $i \in [k]$, $G^\alpha(z)$ is continuous in $z=c_i$, 
we know that for every $i \in [k]$, there exists $0<\tau_i'<\tau_i$ such that for every $z \in \ball_{\tau_i'}(c_i)$, we have $G^\alpha(z)>G^\alpha(c_i)- L/2$.
Hence, for every $z \in \ball_{\tau_i'}(c_i)$, we have
\begin{align}
\label{eq thisoneoverhere}
\begin{split}
   & G^\alpha(z) - \max\{G^\alpha(z) \mid z \in \ball_{r_i}(c_i) \setminus \intr \ball_{\tau_i}(c_i) \} \\
   & \qquad > G^\alpha(c_i)-\frac{L}{2}-\max\{G^\alpha(z) \mid z \in \ball_{r_i}(c_i) \setminus \intr \ball_{\tau_i}(c_i) \} \ge L-\frac{L}{2}=\frac{L}{2}>0.
\end{split}
\end{align}
Let $\beta:=\max_{i \in [k]} \beta_i$ and let $\xi:=L/(8k\beta)$. 
Notice that for every $\alpha' \in \R^k$ with $\norm{\alpha'-\alpha}_\infty \le \xi$ and for every $z \in \R^m$, we have 
\begin{align}
\label{eq perturbation}
\begin{split}
    \abs{\frac{1}{n}C^{\alpha'}(z)-\frac{1}{n}C^\alpha(z)}& =\frac{1}{n}\abs{\sum_{i \in [k]}\sum_{\ell \in [n_i]}\pare{(\alpha'_i-d(z,x^{(i)}_\ell))_+-(\alpha_i-d(z,x^{(i)}_\ell))_+}} \\
    & \le \frac{1}{n}\sum_{i \in [k]}\sum_{\ell \in [n_i]}\abs{(\alpha'_i-d(z,x^{(i)}_\ell))_+-(\alpha_i-d(z,x^{(i)}_\ell))_+} \\
    & \le \frac{1}{n}\sum_{i \in [k]}\sum_{\ell \in [n_i]}\abs{\alpha'_i-\alpha_i} \le \frac{1}{n}\sum_{i \in [k]}\sum_{\ell \in [n_i]}\frac{L}{8k\beta} \le \frac{L}{8}.
\end{split}
\end{align}
For every $i \in [k]$, denote by $A_i$ the event that for every $\alpha' \in \R^k$ with $\norm{\alpha'-\alpha}_\infty \le \xi$, we have $\argmax \{C^{\alpha'}(z) \mid z \in \{x^{(i)}_\ell\}_{\ell \in [n_i]}\} \subseteq \intr \ball_{\tau_i}(c_i)$.
For every $i \in [k]$, denote by $T_i$ the event that there is some $w \in [n_i]$ such that $x_w^{(i)} \in \ball_{\tau'_i}(c_i)$.
For every $i \in [k]$ and $w \in [n_i]$, denote by $M_{iw}$ the event that at least one of the following event happens: 
\begin{itemize}
    \item $x^{(i)}_w \in \ball_{\tau_i'}(c_i)$ and  $\frac{1}{n}C^\alpha(x^{(i)}_w)-G^\alpha(x^{(i)}_w) \ge - L/8$;
    \item $x^{(i)}_w \in \intr \ball_{\tau_i}(c_i) \setminus \ball_{\tau_i'}(c_i)$;
    \item $x^{(i)}_w \in \ball_{r_i}(c_i) \setminus \intr \ball_{\tau_i}(c_i)$ and $\frac{1}{n}C^\alpha(x^{(i)}_w)-G^\alpha(x^{(i)}_w) \le L/8$.
\end{itemize}

Note that for every $i \in [k]$, if $T_i$ is true and $M_{iw}$ is true for every $w \in [n_i]$, then $A_i$ is true.
This is because $\ball_{\tau_i'}(c_i) \cap \{x^{(i)}_w\}_{w \in [n_i]}$ is nonempty and, for every $z \in \ball_{\tau_i'}(c_i) \cap \{x^{(i)}_w\}_{w \in [n_i]}$ and for every $z' \in (\ball_{r_i}(c_i) \setminus \intr \ball_{\tau_i}(c_i) )\cap \{x^{(i)}_w\}_{w \in [n_i]}$, we have $C^{\alpha'}(z) > C^{\alpha'}(z')$ for every $\alpha'$ with $\norm{\alpha'-\alpha}_\infty \le \xi$.
To see the last inequality, we use \eqref{eq perturbation}, the definition of the events $M_{iw}$, and \eqref{eq thisoneoverhere} to obtain 
\begin{align*}
    \frac{1}{n}C^{\alpha'}(z)-\frac{1}{n}C^{\alpha'}(z') & = \pare{\frac{1}{n}C^{\alpha'}(z)-\frac{1}{n}C^{\alpha}(z)}+ \pare{\frac{1}{n}C^\alpha(z)-G^\alpha(z)}
    -\pare{\frac{1}{n}C^\alpha(z')-G^\alpha(z')}\\
    & \hspace{0.5cm} +\pare{\frac{1}{n}C^{\alpha}(z')-\frac{1}{n}C^{\alpha'}(z')}
    +(G^\alpha(z)-G^\alpha(z')) \\
    & > -\frac{L}{8} -\frac{L}{8} -\frac{L}{8} -\frac{L}{8} + \frac{L}{2}
    =0.
\end{align*}

To prove the lemma we just need to show that the event $\bigcap_{i \in [k]}A_i$ happens with high probability.
In the remainder of the proof, we denote by $\bar E$ the complement of an event $E$.
From the above discussion, we have that for every $i \in [k]$,
\begin{align}
\label{eq nowineedthisonea}
    \Pr(A_i) 
    \ge \Pr(T_i \cap \bigcap_{w \in [n_i]} M_{iw})= 1-\Pr(\bar T_i \cup \bigcup_{w \in [n_i]} \bar M_{iw}) 
    \ge 1-\Pr(\bar T_i)-\sum_{w \in [n_i]}\Pr(\bar M_{iw}).
\end{align}
Hence, in the remainder of the proof we will provide a lower bound for $\Pr(A_i)$ by providing upper bounds for $\Pr(\bar T_i)$ and $\Pr(\bar M_{iw})$.

Let $p:= \max_{i \in [k]} \Pr(x^{(i)}_1 \not \in \ball_{\tau_i'}(c_i))$. 
Since for every $i \in [k]$, the probability space $(\mu_i,\ball_{r_i}(c_i))$ satisfies \ref{ass open}, we know that $p \in [0,1)$.
So we get 
\begin{align}
\label{eq nowineedthisoneb}
    \Pr(\bar T_i)=\prod_{\ell \in [n_i]}\Pr(x^{(i)}_\ell \not \in \ball_{\tau_i'}(c_i)) \le p^{n_i} \le p^n.
\end{align}
Next, we derive an upper bound for $\Pr(\bar M_{iw})$. 
We start by observing that for every $i \in [k]$ and $w \in [n_i]$, the event $\bar M_{iw}$ is true if and only if at least one of the events $P_{iw}$ and $Q_{iw}$ happens, where the events $P_{iw}$ and $Q_{iw}$ are defined below.
 \begin{itemize}
     \item[$P_{iw}$:] 
     $x^{(i)}_w \in \ball_{\tau_i'}(c_i)$ and  $\frac{1}{n}C^\alpha(x^{(i)}_w)-G^\alpha(x^{(i)}_w) < - L/8$;
    \item[$Q_{iw}$:] 
    $x^{(i)}_w \in \ball_{r_i}(c_i) \setminus \intr \ball_{\tau_i}(c_i)$ and $\frac{1}{n}C^\alpha(x^{(i)}_w)-G^\alpha(x^{(i)}_w) > L/8$.
 \end{itemize}

Next, we upper bound the probability of the event $P_{iw}$.
We notice that 
\begin{align}
\label{eq prtobeusedlaterforP}
\begin{split}
     \Pr (P_{iw})
     & = \int_{\ball_{\tau_i'}(c_i)}\Pr \pare{ \frac{1}{n}C^\alpha(z)-G^\alpha(z) < - \frac{L}{8} \mid x^{(i)}_w = z } d\mu_i(z). \\
     & \le \sup \bra{ \Pr \pare{ \frac{1}{n}C^\alpha(z)-G^\alpha(z) < - \frac{L}{8} \mid  x^{(i)}_w = z} \mid z \in \ball_{\tau_i'}(c_i)}.
\end{split}
\end{align}
For every $z \in \R^m$ and for every $i \in [k]$, $w \in [n_i]$,
we define the random variable $X_{iw}(z):=(\alpha_i-d(z,x^{(i)}_w))_+$. 
We know that for every $z$, $X_{iw}(z)$  
are independent random variables since $x^{(i)}_w$ are independent.
We then obtain 
\begin{align*}
    C^\alpha(z)=\sum_{i \in [k]}\sum_{w \in [n_i]}(\alpha_i-d(z,x^{(i)}_w))_+
    =\sum_{i \in [k]}\sum_{w \in [n_i]}X_{iw}(z).
\end{align*}
Note that, if we fix $z = x^{(i)}_w$, we can then rewrite $C^\alpha(z)$ in the form
\begin{align}
\label{eq Calphaz}
    C^\alpha(z)
    =\sum_{j \in [k] \setminus \{i\}}\sum_{\ell \in [n_u]}X_{j\ell}(z)+\sum_{\ell \in [n_i] \setminus \{w\}}X_{i\ell}(z)+\alpha_i.
\end{align}
Also, we have  
\begin{align}
\label{eq mean for later}
    \E\pare{ \sum_{j \in [k] \setminus \{i\}}\sum_{\ell \in [n_u]}X_{j\ell}(z)+\sum_{\ell \in [n_i] \setminus \{w\}}X_{i\ell}(z) }
    =\E (C^\alpha(z)-X_{iw}(z)) 
    = nG^\alpha(z) - I(z),
\end{align}
where $I(z):=\int_{\ball_{\alpha_i}(z) \cap \ball_{r_i}(c_i)} (\alpha_i-d(z,x)) d\mu_i(x)$ and the last equality follows from the definition of $G^\alpha(z)$ and using the same argument in the proof of \cref{obs G}.
We then define $M:=\max_{i \in [k]} \alpha_i$ and observe that $X_{j\ell}(z) \in [0,M]$ for every $j \in [k], \ell \in [n_j]$.
Let $\beta:=\max_{i \in [k]} \beta_i$. 
We obtain
\begin{align*}
    & \Pr \pare{ \frac{1}{n}C^\alpha(z)-G^\alpha(z) < - \frac{L}{8} \mid  x^{(i)}_w =z } \\
    = \ & \Pr \pare{ \frac{1}{n} \pare{ \sum_{j \in [k] \setminus \{i\}}\sum_{\ell \in [n_u]}X_{j\ell}(z)+\sum_{\ell \in [n_i] \setminus \{w\}}X_{i\ell}(z) } - \pare{ G^\alpha(z)-\frac{1}{n}I(z) } < -\frac{L}{8}+\frac{1}{n}I(z)-\frac{\alpha_i}{n} \mid x^{(i)}_w = z} \\
    = \ & \Pr \pare{ \frac{1}{n} \pare{ \sum_{j \in [k] \setminus \{i\}}\sum_{\ell \in [n_u]}X_{j\ell}(z)+\sum_{\ell \in [n_i] \setminus \{w\}}X_{i\ell}(z) } - \pare{ G^\alpha(z)-\frac{1}{n}I(z) } < -\frac{L}{8}+\frac{1}{n}I(z)-\frac{\alpha_i}{n} } \\
    = \ & \Pr \pare{ \pare{ \sum_{j \in [k] \setminus \{i\}}\sum_{\ell \in [n_u]}X_{j\ell}(z)+\sum_{\ell \in [n_i] \setminus \{w\}}X_{i\ell}(z) } - \pare{ nG^\alpha(z)-I(z)} < - \pare{ \frac{nL}{8}-I(z)+\alpha_i} }
    \\
    \le \ & \exp \pare{ -\frac{(nL-8I(z)+8\alpha_i)^2}{32M^2(\sum_{i \in [k]}n_i-1)} }
    \le \exp \pare{ -\frac{(nL)^2}{32M^2(\sum_{i \in [k]}n_i-1)}}
    \le \exp \pare{ -\frac{nL^2}{32k\beta M^2}}.
\end{align*}
The first equality follows from \eqref{eq Calphaz} and by adding on both sides $\frac{1}{n}I(z)$.
In the second equality, we use the fact that for every $z$, $X_{j\ell}(z)$ are independent. 
In the first inequality, we use Hoeffding's inequality and \eqref{eq mean for later}. 
The second inequality holds because $\alpha_i-I(z) \ge 0$ and the last inequality follows by the definition of $\beta$.
Using \eqref{eq prtobeusedlaterforP}, we obtain the following upper bound on the probability of the event $P_{iw}$.
\begin{align*}
    \Pr (P_{iw}) 
    \le \sup\bra{ \Pr \pare{ \frac{1}{n}C^\alpha(z)-G^\alpha(z) < - \frac{L}{8} \mid  x^{(i)}_w = z }
     \mid z \in \ball_{\tau_i'}(c_i) } 
     \le \exp \pare{ -\frac{nL^2}{32k\beta M^2} }.
\end{align*}

In a similar fashion, we now obtain an upper bound on the probability of the event $Q_{iw}$.
We have
\begin{align*}
    & \Pr \pare{ \frac{1}{n}C^\alpha(z)-G^\alpha(z) >  \frac{L}{8} \mid  x^{(i)}_w = z } \\
    = \ & \Pr \pare{ \frac{1}{n} \pare{ \sum_{j \in [k] \setminus \{i\}}\sum_{\ell \in [n_u]}X_{j\ell}(z)+\sum_{\ell \in [n_i] \setminus \{w\}}X_{i\ell}(z) } - \pare{ G^\alpha(z)-\frac{1}{n}I(z) } > \frac{L}{8}+\frac{1}{n}I(z)-\frac{\alpha_i}{n} \mid  x^{(i)}_w = z } \\
    = \ & \Pr \pare{ \frac{1}{n} \pare{ \sum_{j \in [k] \setminus \{i\}}\sum_{\ell \in [n_u]}X_{j\ell}(z)+\sum_{\ell \in [n_i] \setminus \{w\}}X_{i\ell}(z) } - \pare{ G^\alpha(z)-\frac{1}{n}I(z) } > \frac{L}{8}+\frac{1}{n}I(z)-\frac{\alpha_i}{n} } \\
    = \ & \Pr \pare{ \pare{ \sum_{j \in [k] \setminus \{i\}}\sum_{\ell \in [n_u]}X_{j\ell}(z)+\sum_{\ell \in [n_i] \setminus \{w\}}X_{i\ell}(z) } -(nG^\alpha(z)-I(z)) >  \frac{nL}{8}+I(z)-\alpha_i }
    \\
    \le \ & \exp \pare{ -\frac{(nL+8I(z)-8\alpha_i)^2}{32M^2(\sum_{i \in [k]}n_i-1)} }
    \le \exp \pare{ -\frac{(nL-8\alpha_i)^2}{32M^2(\sum_{i \in [k]}n_i-1)}}
    \le \exp \pare{ -\frac{(nL-8\alpha_i)^2}{32k\beta n M^2}},
\end{align*}
where the second inequality holds because $I(z) \ge 0$.
We obtain the following upper bound on the probability of the event $Q_{iw}$.
\begin{align*}
    \Pr (Q_{iw})
    & \le \sup \bra{ \Pr \pare{ \frac{1}{n}C^\alpha(z)-G^\alpha(z) >  \frac{L}{8} \mid  x^{(i)}_w = z} \mid z \in \ball_{r_i}(c_i) \setminus \intr \ball_{\tau_i}(c_i)} \\
    & \le \exp \pare{ -\frac{(nL-8\alpha_i)^2}{32k\beta n M^2}} \le \exp \pare{ -\frac{nL^2}{64k\beta M^2}},
\end{align*}
where the last inequality holds when $n>16\alpha_{\max}/((2-\sqrt{2})L)$, and $\alpha_{\max} :=\max_{i \in [k]} \alpha_i$.

Using the union bound, when $n>16\alpha_{\max}/((2-\sqrt{2})L)$, we have 
\begin{align*}
    \Pr(\bar M_{iw}) 
    \le \Pr(P_{iw}) + \Pr(Q_{iw})
    \le \exp \pare{ -\frac{nL^2}{32k\beta M^2}}+\exp \pare{ -\frac{nL^2}{64k\beta M^2}} \le 2\exp \pare{ -\frac{nL^2}{64k\beta M^2}}.
\end{align*}
Using \eqref{eq nowineedthisonea} and \eqref{eq nowineedthisoneb}, when $n>16\alpha_{\max}/((2-\sqrt{2})L)$, we have 
\begin{align*}
    \Pr(A_i) 
    \ge 1-p^n-2n_i\exp \pare{ -\frac{nL^2}{64k\beta M^2} } 
    \ge 1-p^n-2\beta n\exp \pare{ -\frac{nL^2}{64k\beta M^2} }.
\end{align*}
The latter quantity goes to $1$ as $n$ goes to infinity because $p<1$ and $p,k,\beta,L,M$ are all parameters that do not depend on $n$.
Hence each event $A_i$, for $i \in [k]$, happens with high probability.
Therefore, also $\bigcap_{i \in [k]}A_i$ happens with high probability.
So with high probability, for every $\alpha'$ with $\norm{\alpha'-\alpha}_\infty \le \xi$ and for every $i \in [k]$, we have that $\argmax \{C^{\alpha'}(z) \mid z \in \{x^{(i)}_\ell\}_{\ell \in [n_i]}\} \subseteq \intr \ball_{\tau_i}(c_i)$.
\end{prf}

We are now ready to prove \cref{th prob}.
In the proof we use \cref{th deterministic} and \cref{lm one median,lm improve center in one ball,lm opt,lm geometry,lm from balls to sum,lm concentration}.

\subsection{Proof of \cref{th prob}}
\label{sec th prob proof}

In this proof, for every $i \in [k]$, we denote by $x_*^{(i)}$ a median of $\{x^{(i)}_\ell\}_{\ell \in [n_i]}$.
Let $(\bar y, \bar z)$ be the feasible solution to \eqref{pr IP} that assigns each point $x^{(i)}_\ell$ to the ball $\ball_{r_i}(c_i)$ from which it is drawn. 
In particular, in this solution we have $y_p=1$ if and only if $p \in \{x^{(i)}_* \mid i \in [k]\}$. 
Furthermore, we have $z_{pq}=1$ if and only if $y_p=1$ and $p,q$ are drawn from the same ball.

To prove the theorem, we show that $(\bar y,\bar z)$ is the unique optimal solution to \eqref{pr LP} with high probability.
Clearly, $(\bar y,\bar z)$ is a feasible solution to \eqref{pr IP}.
We know from \cref{th deterministic} that $(\bar y,\bar z)$ is the unique optimal solution to \eqref{pr LP} if there exists $\Tilde{\alpha} \in \R^P$ such that
\begin{align}
    & C^{\Tilde\alpha}(a_1)=\dots=C^{\Tilde\alpha}(a_k) \label{eq Th1 a inproof}\\
    & C^{\Tilde\alpha}(q) < C^{\Tilde\alpha}(a_1) && \forall q \in P \setminus \{a_i\}_{i \in [k]} \label{eq Th1 b inproof} \\
    & \Tilde{\alpha}_q \ge d(a_i,q) && \forall i \in [k], \ \forall q \in A_i\label{eq Th1 c inproof} \\
    & \Tilde{\alpha}_q < d(a_i,q) && \forall i \in [k], \ \forall q \in P \setminus A_i. \label{eq Th1 d inproof}
\end{align}

Let $\gamma, \alpha$ as in the statement of \cref{th prob}. 
For all $i \in [k]$, we obtain $r_i < D_i$, and using the definition of $\alpha_i$, we obtain $\alpha_i \in (r_i,D_i)$.
Hence there exists $\xi_1 >0$ such that $[\alpha_i-\xi_1,\alpha_i+\xi_1] \subset(r_i,D_i)$ for all $i \in [k]$. 
From \cref{lm geometry} with $a_i=\alpha_i-\xi_1$ and $b_i=\alpha_i+\xi_1$, we obtain that 
there exist $\tau_i>0$, $\forall i \in [k]$, such that $\forall i,j \in [k]$, $\forall z \in \intr \ball_{\tau_i}(c_i)$, and $\forall \alpha_j' \in [\alpha_j-\xi_1,\alpha_j+\xi_1]$, we have
\begin{align}
\label{eq lm geometry in T2p}
    \ball_{\alpha_j'}(z) \cap \ball_{r_j}(c_j) =
\begin{cases}
\ball_{r_i}(c_i) & \text{if } j=i\\
\emptyset & \text{otherwise.}
\end{cases}
\end{align}
Since the assumptions of \cref{lm concentration} are satisfied, there exists $\xi_2>0$ such that with high probability, for every $\alpha' \in \R^k$ with $\norm{\alpha'-\alpha}_\infty \le \xi_2$ and for every $i \in [k]$, $\argmax \{C^{\alpha'}(z) \mid z \in \{x^{(i)}_\ell\}_{\ell \in [n_i]}\} \subseteq \intr \ball_{\tau_i}(c_i)$. 
Let $\xi:= \min \{\xi_1,\xi_2\}$. 
For every $i \in [k]$, let $\OPT_i:=\sum_{\ell \in [n_i]}d(x_*^{(i)},x^{(i)}_\ell).$ 
We then know from \cref{lm opt} that with high probability, for every $i \in [k]$, we have $\abs{\OPT_i/n_i-E_i}<\xi$.
From \cref{lm improve center in one ball}, we know that with high probability, for every $i \in [k]$, we have $x^{(i)}_* \in \intr \ball_{\tau_i}(c_i)$.

For every $i \in [k]$, fix $\alpha'_i:=\alpha_i+\epsilon_i$, where $\epsilon_i := \OPT_i/n_i-E_i.$ 
For every $q \in P$, we set $\Tilde{\alpha}_q := \alpha'_i$, where $i$ is the unique index in $[k]$ with $q \in \ball_{r_i}(c_i)$. 
We next show that, with this choice of $\Tilde{\alpha}$, \eqref{eq Th1 a inproof}--\eqref{eq Th1 d inproof} are satisfied with high probability. 
Using the definition of $\Tilde{\alpha}$, it suffices to show that
\begin{align}
& C^{\alpha'}(x^{(1)}_*)=\dots=C^{\alpha'}(x^{(k)}_*) \label{eq Th1 a inproof v2} \\
& C^{\alpha'}(x^{(i)}_\ell) < C^{\alpha'}(x^{(i)}_*) && \forall i \in [k], \ \forall \ell \in [n_i] \text{ with } \ x^{(i)}_\ell \neq x^{(i)}_* \label{eq Th1 b inproof v2} \\
& \alpha'_i \ge d(x_*^{(i)},x^{(i)}_\ell) && \forall i \in [k], \ \forall \ell \in [n_i] \label{eq Th1 c inproof v2} \\
& \alpha'_i < d(x_*^{(j)},x^{(i)}_\ell) && \forall i,j \in [k], i \neq j, \ \forall \ell \in [n_i]. \label{eq Th1 d inproof v2}
\end{align}
Since for every $i \in [k]$, we have $\abs{\epsilon_i}<\xi$, we know that $\alpha'_i \in (r_i, D_i)$. 
Since for every $i \in [k]$, we have $x^{(i)}_* \in \intr \ball_{\tau_i}(c_i)$, we have that \eqref{eq lm geometry in T2p} holds with $z = x^{(i)}_*$.
Thus from \cref{lm from balls to sum} (with $s_i = x^{(i)}_*$) we obtain \eqref{eq Th1 c inproof v2}, \eqref{eq Th1 d inproof v2}, and 
\begin{align*}
    C^{\alpha'}(x^{(i)}_*)
    & =n_i \alpha'_i- \sum_{\ell \in [n_i]}d(x_*^{(i)},x^{(i)}_\ell) 
     =n_i (\alpha_i + \epsilon_i) - \OPT_i \\
    & =n_i \pare{ \frac{\gamma}{\beta_i}+ \frac{\OPT_i}{n_i}} -\OPT_i
    =\gamma n+ \OPT_i-\OPT_i=\gamma n
    && \forall i \in [k].
\end{align*}
which implies \eqref{eq Th1 a inproof v2}. 
For every $i \in [k]$, let $s_i \in \intr \ball_{\tau_i} (c_i) \cap \{x^{(i)}_\ell\}_{\ell \in [n_i]}$. 
Since $s_i \in \intr \ball_{\tau_i} (c_i)$, for every $i,j \in [k]$, we have that $z = s_i$ satisfies \eqref{eq lm geometry in T2p}. 
From \cref{lm from balls to sum} we obtain 
\begin{align*}
    C^{\alpha'}(s_i)
    =n_i\alpha'_i-\sum_{\ell \in [n_i]} d(s_i, x^{(i)}_\ell) \le n_i\alpha'_i -\sum_{\ell \in [n_i]} d(x^{(i)}_*, x^{(i)}_\ell)
    =C^{\alpha'}(x^{(i)}_*)
    && \forall i \in [k].
\end{align*}
Since from \cref{lm one median}, the vector $x_*^{(i)}$ is the unique median of $\{x^{(i)}_\ell\}_{\ell \in [n_i]}$ with probability one when $n_i \ge 3$, the above inequality achieves equality if and only if $s_i=x_*^{(i)}$. 
Since $\norm{\alpha'-\alpha}_\infty = \norm{\epsilon}_\infty < \xi \le \xi_2$, for every $i \in [k]$ we have $\argmax \{C^{\alpha'}(z) \mid z \in \{x^{(i)}_\ell\}_{\ell \in [n_i]}\} \subseteq \intr \ball_{\tau_i}(c_i)$. Thus, we know that $x^{(i)}_*$ is the unique point that achieves $\max \{C^{\alpha'}(z) \mid z \in \{x^{(i)}_\ell\}_{\ell \in [n_i]}\}.$ This implies \eqref{eq Th1 b inproof v2}.

Hence $(\bar y,\bar z)$ is the unique optimal optimal solution to \eqref{pr LP} with high probability.
\qed

\subsection{Proof of \cref{cor prob}}
\label{sec cor prob proof}

Let $\gamma:=\alpha'-E_1$. 
Since for every $i \in [k]$, $\beta_i=1$, $r_i=1$, and $E_1 = \cdots = E_k$, we have $\max_{i \in [k]}\beta_i(r_i-E_i)=1-E_1$ and $\min_{i \in [k]}\beta_i(D_i-E_i) = \min_{i \neq j} d(c_i,c_j)-1-E_1$. 
From $1<\alpha'<\min_{i \neq j} d(c_i,c_j)-1$ we then obtain $\max_{i \in [k]}\beta_i(r_i-E_i)<\gamma<\min_{i \in [k]}\beta_i(D_i-E_i)$. 
Notice that for every $i \in [k]$ we have
$$
\alpha_i=\alpha'=E_i+\gamma =E_i+\frac{\gamma}{\beta_i}.
$$ 
Since for $i \in [k]$, $c_i$ is the unique point that achieves $\max \{G^{\alpha}(z) \mid z \in \ball_{r_i}(c_i)\}$, \cref{th prob} implies that \eqref{pr LP} achieves exact recovery with high probability.
\qed

\section{Exact recovery in the ESBM}
\label{sec truth ESBM}

In this section, we present our recovery results for the ESBM. 
We also show that for the ESBM with some special structure, including the SBM, \eqref{pr LP} can perform even better.

For completeness, we start with the simple case $m=1$.
The next two theorems can be seen as a corollaries of \cref{th prob}.

\begin{theorem}
\label{th one dim gen}
Consider the ESBM with $m=1$.
For every $i \in [k]$, assume that the probability space $(\mu_i,\ball_{r_i}(c_i))$ satisfies \ref{ass rotation}, \ref{ass open}, \ref{ass dim}. 
Let $R:=\max_{i \in [k]} r_i$ and $\beta:=\max_{i \in [k]} \beta_i$.
If for every $i \neq j$ we have $d(c_i,c_j)>r_i+r_j+(1+2\beta)R$, then \eqref{pr LP} achieves exact recovery with high probability. 
\end{theorem}

\begin{prf}
It suffices to check that all assumptions of \cref{th prob} are satisfied.
For every $i \in [k]$, denote by $E_i := \Ex d(x,c_i)$, where $x$ is a random vector drawn according to $\mu_i$. 
Let $\gamma:=\max_{i \in [k]} \beta_i(2r_i-E_i)$. 
Using the fact that $r_i \in (0,R]$, $\beta_i \in [1,\beta]$, and $E_i \in (0,R]$,
we can bound $\gamma$ and obtain
\begin{align*}
    \max_{i \in [k]} \beta_i(r_i-E_i)
    < \gamma
    < 2\beta R 
    < \min_{j \in [k]}r_j + (1+2\beta)R -R
    < \min_{i \neq j} d(c_i,c_j)-r_i-R
    \le \min_{i \in [k]}\beta_i(D_i-E_i).
\end{align*}

Let $\alpha_i := E_i+\frac{\gamma}{\beta_i}$ for every $i \in [k]$.
It remains to show that for every $i \in [k]$, $c_i$ is the unique point that achieves $\max \{G^{\alpha}(z) \mid z \in \ball_{r_i}(c_i)\}$.
For every $i \in [k]$, from the definition of $\gamma$ we have $\alpha_i \ge 2r_i$, thus for every $z \in \ball_{r_i}(c_i)$, we have $\ball_{\alpha_i}(z) \cap \ball_{r_i}(c_i)=\ball_{r_i}(c_i)$.
For every $i \in [k]$ we have $\alpha_i \le R+\gamma < (1+2\beta)R$ which implies $\ball_{\alpha_j}(z) \cap \ball_{r_j}(c_j)=\emptyset$ for every $z \in \ball_{r_i}(c_i)$ and every $j \in [k] \setminus \{i\}$. 
From \cref{obs G}, we know that for every $i \in [k]$ and for every $z \in \ball_{r_i}(c_i)$ with $z \neq c_i$, we have 
\begin{align*}
    G^\alpha(z)= \int_{-r_i}^{r_i} (\alpha_i -d(z,x)) d \mu_i(x) = \alpha_i - \E d(z,x) < \alpha_i - \E d(c_i,x) = G^\alpha(c_i),
\end{align*}
where the inequality follows from \cref{lm ball center m = 1}.

The assumptions of \cref{th prob} are satisfied, and so \eqref{pr LP} achieves exact recovery with high probability.
\end{prf}

\begin{theorem}
\label{th one dim 1}
Consider the ESBM with $m=1$.
For every $i \in [k]$, assume that the probability space $(\mu_i,\ball_{r_i}(c_i))$ satisfies \ref{ass rotation}, \ref{ass open}, \ref{ass dim}.
For every $i \in [k]$, assume $n_i=n$, $r_i=1$, and denote by $E_i := \Ex d(x,c_i)$, where $x$ is a random vector drawn according to $\mu_i$. 
We further assume $E_1=\cdots=E_k$.
If for every $i \neq j$ we have $d(c_i,c_j)>2+2$, then \eqref{pr LP} achieves exact recovery with high probability. 
\end{theorem}

\begin{prf}
It suffices to check that all assumptions of \cref{cor prob} are satisfied.
Let $\Theta := \min_{i \neq j}d(c_i,c_j) - 2>2$, let $\epsilon \in (0,\Theta-2)$, and let $\alpha':=2 + \epsilon$. 
Note that we have $2 < \alpha' <\Theta$, which in particular implies 
\begin{align*}
1 < \alpha' < \min_{i \neq j}d(c_i,c_j) - 1.
\end{align*}

Let $\alpha_i := \alpha'$ for every $i \in [k]$.
It remains to show that for every $i \in [k]$, $c_i$ is the unique point that achieves $\max \{G^{\alpha}(z) \mid z \in \ball_{1}(c_i)\}$.
For every $i \in [k]$ and for every $z \in \ball_{1}(c_i)$, $\alpha' > 2$ implies $\ball_{\alpha_i}(z) \cap \ball_{1}(c_i)=\ball_{1}(c_i)$ and $\alpha' <\Theta$ implies $\ball_{\alpha_j}(z) \cap \ball_{1}(c_j)=\emptyset$ for every $j \in [k] \setminus \{i\}$.
From \cref{obs G}, we know that for every $i \in [k]$ and for every $z \in \ball_1(c_i)$ with $z \neq c_i$, we have 
\begin{align*}
    G^\alpha(z)
    = \int_{-1}^1 (\alpha_i -d(z,x)) d \mu_i(x) = \alpha_i - \E d(z,x) < \alpha_i - \E d(c_i,x) = G^\alpha(c_i),
\end{align*}
where the inequality follows from \cref{lm ball center m = 1}.

The assumptions of \cref{cor prob} are satisfied, and so \eqref{pr LP} achieves exact recovery with high probability.
\end{prf}

\cref{th one dim 1} implies that in the SBM with $m=1$, a sufficient condition for \eqref{pr LP} to achieve exact recovery is that the distance between any pair of points from the same ball is always smaller than the distance between any pair of points from different balls. 
We remark that under this assumption a simple threshold algorithm can also achieve exact recovery. 
It is currently unknown if in the SBM with $m=1$ a pairwise distance smaller than 4 may be sufficient to guarantee exact recovery.

Next, we present our most interesting results, which consider exact recovery for the ESBM and the SBM with $m \ge 2$.

\begin{theorem}
\label{th main2}
Consider the ESBM with $m \ge 2$. 
For every $i \in [k]$, assume that the probability space $(\mu_i,\ball_{r_i}(c_i))$ satisfies \ref{ass rotation}, \ref{ass open}, \ref{ass dim}. 
Let $\beta,r,R \in \R$ such that for every $i \in [k]$ we have $r_i \in [r,R]$ and $\beta_i \le \beta$.
Then there is a function $\epsilon(k,m)=C\sqrt{k \log m/m}$, where $C$ is a positive constant, such that, 
if for every $i \neq j$ we have $d(c_i,c_j)>(1+\beta)R+\max\{r_i,r_j\}+\epsilon(k,m)$, then $\eqref{pr LP}$ achieves exact recovery with high probability. 
\end{theorem}

\begin{theorem}
\label{th main3}
Consider the ESBM with $m \ge 2$.
For every $i \in [k]$, assume that the probability space $(\mu_i,\ball_{r_i}(c_i))$ satisfies \ref{ass rotation}, \ref{ass open}, \ref{ass dim}.
For every $i \in [k]$, assume $n_i=n$, $r_i=1$, and denote by $E_i := \Ex d(x,c_i)$, where $x$ is a random vector drawn according to $\mu_i$. 
We further assume $E_1=\cdots=E_k$. 
If for every $i \neq j$ we have $d(c_i,c_j)>2+1.29$,
then \eqref{pr LP} achieves exact recovery with high probability. 
\end{theorem}

\begin{theorem}
\label{th main4}
Consider the ESBM with $m \ge 2$.
For every $i \in [k]$, assume that the probability space $(\mu_i,\ball_{r_i}(c_i))$ satisfies \ref{ass rotation}, \ref{ass open}, \ref{ass dim}.
For every $i \in [k]$, assume $n_i=n$, $r_i=1$, and denote by $E_i := \Ex d(x,c_i)$, where $x$ is a random vector drawn according to $\mu_i$. 
We further assume $E_1=\cdots=E_k$. 
Then there is a function $\epsilon(k,m)=C\sqrt{k \log m/m}$, where $C$ is a positive constant, such that, if for every $i \neq j$ we have $d(c_i,c_j)>2+\epsilon(k,m)$, then $\eqref{pr LP}$ achieves exact recovery with high probability.  
\end{theorem}

\begin{theorem}
\label{th main5}
Consider the SBM with $m \ge 2$. 
Assume that the probability space $(\mu,\ball_1(0))$ satisfies \ref{ass rotation}, \ref{ass open}, \ref{ass dim}. 
Assume that $\mu$ has a density function $p(x)$ and assume that, for $x_1,x_2 \in \ball_1(0)$ with $\norm{x_1} < \norm{x_2}$ we have $p(x_1) > p(x_2)$.
If for every $i \neq j$ we have $d(c_i,c_j)>2$, then \eqref{pr LP} achieves exact recovery with high probability.
\end{theorem}

For the SBM, \cref{th main3} gives the best known sufficient condition for exact recovery, which does not depend on $k$ or $m$. 
Furthermore, if $k$ does not grow fast, \cref{th main4} gives a near optimal condition for exact recovery in high dimension.
Furthermore, \cref{th main5} corrects the result in \cite{awasthi2015relax} by adding assumptions on the probability measure. 
Beyond the SBM, \cref{th main2} shows that if we consider the much more general ESBM, exact recovery still happens, as long as the numbers of points drawn from each ball 
have the same order.
We already discussed in \cref{sec Model} that this assumption is necessary and cannot be dropped (see \cref{ex diff order} in \cref{app same order}).

The remainder of the section is devoted to proving \cref{th main2,th main3,th main4,th main5}.

\subsection{Analysis of $G^\alpha(z)$}
\label{sec analysis G}

According to \cref{sec Probabilistic condition}, we know that exact recovery is closely related to the function $G^\alpha(z)$. 
In this section, we present an in-depth study of the function $G^\alpha(z)$. 
These results on $G^\alpha(z)$ will be used to prove our exact recovery results in dimension $m \ge 2$, i.e., \cref{th main2,th main3,th main4,th main5}.
This is why in several results in this section we assume $m \ge 2$.

\subsubsection{The random variable $\theta$}
\label{sec theta}

In the following, for $r \ge 0$, we denote by $\mu^r$ the uniform probability measure with support $\sphere_r(0)$.
Let $v$ be a fixed unit vector in $\R^m$ and let $x$ be a random vector in $\R^m$ drawn according to $\mu^1$.
We define the random variable $\theta(x)$ to be the angle between $v$ and $x$.
Since both $v$ and $x$ are unit vectors we can write
\begin{align*}
    \theta(x):=\arccos\inner{v}{x}
    \in [0,\pi].
\end{align*}
We can then use $\Tilde{\mu}(\theta)$ to denote the probability measure of $\theta$.
In the next observation we show that the probability measure $\Tilde{\mu}(\theta)$ also arises from probability measures more general than $\mu^1$.
We recall that two random variables $A,B$ have the same probability measure if for every $\psi \in \R$, we have $\Pr(A \le \psi) = \Pr(B \le \psi)$.

\begin{observation}
\label{obs link}
Let $(\mu,\ball_r(0))$ be a probability space that satisfies \ref{ass rotation}, \ref{ass dim}.
Let $v$ be a fixed unit vector in $\R^m$, and let $x$ be a random vector in $\R^m$ drawn according to $\mu$.
We define the random variable $\theta'(x)$ to be the angle between $v$ and $x$ if $x \neq 0$, and $\theta'(x) := \pi/2$ if $x=0$.
Then $\theta'$ has the same probability measure as $\theta$.
\end{observation}

\begin{prf}
Since $v$ is a unit vector we have for every $x \neq 0$,
\begin{align*}
    \theta'(x) 
    =\arccos \inner{v}{\frac{x}{\norm{x}}}
    \in [0,\pi].
\end{align*}
If $x=0$, we have $\theta'(0)=\pi/2$.

Since $(\mu,\ball_r(0))$ satisfies \ref{ass dim}, we have 
$\Pr(x=0)=0$.
Since $(\mu,\ball_r(0))$ satisfies \ref{ass rotation}, we have that, for $x \neq 0$, $x/\norm{x}$ is a random vector drawn according to $\mu^1$. 
So for every $\psi \in [0,\pi]$, we have
\begin{align*}
    \Pr(\theta'(x) \le \psi)
    & = \Pr(\theta'(x) \le \psi, \ x \neq 0)+\Pr(\theta'(x) \le \psi, \ x = 0) \\
    & = \Pr(\theta'(x) \le \psi \mid x \neq 0) \ \Pr(x \neq 0) 
    = \Pr(\theta(x) \le \psi).
\end{align*}
We then obtain that $\theta'$ and $\theta$ have the same probability measure.
\end{prf}

When $m \ge 2$ we note that the random variable $\theta$ has a density function and we denote it by $p^{(m)}(\theta) := d\Tilde{\mu}/d\theta$.
In the remainder of this section we study the density function $p^{(m)}(\theta)$ thus we always assume $m \ge 2$.
In the following, we denote by $\Gamma(x)$ the gamma function.

\begin{observation}
\label{obs T3}
Let $m \ge 2$.
We have 
\begin{align*}
p^{(m)}(\theta)=\frac{1}{\sqrt{\pi}}\frac{\Gamma(\frac{m}{2})}{\Gamma(\frac{m-1}{2})}\sin^{m-2}\theta.
\end{align*}
\end{observation}

\begin{prf}
Let $\psi$ be a fixed angle in $[0,\pi]$.
We know that those $x \in \sphere_1(0)$ such that $\theta(x)=\psi$ form a $m-2$ dimensional sphere in $\R^m$ centered at $v \cos\psi$ with radius $\sin \psi$, which we denote by $S^{m-2}_{\sin \psi}(v \cos \psi)$. 
Formally, we define
\begin{align*}
S^{m-2}_{\sin \psi}(v \cos \psi) :=
    \{x\in \sphere_1(0)\mid\theta(x) = \psi\}.
\end{align*}
In the following we denote by $\lambda^{m-1}(\cdot)$ the $(m-1)$-dimensional volume and by $\lambda^{m-2}(\cdot)$ the $(m-2)$-dimensional volume.
Then 
\begin{align*}
    \lambda^{m-1}(\{x\in \sphere_1(0) \mid \theta(x) \le \psi\})=\int_0^\psi \lambda^{m-2}(S^{m-2}_{\sin \theta}(v \cos \theta ))d\theta
    =\lambda^{m-2}(S^{m-2}_1(0))\int_0^\psi \sin^{m-2}\theta d\theta.
\end{align*}
In particular, 
\begin{align*}
    \lambda^{m-1}(\sphere_1(0))
    & =\lambda^{m-1}(\{x\in \sphere_1(0)\mid \theta(x) \le \pi\}) 
     =\lambda^{m-2}(S^{m-2}_1(0))\int_0^\pi \sin^{m-2}\theta d\theta \\
    & =\lambda^{m-2}(S^{m-2}_1(0))\sqrt{\pi}\frac{\Gamma(\frac{m-1}{2})}{\Gamma(\frac{m}{2})}.
\end{align*}
Since $x$ is drawn uniformly from $\sphere_1(0)$, we know that 
\begin{align*}
    \Pr (\theta \le \psi)
    = \frac{\lambda^{m-1}(\{x\in \sphere_1(0)\mid \theta(x) \le \psi\})}{\lambda^{m-1}(\sphere_1(0))} = \frac{1}{\sqrt{\pi}}\frac{\Gamma(\frac{m}{2})}{\Gamma(\frac{m-1}{2})}\int_0^\psi \sin^{m-2}\theta d\theta .
\end{align*}
Thus, we obtain
\begin{align*}
    p^{(m)}(\theta)=\frac{1}{\sqrt{\pi}}\frac{\Gamma(\frac{m}{2})}{\Gamma(\frac{m-1}{2})}\sin^{m-2}\theta.
\end{align*}
\end{prf}

\begin{observation}
\label{obs ptheta}
Let $m \ge 2$.
Then there exists a threshold $s_m \in (0,1)$ such that 
\begin{align*}
p^{(m)}(\theta)-p^{(m+1)}(\theta)
\begin{cases}
\ge 0 & \text{if } 0 \le \sin \theta \le s_m, \\
< 0 & \text{if } s_m < \sin \theta \le 1.
\end{cases}
\end{align*}
\end{observation}

\begin{prf}
Using \cref{obs T3} we can write
\begin{align}
\label{eq T3 twice}
\begin{split}
    p^{(m)}(\theta)-p^{(m+1)}(\theta)
    &=\frac{1}{\sqrt{\pi}}\frac{\Gamma(\frac{m}{2})}{\Gamma(\frac{m-1}{2})}\sin^{m-2}\theta-\frac{1}{\sqrt{\pi}}\frac{\Gamma(\frac{m+1}{2})}{\Gamma(\frac{m}{2})}\sin^{m-1}\theta \\
    & = \pare{\frac{\Gamma(\frac{m}{2})}{\Gamma(\frac{m-1}{2})}-\frac{\Gamma(\frac{m+1}{2})}{\Gamma(\frac{m}{2})} \sin\theta}\frac{1}{\sqrt{\pi}}\sin^{m-2}\theta.
\end{split}
\end{align}
We set 
\begin{align*}
s_m := \frac{\Gamma(\frac{m}{2})^2}{\Gamma(\frac{m-1}{2})\Gamma(\frac{m+1}{2})}.
\end{align*}
Since $\Gamma(x)$ is a positive strictly logarithmically convex function for $x \in (0,\infty)$, we have $s_m \in (0,1)$.
We note that if $\sin \theta = s_m$, then $p^{(m)}(\theta)-p^{(m+1)}(\theta) = 0$.
Since the gamma function is positive, when $0 \le \sin \theta<s_m$, from \eqref{eq T3 twice} we obtain
\begin{align*}
    p^{(m)}(\theta)-p^{(m+1)}(\theta)
    \ge 0. 
\end{align*}
On the other hand, when $s_m < \sin \theta \le 1$, from \eqref{eq T3 twice} we obtain
\begin{align*}
    p^{(m)}(\theta)-p^{(m+1)}(\theta)
    < 0.
\end{align*}
\end{prf}

\begin{observation}
\label{obs monotone}
Let $m \ge 2$ and let $\bar \theta \le \frac{\pi}{2}$. 
Let $g(\theta)$ be a nonnegative decreasing function on $(0,\bar \theta)$.
Then we have $\int_0^{\bar \theta}g(\theta)(p^{(m)}(\theta)-p^{(m+1)}(\theta))d\theta \ge 0$.
\end{observation}

\begin{prf}
Let $s_m \in (0,1)$ be the threshold for $p^{(m)}-p^{(m+1)}$ from \cref{obs ptheta}.
Let $\psi \in (0,\pi/2)$ such that $\sin \psi=s_m$.
We then have
\begin{align*}
p^{(m)}(\theta)-p^{(m+1)}(\theta)
\begin{cases}
\ge 0 & \text{if } 0 \le \theta \le \psi, \\
< 0 & \text{if } \psi < \theta \le \pi/2.
\end{cases}
\end{align*}

We consider separately two cases.
In the first case we assume $\bar \theta \le \psi$.
Since $p^{(m)}(\theta)-p^{(m+1)}(\theta) \ge 0$ when $\theta \in (0,\bar \theta)$ and since $g(\theta)$ is a nonnegative decreasing function on $(0,\bar \theta)$, 
we obtain
\begin{align*}
    \int_0^{\bar \theta}g(\theta)(p^{(m)}(\theta)-p^{(m+1)}(\theta))d\theta \ge g(\bar \theta) \int_0^{\bar \theta} (p^{(m)}(\theta)-p^{(m+1)}(\theta))d\theta \ge 0.
\end{align*}

In the second case we assume $\psi<\bar \theta \le \frac{\pi}{2}$.
We have 
\begin{align*}
&    \int_0^{\bar \theta}g(\theta)(p^{(m)}(\theta)- p^{(m+1)}(\theta))d\theta \\
    = \ & \int_0^{\psi}g(\theta)(p^{(m)}(\theta)-p^{(m+1)}(\theta))d\theta+\int_{\psi}^{\bar \theta}g(\theta)(p^{(m)}(\theta)-p^{(m+1)}(\theta))d\theta \\
    \ge \ & g(\psi) \int_0^{\psi}(p^{(m)}(\theta)-p^{(m+1)}(\theta))d\theta+g(\psi) \int_{\psi}^{\bar \theta}(p^{(m)}(\theta)-p^{(m+1)}(\theta))d\theta \\
    = \ & g(\psi) \int_0^{\bar \theta}(p^{(m)}(\theta)-p^{(m+1)}(\theta))d\theta  
    = -g(\psi) \int_{\bar \theta}^{\frac{\pi}{2}}(p^{(m)}(\theta)-p^{(m+1)}(\theta))d\theta \ge 0.
\end{align*}
The last equality follows from the fact that $\int_0^{\frac{\pi}{2}}p^{(m)}(\theta)d\theta = \int_0^{\frac{\pi}{2}}p^{(m+1)}(\theta)d\theta = \frac{1}{2}.$ 
The last inequality uses the fact that $g(\psi) \ge 0$ and $\int_{\bar \theta}^{\frac{\pi}{2}}(p^{(m)}(\theta)-p^{(m+1)}(\theta))d\theta \le 0.$
\end{prf}


\begin{lemma}
\label{lem T6}
Let $m \ge 2$ and 
let $[\phi_1,\phi_2] \subseteq [0,\pi]$ such that $\frac{\pi}{2} \not\in [\phi_1,\phi_2]$. 
Denote by $\phi$ an angle $\theta \in \{\phi_1,\phi_2\}$ for which $\sin \theta$ is the largest.
Then $\Pr (\theta \in [\phi_1,\phi_2]) <\frac{\sqrt{\pi}}{2}\sqrt{\frac{m}{2}}\sin^{m-2}\phi$.
\end{lemma}

\begin{prf}
Gautschi's inequality implies that for every $x>0$ and every $s \in (0,1)$, the following inequality holds
\begin{align*}
      \frac{\Gamma(x+1)}{\Gamma(x+s)} < (x+1)^{1-s}.
\end{align*}
We apply Gautschi's inequality with $x = m/2-1$, for $m \ge 3$, and $s = 1/2$.
Hence, when $m \ge 3$ we have
\begin{align*}
    \frac{\Gamma(\frac{m}{2})}{\Gamma(\frac{m-1}{2})}
    < \sqrt{\frac{m}{2}}.
\end{align*}
Notice that the above inequality also holds for $m=2$ since 
\begin{align*}
    \frac{\Gamma(1)}{\Gamma(\frac{1}{2})}
    =\frac{1}{\sqrt{\pi}} < 1 = \sqrt{\frac{2}{2}}.
\end{align*}
Using \cref{obs T3} we obtain
\begin{align*}
 \Pr (\theta \in [\phi_1,\phi_2])
 & =\int_{\phi_1}^{\phi_2}p^{(m)}(\theta)d\theta 
 = \int_{\phi_1}^{\phi_2} \frac{1}{\sqrt{\pi}}\frac{\Gamma(\frac{m}{2})}{\Gamma(\frac{m-1}{2})}\sin^{m-2}\theta d\theta \\
 & \le \frac{1}{\sqrt{\pi}}\sqrt{\frac{m}{2}} (\phi_2-\phi_1) \sin^{m-2}\phi  <\frac{\sqrt{\pi}}{2}\sqrt{\frac{m}{2}}\sin^{m-2}\phi,
\end{align*}
where the last inequality holds because $\phi_2-\phi_1<\frac{\pi}{2}$.
\end{prf}

\subsubsection{Three functions related to $G^\alpha(z)$}
\label{sec three functions}

According to \cref{obs G}, we have 
\begin{align*}
    G^\alpha(z) = \sum_{i \in [k]} \beta_i\int_{\ball_{\alpha_i}(z) \cap \ball_{r_i}(c_i)} (\alpha_i-d(z,x)) d\mu_i(x).
\end{align*}
Then, for every vector $z \in \R^m$, the function $G^\alpha(z)$ can be seen as the sum of the contributions that $z$ gets from each singular ball $\ball_{r_i}(c_i)$.
Motivated by this observation, in this section we will analyze $G^\alpha(z)$ by defining three new functions.
The first function can be seen as $G^\alpha(c_1) - G^\alpha(z)$ for $z \in \ball_{r_1}(c_1)$, in the case $k=1$, $c_1=0$, $\alpha_1 > r_1$, and $\beta_1 = 1$.
\begin{definition}
\label{def H}
Let $(\mu,\ball_r(0))$ be a probability space that satisfies \ref{ass rotation} and let $\alpha > r$.
We define the function $H^{(\alpha,\mu,m)}(z) : \ball_r(0) \to \R$ as
\begin{align*}
    H^{(\alpha,\mu,m)}(z):=
    \int_{\ball_{r}(0)}(\alpha-\norm{x})d\mu(x)
    -\int_{\ball_{\alpha}(z)\cap \ball_{r}(0)} (\alpha-d(z,x))d\mu(x).
\end{align*}
\end{definition}

The second function is the special case of $H^{(\alpha,\mu,m)}$ where $\mu = \mu^r$.
\begin{definition}
\label{def T}
Let $r,\alpha \in \R_+$ with $\alpha > r$.
We define the function $T^{(\alpha,m)}(z) : \ball_{r}(0) \to \R$ as
\begin{align*}
    T^{(\alpha,m)}(z) := 
    \int_{\ball_{r}(0)}(\alpha-\norm{x})d\mu^{r}(x) -\int_{\ball_{\alpha}(z)\cap \ball_{r}(0)} (\alpha-d(z,x))d\mu^{r}(x).
\end{align*}
\end{definition}

The third function can be seen as the part of $G^\alpha(z)$ for $z \notin \ball_{r_1}(c_1)$, coming from the ball $1$, in the case $c_1=0$, $\alpha_1 \ge r_1$, $\beta_1 = 1$.
\begin{definition}
\label{def R}
Let $(\mu,\ball_r(0))$ be a probability space that satisfies \ref{ass rotation} and let $\alpha > r$. 
We define the function $R^{(\alpha,\mu,m)}(z) : \R^m \setminus \ball_r(0) \to \R$ as
\begin{align*}
    R^{(\alpha,\mu,m)}(z):=
    \int_{\ball_{\alpha}(z)\cap \ball_{r}(0)} (\alpha-d(z,x))d\mu(x).
\end{align*}
\end{definition}

Since the probability measures considered in \cref{def H,def T,def R} are invariant under rotations centered in the origin, we obtain that $H^{(\alpha,\mu,m)}(z)$, $T^{(\alpha,m)}(z)$, and $R^{(\alpha,\mu,m)}(z)$ are also invariant under rotations centered in the origin.
Therefore, in some parts of this section, we fix a unit vector $v$, and we study the three above functions evaluated in points of the form $z=tv$, where $t = \norm{z} \ge 0$.

The rest of the section is devoted to deriving bounds for $H^{(\alpha,\mu,m)}$, $T^{(\alpha,m)}$, and $R^{(\alpha,\mu,m)}$.
We start with an observation that will be used several times in the analysis of $H^{(\alpha,\mu,m)}$ and $T^{(\alpha,m)}$.

\begin{observation}
\label{obs H form}
Let $r,\alpha \in \R_+$ with $\alpha > r$ and let $s \in [0,r]$.
Let $v$ be a unit vector in $\R^m$ and let $t \in [0,r]$. 
Then $H^{(\alpha,\mu^s,m)}(tv)$ can be written in the form
\begin{align*}
    H^{(\alpha,\mu^s,m)}(tv)=
    \begin{cases}
    t & \text{if } s = 0 \\
    \alpha-s-\int_0^\pi\pare{\alpha-\sqrt{s^2+t^2-2st\cos\theta}}d\Tilde{\mu}(\theta)=\E d(tv,x)-s & \text{if } 0 < s \le \alpha-t \\
    \alpha -s-\int_0^{\bar \theta}\pare{\alpha-\sqrt{s^2+t^2-2st\cos\theta}}d\Tilde{\mu}(\theta) & \text{if } s \ge \alpha-t.
    \end{cases}
\end{align*}
In the second case $x$ is a random vector drawn according to $\mu^s$.
In the third case 
\begin{align*}
    \bar \theta:=\arccos\frac{s^2+t^2-\alpha^2}{2st} \le \pi.
\end{align*}
\end{observation}

\begin{prf}
Note that we can write $H^{(\alpha,\mu^s,m)}(tv)$ in the form
\begin{align}
\label{eq H form inproof}
    H^{(\alpha,\mu^s,m)}(tv)
    =\alpha-s-\int_{\ball_\alpha(tv) \cap \sphere_s(0)}\pare{\alpha-d(tv,x)} d\mu^s(x).
\end{align}

If $s=0$ we have $\sphere_0(0)\subseteq \ball_\alpha(tv)$ and we obtain
\begin{align*}
    H^{(\alpha,\mu^0,m)}(tv) =\alpha-\int_{\sphere_0(0)}\pare{\alpha-d(tv,x)} d\mu^0(x) = \alpha - \alpha + t = t.
\end{align*}

In the rest of the proof we assume $s > 0$.
For $x \in \sphere_s(0)$, we have $d(tv,x)=\sqrt{s^2+t^2-2st\cos \theta'}$, where $\theta'$ is the angle between $v$ and $x$. 
This implies that the function under the integral sign in \eqref{eq H form inproof} can be written as a function of $\theta'$. 
Let $x$ be a random vector in $\R^m$ drawn according to $\mu^s$ and denote by $\hat{\mu}$ the probability measure of $\theta'$. According to \cref{obs link}, the random variable $\theta'$ has the same probability measure as the random variable $\theta$ studied in \cref{sec theta}.

We now consider separately two cases. 
In the first case we assume $0 < s \le \alpha-t$. 
We then have $\sphere_s(0)\subseteq \ball_\alpha(tv)$ and from \eqref{eq H form inproof} we obtain
\begin{align}
\label{eq H form inproof case a}
    H^{(\alpha,\mu^s,m)}(tv)
    & =\alpha-s-\int_{\sphere_s(0)}\pare{\alpha-d(tv,x)} d\mu^s(x),
\end{align}
thus $H^{(\alpha,\mu^s,m)}(tv) = \E d(tv,x)-s$, where $x$ is a random vector drawn according to $\mu^s$.
From \eqref{eq H form inproof case a} we continue
\begin{align*}
    H^{(\alpha,\mu^s,m)}(tv)
    & =\alpha-s-\int_0^\pi \pare{\alpha-\sqrt{s^2+t^2-2st\cos\theta'}}d\hat{\mu}(\theta') \\
    & =\alpha-s-\int_0^\pi \pare{\alpha-\sqrt{s^2+t^2-2st\cos\theta}}d\Tilde{\mu}(\theta).
\end{align*}

In the second case we assume $s \ge \alpha-t$.
We define the angle
$$
\bar \theta:=\arccos\frac{s^2+t^2-\alpha^2}{2st},
$$ 
and observe that $\bar \theta \le \pi$.
Then we get
\begin{align*}
     H^{(\alpha,\mu^s,m)}(tv) & = \alpha -s-\int_0^{\bar \theta}\pare{\alpha-\sqrt{s^2+t^2-2st\cos\theta'}}d\hat{\mu}(\theta') \\
     & = \alpha -s-\int_0^{\bar \theta}\pare{\alpha-\sqrt{s^2+t^2-2st\cos\theta}}d\Tilde{\mu}(\theta).
\end{align*}
\end{prf}

\paragraph{Analysis of the function $T^{(\alpha,m)}$.}

Our goal in the next lemmas is to study the properties of $T^{(\alpha,m)}(z)$ in order to obtain a lower bound for it.

\begin{lemma}
\label{lem T1}
Let $r,\alpha \in \R_+$ with $\alpha > r$.
Let $z \in \ball_{r}(0) \setminus \{0\}$ with $\norm{z} \le \alpha-r$.
Then we have $T^{(\alpha,m)}(z)> 0$.
\end{lemma}

\begin{prf}
From \cref{obs H form} with $s=r$ and $tv=z$ we have $T^{(\alpha,m)}(z) = \E d(z,x) - r = \E d(z,x) - \E \norm{x}$, where $x$ is a random vector drawn according to $\mu^r$.
Since $z \neq 0$, from \cref{lm ball center m = 1,lm ball center m >= 2}, we obtain $\E d(z,x) - \E \norm{x}>0$.
\end{prf}


\begin{lemma}
\label{lem T2}
Let $m \ge 2$ and let $r \in \R_+$.
Let $z \in \ball_r(0)$.
Then $T^{(\alpha,m)}(z)$ is strictly increasing in $\alpha$ when $\alpha \in (r, r+\norm{z})$ and is constant in $\alpha$ when $\alpha \ge r+\norm{z}$.
\end{lemma}

\begin{prf}
Let $v$ be a unit vector in $\R^m$.
Since $T^{(\alpha,m)}(z)$ is invariant under rotations centered in the origin, it suffices to consider vectors $z \in \ball_r(0)$ of the form $z=tv$, for $t \in [0,r]$.

Consider first the case $\alpha \ge r+ t$.
From \cref{obs H form} with $s=r$ we have $T^{(\alpha,m)}(tv)=\E d(tv,x) - r$, where $x$ is a random vector drawn according to $\mu^r$. 
Hence in this case $T^{(\alpha,m)}(tv)$ is constant in $\alpha$.

Next, consider the case $\alpha \in (r, r+t)$. 
From \cref{obs H form} with $s=r$ we have
\begin{align*}
    T^{(\alpha,m)}(tv) & =\alpha -r-\int_0^{\bar \theta}\pare{\alpha-\sqrt{r^2+t^2-2rt\cos\theta}}d\Tilde{\mu}(\theta) \\
    & =\alpha -r-\int_0^{\bar \theta}\pare{\alpha-\sqrt{r^2+t^2-2rt\cos\theta}} p^{(m)}(\theta) d\theta,
\end{align*}
where 
\begin{align*}
    \bar \theta:=\arccos\frac{r^2+t^2-\alpha^2}{2rt}<\pi.
\end{align*}
We derive with respect to the variable $\alpha$ and obtain 
\begin{align*}
    \frac{\partial T^{(\alpha,m)}}{\partial \alpha}(tv) 
    & = 1-\int_{0}^{\bar \theta}d\Tilde{\mu}(\theta)-\pare{\alpha-\sqrt{r^2+t^2-2rt\cos\bar \theta}}p^{(m)}(\bar \theta) \frac{\partial\bar \theta}{\partial \alpha
    } \\
    & = 1-\int_{0}^{\bar \theta}d\Tilde{\mu}(\theta) 
    = 1-P(\theta \le \bar \theta) > 0.
\end{align*}
Here, the second equality holds because $\alpha-\sqrt{r^2+t^2-2rt\cos\bar \theta} = 0$ and the second equality holds because $\bar \theta < \pi$.
Hence in this case $T^{(\alpha,m)}(tv)$ is strictly increasing in $\alpha$.
\end{prf}

\begin{lemma}
\label{lem T4}
Let $m \ge 2$ and let $r,\alpha \in \R_+$ with $\alpha > r$.
Let $z \in \ball_{r}(0)$ and $z' \in B^{m+1}_{r}(0)$ with $\norm{z} = \norm{z'}$.
Then we have $T^{(\alpha,m+1)}(z') \ge T^{(\alpha,m)}(z)$.
\end{lemma}

\begin{prf}
Let $z=tv$ and $z'=tv'$, where $v$ is a unit vector in $\R^m$ and $v'$ is a unit vector in $\R^{m+1}$. 
Then according to \cref{obs H form} with $s=r$, we have 
\begin{align*}
    T^{(\alpha,m+1)}(tv')-T^{(\alpha,m)}(tv)
    = \int_0^{\bar \theta}\pare{\alpha-\sqrt{r^2+t^2-2rt\cos\theta}}(p^{(m)}(\theta)-p^{(m+1)}(\theta))d\theta,
\end{align*}
where 
\begin{align*}
\bar \theta:=
\begin{cases}
\pi & \text{if } t \le \alpha-r, \\
\arccos\frac{r^2+t^2-\alpha^2}{2rt} < \pi & \text{if } t>\alpha-r.
\end{cases}
\end{align*}
We let $f(t,\theta):=\alpha-\sqrt{r^2+t^2-2rt\cos\theta}$ and write
\begin{align}
\label{eq L22 start}
    T^{(\alpha,m+1)}(tv')-T^{(\alpha,m)}(tv)
    = \int_0^{\bar \theta} f(t,\theta)(p^{(m)}(\theta)-p^{(m+1)}(\theta))d\theta.
\end{align}
We define $\hat{f}(t,\theta):=f(t,\pi-\theta)=\alpha-\sqrt{r^2+t^2+2rt\cos\theta}.$ 
It can be checked that $f(t,\theta)$ is a decreasing function in $\theta$, when $\theta \in (0,\pi)$ and $t$ is fixed in $[0,r]$.
Furthermore, we have $f(t,\bar \theta) \ge 0$.
Thus $f(t,\theta) \ge 0$ when $\theta \le \bar \theta$.
Next, we will discuss several cases for $\bar \theta$.

In the first case we assume $\bar \theta \le \frac{\pi}{2}$.
From \eqref{eq L22 start} and \cref{obs monotone}, we obtain
$
    T^{(\alpha,m+1)}(tv')-T^{(\alpha,m)}(tv) 
    \ge 0.
$

In the second case we assume $\bar \theta > \frac{\pi}{2}$.
Let $s_m \in (0,1)$ be the threshold for $p^{m}(\theta)-p^{(m+1)}(\theta)$ from \cref{obs ptheta}.
Let $\psi \in (\pi/2,\pi)$ such that $\sin \psi=s_m$.

We first show that 
\begin{align}
\label{eq L22 target 2}
\int_{\bar\theta}^\pi f(t,\theta)(p^{(m)}(\theta)-p^{(m+1)}(\theta))d\theta \le 0. 
\end{align}
If $\bar\theta =\pi$, \eqref{eq L22 target 2} obviously hold, so we assume $\bar\theta < \pi$.
Now assume $\bar\theta \in [\psi,\pi)$.
We have $f(t,\bar\theta)=0$, thus $f(t,\theta) \le 0$ for every $\theta \in [\bar\theta,\pi]$.
On the other hand we have $p^{(m)}(\theta)-p^{(m+1)}(\theta) \ge 0$ for every $\theta \in [\bar\theta,\pi]$.
Hence \eqref{eq L22 target 2} holds also in this case.
So we now assume $\frac{\pi}{2}<\bar\theta<\psi$. 
We notice that 
\begin{align*}
    \int_{\bar\theta}^\pi f(t,\theta)(p^{(m)}(\theta)-p^{(m+1)}(\theta))d\theta 
    & =-\int_{\bar\theta}^\pi -f(t,\theta)(p^{(m)}(\theta)-p^{(m+1)}(\theta))d\theta \\
    & =-\int_{\pi-\bar\theta}^0 -f(t,\pi-\xi)(p^{(m)}(\pi-\xi)-p^{(m+1)}(\pi-\xi))d(\pi-\xi) \\
    & =-\int_{0}^{\pi-\bar\theta} -\hat{f}(t,\xi)(p^{(m)}(\xi)-p^{(m+1)}(\xi))d\xi.
\end{align*}
Here, in the second equality we perform the change of variable $\theta = \pi - \xi$ and in the third equality we use the fact that $p^{(m)}(\xi)=p^{(m)}(\pi-\xi)$ for every $m$ and every $\xi$.
We observe that, for $\xi \in [0,\pi-\bar\theta]$, $-\hat{f}(t,\xi)$ is a decreasing function and 
$$
-\hat{f}(t,\xi)
=\sqrt{r^2+t^2+2rt\cos \xi}-\alpha 
\ge \sqrt{r^2+t^2+2rt\cos (\pi- \bar\theta)}-\alpha 
=\sqrt{r^2+t^2-2rt\cos \bar\theta}-\alpha
= 0.
$$ 
By \cref{obs monotone}, we conclude that $\int_{0}^{\pi-\bar\theta} -\hat{f}(t,\xi)(p^{(m)}(\xi)-p^{(m+1)}(\xi))d\xi \ge 0$, thus \eqref{eq L22 target 2} holds.
This concludes the proof of \eqref{eq L22 target 2}.

Next, we show
\begin{align}
\label{eq L22 target}
    T^{(\alpha,m+1)}(tv')-T^{(\alpha,m)}(tv) \ge \int_0^{\frac{\pi}{2}}(f(t,\theta)+\hat{f}(t,\theta))(p^{(m)}(\theta)-p^{(m+1)}(\theta))d\theta.
\end{align}
From \eqref{eq L22 start} we have
\begin{align*}
    T^{(\alpha,m+1)}(tv')-T^{(\alpha,m)}(tv) 
    & = \int_0^{\bar\theta}f(t,\theta)(p^{(m)}(\theta)-p^{(m+1)}(\theta))d\theta \\
    & \hspace{-1in} = \int_0^{\frac{\pi}{2}}f(t,\theta)(p^{(m)}(\theta)-p^{(m+1)}(\theta))d\theta+\int_{\frac{\pi}{2}}^{\bar\theta}f(t,\theta)(p^{(m)}(\theta)-p^{(m+1)}(\theta))d\theta.
\end{align*}
Now note that
\begin{align*}
    \int_{\frac{\pi}{2}}^{\bar\theta}f(t,\theta)(p^{(m)}(\theta)-p^{(m+1)}(\theta))d\theta
    & \ge \int_{\frac{\pi}{2}}^{\pi}f(t,\theta)(p^{(m)}(\theta)-p^{(m+1)}(\theta))d\theta \\
    & = \int_{\frac{\pi}{2}}^{0}f(t,\pi-\xi)(p^{(m)}(\pi-\xi)-p^{(m+1)}(\pi-\xi))d(\pi-\xi) \\
    & = \int_{0}^{\frac{\pi}{2}}f(t,\pi-\xi)(p^{(m)}(\xi)-p^{(m+1)}(\xi))d\xi.
\end{align*}
Here, in the inequality we use \eqref{eq L22 target 2}, in the first equality we perform the change of variable $\theta = \pi - \xi$, and in the last equality, we use the fact that $p^{(m)}(\theta)=p^{(m)}(\pi-\theta)$ for every $m$ and every $\theta$.
Thus we continue
\begin{align*}
    T^{(\alpha,m+1)}(tv')-T^{(\alpha,m)}(tv) 
    & \ge \int_0^{\frac{\pi}{2}}f(t,\theta)(p^{(m)}(\theta)-p^{(m+1)}(\theta))d\theta \\ 
    & \qquad +\int_{0}^{\frac{\pi}{2}}f(t,\pi-\xi)(p^{(m)}(\xi)-p^{(m+1)}(\xi))d\xi \\
    & = \int_0^{\frac{\pi}{2}}(f(t,\theta)+\hat{f}(t,\theta))(p^{(m)}(\theta)-p^{(m+1)}(\theta))d\theta,
\end{align*}
where in the equality we use the fact that $\hat{f}(t,\theta)=f(t,\pi-\theta).$
This concludes the proof of \eqref{eq L22 target}.

To finish the proof it suffices to show that $f(t,\theta)+\hat{f}(t,\theta)$ is a nonnegative decreasing function for $\theta \in (0,\frac{\pi}{2})$.
In fact, using \eqref{eq L22 target} and \cref{obs monotone}, we can then conclude that $T^{(\alpha,m+1)}(tv')-T^{(\alpha,m)}(tv) \ge 0.$

We derive $f(t,\theta)+\hat{f}(t,\theta)$ with respect to the variable $\theta$ and obtain 
\begin{align*}
    \frac{\partial (f+\hat{f})}{\partial\theta} (t,\theta) & =\frac{\partial\pare{2\alpha-\sqrt{r^2+t^2-2rt\cos\theta}-\sqrt{r^2+t^2+2rt\cos\theta}}}{\partial\theta} \\
    & = r t\sin \theta\pare{\frac{1}{\sqrt{r^2+t^2+2rt\cos \theta}}-\frac{1}{\sqrt{r^2+t^2-2rt\cos\theta}}}.
\end{align*}
Hence the derivative is nonpositive for $\theta \in (0,\frac{\pi}{2})$ and so $f(t,\theta)+\hat{f}(t,\theta)$ is decreasing for $\theta \in (0,\frac{\pi}{2})$.
This implies that, for $\theta \in (0,\frac{\pi}{2})$, we have $f(t,\theta)+\hat{f}(t,\theta) \ge f(t,\frac{\pi}{2})+\hat{f}(t,\frac{\pi}{2})=2\alpha-2\sqrt{r^2+t^2}$.
The latter quantity is nonnegative.
In the case $t>\alpha-r$, this is because $\cos \bar\theta =\frac{r^2+t^2-\alpha^2}{2rt} < 0$ when $\bar\theta > \frac{\pi}{2}.$
In the case $t\le\alpha-r$, this is because we have $\alpha^2 \ge (r+t)^2 \ge r^2 + t^2$.
\end{prf}

In the next lemma, we use \cref{lem T6} to bound the function $T^{(\alpha,m)}(z)$.

\begin{lemma}\label{lm bound1}
Let $m \ge 2$, let $r \in \R_+$, let $\epsilon \in (0, 1)$, and let $\alpha = r(1+\epsilon)$. 
Let $z \in \ball_r(0)$ with $\norm{z} \ge \epsilon r$.
Then we have
\begin{align*}
    T^{(\alpha,m)}(z) \ge \frac{r \epsilon^2}{8}-r \sqrt{\frac{\pi m}{2}}\pare{1-\frac{\epsilon^2}{16}}^{\frac{m-2}{2}}.
\end{align*}
\end{lemma}

\begin{prf}
Let $v$ be a unit vector in $\R^m$.
Since $T^{(\alpha,m)}(z)$ is invariant under rotations centered in the origin, it suffices to consider vectors $z \in \ball_r(0)$ of the form $z=tv$, for $t \in [\epsilon r,r]$.
We define 
\begin{align*}
\theta_\epsilon := \arccos\frac{\epsilon}{4}< \frac{\pi}{2}, 
\qquad 
\bar \theta:=\arccos\frac{r^2+t^2-\alpha^2}{2rt}<\pi.
\end{align*}

We consider two cases. 
In the first case we assume $\bar \theta\le\theta_\epsilon$.
Then, according to Observation~\ref{obs H form}, we have 
\begin{align*}
    T^{(\alpha,m)}(tv) & 
    = \alpha -r-\int_0^{\bar\theta}\pare{\alpha-\sqrt{r^2+t^2-2rt\cos\theta}}p^{(m)}(\theta) d\theta \\ 
    & \ge \alpha(1-\Pr (\theta \in (0,\bar\theta)))-r 
    \ge \alpha(1-\Pr (\theta \in (0,\theta_\epsilon)))-r  \\
    & \ge r \sqrt{1+\frac{\epsilon^2}{2}}(1-\Pr (\theta \in (0,\theta_\epsilon)))-r
    \ge r \sqrt{1+\frac{\epsilon^2}{2}}\pare{1-\frac{\sqrt{\pi}}{2}\sqrt{\frac{m}{2}}\sin^{m-2}\theta_\epsilon}-r \\
    & =r \sqrt{1+\frac{\epsilon^2}{2}}\pare{1-\frac{\sqrt{\pi}}{2}\sqrt{\frac{m}{2}}\pare{1-\frac{\epsilon^2}{16}}^{\frac{m-2}{2}}}-r.
\end{align*}
The third inequality holds because $\alpha = r(1+\epsilon)>r\sqrt{1+\epsilon^2/2},$ 
and the fourth inequality follows from \cref{lem T6}.

In the second case we assume $\bar \theta>\theta_\epsilon$.
Then, according to Observation \ref{obs H form}, we have 
\begin{align*}
        T^{(\alpha,m)}(tv) & 
        =\alpha -r-\int_0^{\bar\theta}\pare{\alpha-\sqrt{r^2+t^2-2rt\cos\theta}}p^{(m)}(\theta) d\theta \\ 
        & = \alpha-r-\int_0^{\theta_\epsilon}\pare{\alpha-\sqrt{r^2+t^2-2rt\cos\theta}}p^{(m)}(\theta) d\theta \\
        & \qquad -\int_{\theta_\epsilon}^{\bar\theta}\pare{\alpha-\sqrt{r^2+t^2-2rt\cos\theta}}p^{(m)}(\theta) d\theta  \\ 
        & \ge \alpha-r-\alpha \Pr (\theta \in (0,\theta_\epsilon))-\int_{\theta_\epsilon}^{\bar \theta}\pare{\alpha-\sqrt{r^2+t^2-2rt\cos\theta}}p^{(m)}(\theta) d\theta   \\
        & \ge \alpha-r-\alpha \Pr (\theta \in (0,\theta_\epsilon))-\pare{\alpha-r\sqrt{1+\frac{\epsilon^2}{2}}}\Pr (\theta \in (\theta_\epsilon,\bar\theta)) \\
         & \ge \alpha-r-\alpha \Pr (\theta \in (0,\theta_\epsilon))-\pare{\alpha-r\sqrt{1+\frac{\epsilon^2}{2}}}(1-\Pr (\theta \in (0,\theta_\epsilon))) \\
        & = r \sqrt{1+\frac{\epsilon^2}{2}}\pare{1-\Pr (\theta \in (0,\theta_\epsilon))}-r 
        \ge r \sqrt{1+\frac{\epsilon^2}{2}}\pare{1-\frac{\sqrt{\pi}}{2}\sqrt{\frac{m}{2}}\sin^{m-2}\theta_\epsilon} -r \\
         & =r \sqrt{1+\frac{\epsilon^2}{2}}\pare{1-\frac{\sqrt{\pi}}{2}\sqrt{\frac{m}{2}}\pare{1-\frac{\epsilon^2}{16}}^{\frac{m-2}{2}}}-r.
\end{align*}
Here, the second inequality holds because 
$$
\sqrt{r^2+t^2-2rt\cos\theta}
\ge \sqrt{r^2+t^2-2rt\cos\theta_\epsilon} 
\ge \sqrt{r^2+\epsilon^2r^2-2\epsilon r^2\cos\theta_\epsilon}
= r \sqrt{1+\epsilon^2/2}
$$ 
when $\theta \ge \theta_\epsilon$ and $t \ge \epsilon r$ and the last inequality follows from \cref{lem T6}.

Since $ (1+\epsilon^2/8) \le \sqrt{1+\epsilon^2/2} \le 2$, we obtain
\begin{align*}
    T^{(\alpha,m)}(z) 
    \ge r \sqrt{1+\frac{\epsilon^2}{2}}\pare{1-\frac{\sqrt{\pi}}{2}\sqrt{\frac{m}{2}}\pare{1-\frac{\epsilon^2}{16}}^{\frac{m-2}{2}}}-r 
    \ge \frac{r \epsilon^2}{8}-r\sqrt{\frac{\pi m}{2}}\pare{1-\frac{\epsilon^2}{16}}^{\frac{m-2}{2}}.
\end{align*}
\end{prf}

\paragraph{Analysis of the function $H^{(\alpha,\mu,m)}$.}

Our next goal is to derive a lower bound on $H^{(\alpha,\mu,m)}$ using the lower bound for $T^{(\alpha,m)}$ given in \cref{lm bound1}.

\begin{lemma}
\label{lm distribution}
Let $(\mu,\ball_{r}(0))$ be a probability space with $m \ge 2$ that satisfies \ref{ass rotation} and let $\alpha > r$. 
Let $z \in \ball_{r}(0)$.
Then we have $H^{(\alpha,\mu,m)}(z) \ge T^{(\alpha,m)}(z)$.
\end{lemma}

\begin{prf}
Let $x$ be a random vector drawn according to $\mu$. 
Since $(\mu,\ball_{r}(0))$ satisfies \ref{ass rotation}, we know that, conditioned on the event that $\norm{x}=s$, $x$ is drawn according to $\mu^s$.

In the following, we let $D := \ball_{\alpha}(z)\cap \ball_{r}(0),$ we denote by $I_D(x)$ the indicator function of $D$, and by $\nu$ be the probability measure of $\norm{x}$.
Then we have 
\begin{align*}
    H^{(\alpha,\mu,m)}(z) 
    & =\int_{\ball_{r}(0)} \pare{\alpha-\norm{x}-(\alpha-d(z,x))I_D(x)}d \mu (x)\\
    & = \int_0^r d\nu(s) \int_{\sphere_s(0)} \pare{\alpha-\norm{x}-(\alpha-d(z,x))I_D(x)} d\mu^s(x) \\
    & = \int_0^r H^{(\alpha,\mu^s,m)}(z) d\nu(s).
\end{align*}
To complete the proof of the lemma it suffices to show that the scalar $r$ achieves
\begin{align}
\label{eq H min}
\min \bra{H^{(\alpha,\mu^s,m)}(z) \mid s \in [0,r]}.
\end{align}
In fact, this implies
\begin{align*}
    H^{(\alpha,\mu,m)}(z) \ & 
    = \int_0^r H^{(\alpha,\mu^s,m)}(z) d\nu(s)
    \ge H^{(\alpha,\mu^r,m)}(z)
    = T^{(\alpha,m)}(z).
\end{align*}



Let $v$ be a unit vector in $\R^m$.
Since $H^{(\alpha,\mu^s,m)}(z)$ is invariant under rotations centered in the origin, it suffices to consider vectors $z \in \ball_r(0)$ of the form $z=tv$, for $t \in [0,r]$.

If $s=0$, then from \cref{obs H form} we have $H^{(\alpha,\mu^s,m)}(tv)=t$. 
If $0< s \le \alpha-t$, \cref{obs H form} implies  $H^{(\alpha,\mu^s,m)}(tv)=\E d(tv,x)-s=\E d(tv,x)- \E \norm{x}=\E (d(tv,x)-\norm{x}) \le \E t = t$, where $x$ is a random vector drawn according to $\mu^s$.
So we only need to show \eqref{eq H min} for $s \in (0,r]$ rather than $s \in [0,r]$. 

We now consider separately two cases. 
In the first case we assume $s \in (0, \alpha-t]$.
From \cref{obs H form} we can write
\begin{align*}
    H^{(\alpha,\mu^s,m)}(tv)
    =\alpha-s-\int_0^\pi \pare{\alpha-\sqrt{s^2+t^2-2st\cos\theta}}d\Tilde{\mu}(\theta).
\end{align*}
We derive with respect to the variable $s$ and obtain 
\begin{align*}
    \frac{\partial H^{(\alpha,\mu^s,m)}}{\partial s} (tv) 
    = \int_0^{\pi} \frac{s-t\cos \theta}{\sqrt{s^2+t^2-2st\cos\theta}} d\Tilde{\mu}(\theta)-1 
    \le 0,
\end{align*}
because $(s-t\cos \theta)/\sqrt{s^2+t^2-2st\cos\theta} \le 1$. 
This implies that the function $H^{(\alpha,\mu^s,m)}(tv)$ is decreasing in $s$, when $s \in (0, \alpha-t]$.

In the second case we assume $s \in (\alpha-t,r]$.
From \cref{obs H form} we can write
\begin{align*}
     H^{(\alpha,\mu^s,m)}(tv) 
     & =\alpha -s-\int_0^{\bar \theta}\pare{\alpha-\sqrt{s^2+t^2-2st\cos\theta}}d\Tilde{\mu}(\theta) \\
     & =\alpha -s-\int_0^{\bar \theta}\pare{\alpha-\sqrt{s^2+t^2-2st\cos\theta}}p^{(m)}(\theta)d\theta,
\end{align*}
where
$$
\bar \theta:=\arccos\frac{s^2+t^2-\alpha^2}{2st} < \pi.
$$ 
We derive with respect to the variable $s$ and obtain 
\begin{align*}
    \frac{\partial H^{(\alpha,\mu^s,m)}}{\partial s} (tv) 
    & = -1 +\int_0^{\bar \theta}\frac{s-t\cos \theta}{\sqrt{s^2+t^2-2st\cos\theta}}d\Tilde{\mu}(\theta)-\pare{\alpha-\sqrt{s^2+t^2-2st\cos\bar \theta}}p^{(m)}(\bar \theta) \frac{\partial\bar \theta}{\partial s
    } \\
    \ & = -1 +\int_0^{\bar \theta}\frac{s-t\cos \theta}{\sqrt{s^2+t^2-2st\cos\theta}}d\Tilde{\mu}(\theta) 
    \le -1 +\Pr(\theta \le \bar \theta) < 0.
\end{align*}
Here, the second equality holds because $\alpha-\sqrt{s^2+t^2-2st\cos\bar \theta}=0$ and the first inequality holds because $(s-t\cos \theta)/\sqrt{s^2+t^2-2st\cos\theta} \le 1$ and $\bar \theta < \pi$.
So we conclude that $H^{(\alpha,\mu^s,m)}(tv)$ is also decreasing in $s$, when $s \in (\alpha-t,r]$.

The above two cases imply that  $H^{(\alpha,\mu^s,m)}(tv)$ is decreasing in $s$, when $s \in (0,r]$.
Thus, for every $z \in \ball_{r}(0)$, the scalar $r$ achieves \eqref{eq H min}.
\end{prf}

According to \cref{lm distribution}, we know that every lower bound for $T^{(\alpha,m)}(z)$ is also a lower bound for $H^{(\alpha,\mu,m)}(z)$.

\begin{lemma}\label{lm H}
Let $(\mu,\ball_{r}(0))$ be a probability space with $m \ge 2$ that satisfies \ref{ass rotation}, let $\epsilon \in (0,1)$, and let $\alpha = r(1+\epsilon)$.
Let $z \in \ball_{r}(0)$ with $\norm{z} \ge \epsilon r$.
Then we have
\begin{align*}
    H^{(\alpha,\mu,m)}(z) 
    \ge \frac{r\epsilon^2}{8}-r\sqrt{\frac{\pi m}{2}}\pare{1-\frac{\epsilon^2}{16}}^{\frac{m-2}{2}}.
\end{align*}
\end{lemma}

\begin{prf}
Directly from \cref{lm distribution,lm bound1}.
\end{prf}


\paragraph{Analysis of the function $R^{(\alpha,\mu,m)}$.}


In the next lemma, we will provide an upper bound for $R^{(\alpha,\mu,m)}(z)$

\begin{lemma}
\label{lm residual}
Let $(\mu,\ball_{r}(0))$ be a probability space with $m \ge 2$ that satisfies \ref{ass rotation}, \ref{ass dim} and let $\alpha > r$. 
Let $z \in \R^m$ with $\norm{z} \in (\alpha,\alpha+r)$. 
Then we have
\begin{align*}
    R^{(\alpha,\mu,m)}(z)\le (\alpha+r-\norm{z})\frac{\sqrt{\pi}}{2}\sqrt{\frac{m}{2}}\pare{\frac{\alpha}{\norm{z}}}^{m-2}.
\end{align*}
\end{lemma}

\begin{prf}
For every $x \in \ball_r(0)$, let $\theta'(x)$ be the angle between $x$ and $z$.
Let $D:=\ball_{\alpha}(z)\cap \ball_{r}(0)$. For every $x \in D$, we denote by $\Pi_x(z)$ the orthogonal projection of $z$ on the line containing $0$ and $x$. 
Then we know that for every $x \in D$ we have
\begin{align*}
   \sin \theta'(x)=\frac{d(z,\Pi_x(z))}{\norm{z}} \le \frac{d(z,x)}{\norm{z}} \le \frac{\alpha}{\norm{z}}<1. 
\end{align*}
Let $\theta^*:=\arcsin(\alpha/\norm{z}) \in (0,\pi/2)$. 
Thus we obtain $D\subseteq \{x \in \ball_r(0) \mid \theta'(x) \le \theta^*\}$.

According to \cref{obs link}, the random variable $\theta'$ has the same probability measure as the random variable $\theta$ studied in \cref{sec theta}.
So we obtain
\begin{align*}
    R^{(\alpha,\mu,m)}(z) 
    & =\int_{\ball_{\alpha}(z)\cap \ball_{r}(0)} (\alpha-d(z,x))d\mu(x)
    \le (\alpha+r-\norm{z}) \Pr(x \in D) \\
    & \le (\alpha+r-\norm{z}) \Pr(\theta \le \theta^*)
    \le (\alpha+r-\norm{z})\frac{\sqrt{\pi}}{2}\sqrt{\frac{m}{2}}\pare{\frac{\alpha}{\norm{z}}}^{m-2},
\end{align*}
where the first inequality holds because $d(z,x) \ge \norm{z}-r$ and the last inequality holds by \cref{lem T6}.
\end{prf}

\subsection{Proof of \cref{th main2}}

It suffices to check that all assumptions of \cref{th prob} are satisfied.
For every $i \neq j$, we define $\Theta_{ij} >0$ so that  $d(c_i,c_j)=(1+\beta)R+\max\{r_i,r_j\}+2\Theta_{ij}$. 
We also define $\Theta:=\min_{i \neq j}\Theta_{ij}$ and $\gamma:=\beta R+\Theta$. 
For every $i \in [k]$, denote by $E_i := \Ex d(x,c_i)$, where $x$ is a random vector drawn according to $\mu_i$. 
We can then bound $\gamma$ as follows. 
\begin{align*}
\max_{i \in [k]}\beta_i(r_i-E_i)
& <\beta R
<\gamma
<(1+\beta)R+2\Theta-R \\
& \le \min_{i \neq j} d(c_i,c_j) - \max\{r_i,r_j\} - R
\le\min_{i \in [k]}\beta_i(D_i-E_i).
\end{align*}

For every $i \in [k]$, we define $\alpha_i:=E_i+\frac{\gamma}{\beta_i}$. 
It remains to show that for every $i \in [k]$, $c_i$ is the unique point that achieves $\max \{G^{\alpha}(z) \mid z \in \ball_{r_i}(c_i)\}$.
Using the fact that $r_i \in [r,R]$, $\beta_i \in [1,\beta]$, and $E_i \in (0,R]$ for $i \in [k]$, we obtain 
\begin{align}
\label{eq alpha bound}
    r_i
    < R+\frac{\Theta}{\beta} 
    = \frac \gamma \beta
    < \alpha_i 
    \le R + \gamma
    = (1+\beta)R+\Theta< \min_{j 
    \in [k] \setminus \{i\}}d(c_i,c_j)-r_i=D_i .
\end{align}
\cref{lm geometry} implies that $\ball_{\alpha_j}(c_i) \cap \ball_{r_j}(c_j) = \emptyset$ for every $j \in [k]\setminus \{i\}$. From \cref{obs G}, we know that for every $i \in [k]$,
\begin{align*}
    G^\alpha(c_i)=\ \beta_i \int_{\ball_{r_i}(c_i)} (\alpha_i-d(c_i,x))d\mu_i(x).
\end{align*}

From \cref{obs G} we obtain that for every $z \in \ball_{r_i}(c_i)$, 
\begin{align}
\label{eq certificate}
\begin{split}
    G^\alpha(c_i)-G^\alpha(z) & = \beta_i \pare{\int_{\ball_{r_i}(c_i)} (\alpha_i-d(c_i,x))d\mu_i(x)- \int_{\ball_{\alpha_i}(z) \cap \ball_{r_i}(c_i)} (\alpha_i-d(z,x)) d\mu_i(x)}\\
    & \qquad -\sum_{j \in [k] \setminus \{i\}}\beta_j\int_{\ball_{\alpha_j}(z) \cap \ball_{r_j}(c_j)} (\alpha_j-d(z,x)) d\mu_j(x).
\end{split}
\end{align}
It then suffices to show that, when $\Theta$ is large, the right hand side of \eqref{eq certificate} is positive for every $z \in \ball_{r_i}(c_i) \setminus \{c_i\}$.
So we now fix a vector $z$ in $\ball_{r_i}(c_i) \setminus \{c_i\}$.

From \eqref{eq alpha bound} we obtain
\begin{align}\label{eq alphai}
    \alpha_i 
    > R+\frac{\Theta}{\beta}
    =\pare{1+\frac{\Theta}{\beta R}}R 
    \ge \pare{1+\frac{\Theta}{\beta R}}r_i. 
\end{align}
We now consider separately two cases.

In the first case we assume $d(c_i,z)\le \Theta r_i/(\beta R)$.
Notice that under this assumption, for every $j \in [k] \setminus \{i\}$ and for every $x \in \ball_{\alpha_j}(z),\  y \in \ball_{r_j}(c_j)$,
we have 
\begin{align*}
    d(x,y) & \ge d(z,c_j)-r_j-\alpha_j \ge d(c_i,c_j)-d(z,c_i)-r_j-\alpha_j \\
    & \ge d(c_i,c_j)-d(z,c_i)-r_j-(1+\beta)R-\Theta \\
    & \ge 2\Theta_{ij}-\Theta-d(z,c_i)
    \ge \Theta-d(z,c_i)> 0,
\end{align*}
where the third inequality follows from \eqref{eq alpha bound}.
So we must have $\ball_{\alpha_j}(z) \cap \ball_{r_j}(c_j)=\emptyset$ for $j \in [k] \setminus \{i\}$. 
Therefore, from \eqref{eq certificate} we have
\begin{align*}
    G^\alpha(c_i)-G^\alpha(z) \ & =\beta_i\pare{\int_{\ball_{r_i}(c_i)} (\alpha_i-d(c_i,x))d\mu_i(x)- \int_{\ball_{\alpha_i}(z) \cap \ball_{r_i}(c_i)} (\alpha_i-d(z,x)) d\mu_i(x)} \\
    \ & = \beta_i H^{(\alpha_i,\mu'_i,m)}(z-c_i) \ge \beta_i T^{(\alpha_i,m)}(z-c_i)>0,
\end{align*}
where $\mu'_i$ is the image of $\mu_i$ under the translation $x'=x-c_i$.
The first inequality above follows from \cref{lm distribution} and the last inequality follows from \cref{lem T1} because from \eqref{eq alphai} we have $d(c_i,z)\le \Theta r_i/(\beta R) < \alpha_i - r_i$.
Thus, in the first case \cref{th prob} implies that \eqref{pr LP} achieves exact recovery with high probability.

In the remainder of the proof we only need to consider the second case, where we assume $d(c_i,z) > \Theta r_i/(\beta R)$. 
We notice that in this case $d(c_i,z) \le r_i$ implies  $\Theta/(\beta R) < 1$. 
We first show that for every $j \in [k] \setminus \{i\}$, we have 
\begin{align}
\label{eq letsdoit}
    \int_{\ball_{\alpha_j}(z) \cap \ball_{r_j}(c_j)} (\alpha_j-d(z,x)) d\mu_j(x)
    \le R\frac{\sqrt{\pi}}{2}\sqrt{\frac{m}{2}}\pare{1-\frac{\Theta}{(1+\beta)R+2\Theta}}^{m-2}.
\end{align}
If $d(z,c_j) \ge \alpha_j+r_j$, then $\ball_{r_j}(c_j) \cap \ball_{\alpha_j}(z) $ contains at most one point and \eqref{eq letsdoit} clearly holds because \ref{ass dim} implies
\begin{align*}
    \int_{\ball_{\alpha_j}(z) \cap \ball_{r_j}(c_j)} (\alpha_j-d(z,x)) d\mu_j(x)=0.
\end{align*}
If $d(z,c_j) < \alpha_j+r_j$, we can apply \cref{lm residual} to $z$, since we also have 
\begin{align}
\label{eq I will need this}
    d(z,c_j) 
    \ge d(c_i,c_j)-r_i 
    \ge (1+\beta) R+ 2\Theta
    \ge \alpha_j + \Theta
    >\alpha_j,
\end{align}
where the last inequality follows by \eqref{eq alpha bound}.
If we denote by $\mu'_j$ the image of $\mu_j$ under the translation $x'=x-c_j$, we then obtain
\begin{align*}
  \int_{\ball_{\alpha_j}(z) \cap \ball_{r_j}(c_j)} & (\alpha_j-d(z,x)) d\mu_j(x) 
  = R^{(\alpha_j,\mu'_j,m)}(z-c_j) \\
  & \le (\alpha_j+r_j-d(z,c_j))\frac{\sqrt{\pi}}{2}\sqrt{\frac{m}{2}}\pare{\frac{\alpha_j}{d(z,c_j)}}^{m-2}
  \le (r_j-\Theta)\frac{\sqrt{\pi}}{2}\sqrt{\frac{m}{2}}\pare{\frac{\alpha_j}{\alpha_j+\Theta}}^{m-2} \\
  & = (r_j-\Theta)\frac{\sqrt{\pi}}{2}\sqrt{\frac{m}{2}}\pare{1-\frac{\Theta}{\alpha_j+\Theta}}^{m-2}
  \le R\frac{\sqrt{\pi}}{2}\sqrt{\frac{m}{2}}\pare{1-\frac{\Theta}{(1+\beta)R+2\Theta}}^{m-2},
\end{align*}
where in the second inequality we use $d(z,c_j) \ge \alpha_j + \Theta$ from \eqref{eq I will need this} and the last inequality follows because $\alpha_j \le (1+\beta)R+\Theta$ from \eqref{eq alpha bound}.
This concludes the proof of \eqref{eq letsdoit}.

From \eqref{eq letsdoit}
we obtain
\begin{align}
\label{eq put together a}
\begin{split}
    \sum_{j \in [k] \setminus \{i\}}\beta_j\int_{\ball_{\alpha_j}(z) \cap \ball_{r_j}(c_j)} (\alpha_j-d(z,x)) d\mu_j(x) & \le k\beta R\frac{\sqrt{\pi}}{2}\sqrt{\frac{m}{2}}\pare{1-\frac{\Theta}{(1+\beta)R+2\Theta}}^{m-2} \\
    & \le k\beta R\frac{\sqrt{\pi}}{2}\sqrt{\frac{m}{2}}\exp\pare{-\frac{(m-2)\Theta}{(1+\beta)R+2\Theta}},
\end{split}
\end{align}
where the last inequality we use the fact that $1-x \le e^{-x}$ for every $x$.

Now let $\alpha_i':= r_i(1+\frac{\Theta}{\beta R})$. We know from \eqref{eq alphai} that $\alpha_i>\alpha_i'$.
If we denote by $\mu'_i$ the image of $\mu_i$ under the translation $x'=x-c_i$, we obtain
\begin{align}
\label{eq put together b}
\begin{split}
  \ &  \int_{\ball_{r_i}(c_i)} (\alpha_i-d(c_i,x))d\mu_i(x)- \int_{\ball_{\alpha_i}(z) \cap \ball_{r_i}(c_i)} (\alpha_i-d(z,x)) d\mu_i(x)=H^{(\alpha_i,\mu'_i,m)}(z-c_i) \\ 
  \ge \ & T^{(\alpha_i,m)}(z-c_i) \ge T^{(\alpha_i',m)}(z-c_i) \ge \frac{r_i \Theta^2}{8\beta^2R^2}-r_i\sqrt{\frac{\pi m}{2}}\pare{1-\frac{\Theta^2}{16\beta^2R^2}}^{\frac{m-2}{2}} \\ 
  \ge \ & \frac{r_i\Theta^2}{8\beta^2R^2}-r_i\sqrt{\frac{\pi m}{2}}\exp\pare{-\frac{(m-2)\Theta^2}{32\beta^2R^2}}.
  \end{split}
\end{align}
The first inequality holds by \cref{lm distribution}, the second inequality holds by \cref{lem T2}, and the third inequality holds by \cref{lm bound1} with $\epsilon := \Theta / (\beta R)$ which satisfies $\epsilon \in (0,1)$.
In the last inequality we use $1-x \le e^{-x}$ for every $x$.

To show $G^\alpha(c_i)-G^\alpha(z)>0$, it is sufficient to show $(G^\alpha(c_i)-G^\alpha(z))/(\beta_ir_i)>0.$
From \eqref{eq certificate}, \eqref{eq put together a}, and \eqref{eq put together b} we then obtain 
\begin{align}\label{eq G alpha}
\begin{split}
 \frac{G^\alpha(c_i)-G^\alpha(z)}{\beta_ir_i} & 
    \ge \frac{\Theta^2}{8\beta^2R^2}-\sqrt{\frac{\pi m}{2}}\exp\pare{-\frac{(m-2)\Theta^2}{32\beta^2R^2}}-\frac{k\beta R}{\beta_ir_i}\frac{\sqrt{\pi}}{2}\sqrt{\frac{m}{2}}\exp\pare{-\frac{(m-2)\Theta}{(1+\beta)R+2\Theta}} \\
    & \ge \frac{\Theta^2}{8\beta^2R^2}-\sqrt{\frac{\pi m}{2}}\exp\pare{-\frac{(m-2)\Theta^2}{32\beta^2R^2}}-\frac{k\beta R}{r}\frac{\sqrt{\pi}}{2}\sqrt{\frac{m}{2}}\exp\pare{-\frac{(m-2)\Theta}{(1+\beta)R+2\Theta}} \\
    & \ge \frac{\Theta^2}{8\beta^2R^2}-k\sqrt{\frac{\pi m}{2}}\exp\pare{-\frac{(m-2)\Theta^2}{32\beta^2R^2}}-\frac{k\beta R}{r}\frac{\sqrt{\pi}}{2}\sqrt{\frac{m}{2}}\exp\pare{-\frac{(m-2)\Theta}{(1+\beta)R+2\Theta}}.
\end{split}
\end{align}
Note that the lower bound on $(G^\alpha(c_i)-G^\alpha(z))
/(\beta_ir_i)$ obtained in \eqref{eq G alpha} does not depend on the index $i$ and is an increasing function in $\Theta$. 
Next we show that $(G^\alpha(c_i)-G^\alpha(z))
/(\beta_ir_i)$ is positive when $\Theta>C\sqrt{k \log m /m}$, where $C$ is a large constant.
To do so we use the lower bound in \eqref{eq G alpha} and the fact that $\beta,r,R$ are fixed constants. 
We have 
\begin{align*}
& \frac{G^\alpha(c_i)-G^\alpha(z)}{\beta_ir_i}
\\ & \ge 
    \frac{\Theta^2}{8\beta^2R^2}-k\sqrt{\frac{\pi m}{2}}\exp\pare{-\frac{(m-2)\Theta^2}{32\beta^2R^2}}-\frac{k\beta R}{r}\frac{\sqrt{\pi}}{2}\sqrt{\frac{m}{2}}\exp\pare{-\frac{(m-2)\Theta}{(1+\beta)R+2\Theta}} \\
   &> k \pare{\frac{C^2\log m}{8m\beta^2R^2}-\sqrt{\frac{\pi m}{2}}\exp\pare{-\frac{C^2 (m-2)k\log m}{32m\beta^2R^2}}-\frac{\beta R}{r}\frac{\sqrt{\pi}}{2}\sqrt{\frac{m}{2}}\exp\pare{-\frac{C(m-2)\sqrt{k \log m /m}}{(1+\beta)R+2C\sqrt{k \log m /m}}}} \\
   & \ge k \pare{\frac{C^2\log m}{8m\beta^2R^2}-\sqrt{\frac{\pi m}{2}}\exp\pare{-\frac{C^2 (m-2)\log m}{32m\beta^2R^2}}-\frac{\beta R}{r}\frac{\sqrt{\pi}}{2}\sqrt{\frac{m}{2}}\exp\pare{-\frac{C(m-2)\sqrt{ \log m /m}}{(1+\beta)R+2C\sqrt{ \log m /m}}}} \\
   & = \frac{k \log m}{m} \pare{\frac{C^2}{8\beta^2R^2}-F(C,m)},
   \end{align*}
where, to simplify the notation, we let
\begin{align*}
    F(C,m)& := \frac{m}{\log m}\sqrt{\frac{\pi m}{2}}\exp\pare{-\frac{C^2 (m-2)\log m}{32m\beta^2R^2}} \\
   & \qquad +\frac{m}{\log m}\frac{\beta R}{r}\frac{\sqrt{\pi}}{2}\sqrt{\frac{m}{2}}\exp\pare{-\frac{C(m-2)\sqrt{ \log m /m}}{(1+\beta)R+2C\sqrt{ \log m /m}}}.
\end{align*}
It then suffices to show that, for every $m \ge 2$, we have $C^2/(8\beta^2R^2) > F(C,m)$ for some constant $C$ large enough.
It can be checked that for every $m \ge 2$, $F(C,m)$ is a decreasing function in $C$. 
Also it can be checked that there is some threshold $C'>0$ such that if $C \ge C'$ then $\lim_{m \to \infty}F(C,m)=0$. 
This implies that $\sup \{F(C,m) \mid C \ge C',m\ge 2\} = \sup \{F(C',m) \mid m\ge 2\} <\infty$. 
Therefore it suffices to choose $C> C'$ large enough so that $C^2/(8\beta^2R^2)>\sup \{F(C,m) \mid C \ge C',m\ge 2\}$.
\qed

\subsection{Proof of \cref{th main3}}
It suffices to check that all assumptions of \cref{cor prob} are satisfied.
Let $\Theta := \min_{i \neq j}d(c_i,c_j) - 2>1.29$, $\alpha':=1.29$, and let $\alpha_i := \alpha'$ for every $i \in [k]$. 
It remains to show that for every $i \in [k]$, $c_i$ is the unique point that achieves $\max \{G^\alpha(z) \mid z \in \ball_{1}(c_i)\}$.

Since $\min_{i \neq j}d(c_i,c_j)=2+\Theta>3.29$, we know that for every $i \in [k]$ and for every $z \in \ball_1(c_i)$, we have $\ball_{1.29}(z) \cap \ball_1(c_j)=\emptyset$ for every $j \in [k]$ with $j \neq i$.
Thus according to \cref{obs G}, for every $i \in [k]$ and for every $z \in \ball_1(c_i)$ we have 
$$
G^\alpha(z) = \int_{\ball_{1.29}(z)\cap \ball_1(c_i)}  (1.29-d(z,x))d\mu_i(x).
$$ 
Now we fix $i \in [k]$ and $z \in \ball_1(c_i) \setminus \{c_i\}$.
Then we have
\begin{align*}
    G^{\alpha}(c_i)-G^{\alpha}(z) = H^{(1.29,\mu'_i,m)}(z-c_i) \ge T^{(1.29,m)}(z-c_i) \ge T^{(1.29,2)}(z-c_i),
\end{align*}
where $\mu'_i$ is the image of $\mu_i$ under the translation $x'=x-c_i$.
The first inequality follows by \cref{lm distribution} and the second inequality follows by \cref{lem T4}. 
Since $T^{(1.29,2)}$ is invariant under rotations centered in $c_i$, we define a unit vector $v \in \R^m$ and a scalar $t \in (0,1]$ such that $z=tv$.
If $t \le 0.29$, then \cref{lem T1} implies $T^{(1.29,2)} > 0.$ 
Hence, in the remainder of the proof we assume $t>0.29$. According to \cref{obs T3}, we know that $p^{(2)}(\theta)=1/\pi$. So applying \cref{obs H form} with $r=s=1$ and $\alpha=1.29$, we get
\begin{align*}
    T^{(1.29,2)}(tv) = H^{(1.29,\mu^1,2)}(tv)
    =0.29-\frac{1}{\pi}\int_0^{\bar \theta}\pare{1.29-\sqrt{1+t^2-2t\cos\theta}}d\theta,
\end{align*}
where 
\begin{align*}
    \bar \theta= \arccos\frac{1+t^2-1.29^2}{2t}.
\end{align*}
Using the above formula it can be checked that $T^{(1.29,2)}(z) > 0$ for every $t \in (0.29,1].$
The graph of the function $T^{(1.29,2)}(z)$ can be seen in \cref{fig: 1.29}.
\begin{figure}[htbp]
    \centering
    \includegraphics[scale=0.5]{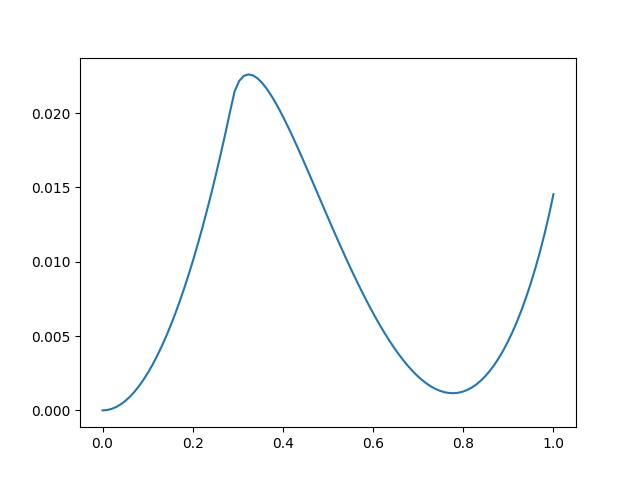}
    \caption{The graph of the function $T^{(1.29,2)}(z)$ in the proof of \cref{th main3}.}
    \label{fig: 1.29}
\end{figure}
Thus, for every $i \in [k]$, $c_i$ is the unique point that achieves $\max \{G^\alpha(z) \mid z \in \ball_{1}(c_i)\}$. 
\qed

\subsection{Proof of \cref{th main4}}




It suffices to check that all assumptions of \cref{cor prob} are satisfied.
Let $\Theta := \min_{i \neq j}d(c_i,c_j) - 2$, $\alpha':=1+\Theta/2 \in (1,1+\Theta)$, and let $\alpha_i:=\alpha'$ for every $i \in [k]$. 
It remains to show that for every $i \in [k]$, $c_i$ is the unique point that achieves $\max \{G^\alpha(z) \mid z \in \ball_{1}(c_i)\}$.

For every $i \in [k]$, from \cref{obs G} and 
\cref{lm geometry} (with $a_i = b_i = \alpha_i$) we obtain
\begin{align*}
    G^\alpha(c_i)
    = \sum_{j \in [k]} \int_{\ball_{\alpha_j}(c_i) \cap \ball
    _1(c_j)} (\alpha_j-d(c_i,x)) d\mu_j(x)
    = \int_{\ball
    _1(c_i)} (\alpha_i-d(c_i,x)) d\mu_i(x).
\end{align*}
So for every $i \in [k]$ and for every $z \in \ball_1(c_i)$ we have 
\begin{align}
\label{eq certificate 2}
\begin{split}
    G^\alpha(c_i)-G^\alpha(z) & =  \pare{\int_{\ball_1(c_i)} (\alpha_i-d(c_i,x))d\mu_i(x)- \int_{\ball_{\alpha_i}(z) \cap \ball_1(c_i)} (\alpha_i-d(z,x)) d\mu_i(x)} \\
    & \qquad -\sum_{j \in [k] \setminus \{i\}}\int_{\ball_{\alpha_j}(z) \cap \ball_1(c_j)} (\alpha_j-d(z,x)) d\mu_j(x).
\end{split}
\end{align}
We will show that,
under the assumptions of the theorem,
the right hand side of \eqref{eq certificate 2} is positive for every $z \in \ball_1(c_i)\setminus\{c_i\}$.

We now fix $i \in [k]$ and $z \in \ball_1(c_i) \setminus \{c_i\}$. 
If $d(c_i,z) \le \Theta/2$, then for every $j \in [k]\setminus \{i\}$, the set $\ball_{\alpha_j}(z) \cap \ball_{1}(c_j)$ contains at most one point. 
In this case, \ref{ass dim} implies 
\begin{align*}
    G^\alpha(c_i)-G^\alpha(z)=H^{(\alpha_i,\mu'_i,m)}(z-c_i) \ge T^{(\alpha_i,m)}(z-c_i)>0,
\end{align*}
where $\mu'_i$ is the image of $\mu_i$ under the translation $x'=x-c_i$.
The first inequality follows by \cref{lm distribution} and the last inequality follows by \cref{lem T1}.
Hence, in the remainder of the proof we assume  $d(c_i,z) > \Theta/2$. 
Then, we must have $\Theta/2<1$, since $1 \ge d(c_i,z) > \Theta/2$.

Next, we show that for every $j \in [k] \setminus \{i\}$ we have
\begin{align}
\label{eq letsdothisnow}
    \int_{\ball_{\alpha_j}(z) \cap \ball_{1}(c_j)} (\alpha_j-d(z,x)) d\mu_j(x) 
    \le \frac{\sqrt{\pi}}{2}\sqrt{\frac{m}{2}}\pare{1-\frac{\Theta}{2(1+\Theta)}}^{m-2}.
\end{align}
First consider the case $d(c_j,z) \ge \alpha_j+1$.
Then $\ball_{\alpha_j}(z) \cap \ball_{1}(c_j)$ contains at most one point and \ref{ass dim} implies 
\begin{align*}
    \int_{\ball_{\alpha_j}(z) \cap \ball_{1}(c_j)} (\alpha_j-d(z,x)) d\mu_j(x)=0 \le \frac{\sqrt{\pi}}{2}\sqrt{\frac{m}{2}}\pare{1-\frac{\Theta}{2(1+\Theta)}}^{m-2}.
\end{align*}
Now consider the case $d(c_j,z) < \alpha_j+1$. 
We obtain
\begin{align*}
    \int_{\ball_{\alpha_j}(z) \cap \ball_{1}(c_j)} (\alpha_j-d(z,x)) d\mu_j(x) 
    & =  R^{(\alpha_j,\mu'_j,m)}(z-c_j) \\
&\hspace{-1in} \le (\alpha_j+1-d(c_j,z))\frac{\sqrt{\pi}}{2}\sqrt{\frac{m}{2}}\pare{\frac{\alpha_j}{d(c_j,z)}}^{m-2}  \\ 
&\hspace{-1in} \le \pare{1-\frac{\Theta}{2}}\frac{\sqrt{\pi}}{2}\sqrt{\frac{m}{2}}\pare{\frac{1+\frac{\Theta}{2}}{1+\Theta}}^{m-2} 
\le \frac{\sqrt{\pi}}{2}\sqrt{\frac{m}{2}}\pare{1-\frac{\Theta}{2(1+\Theta)}}^{m-2},
\end{align*}
where $\mu'_j$ is the image of $\mu_j$ under the translation $x'=x-c_j$.
The first inequality follows from \cref{lm residual} since $d(c_j,z) \ge 1+\Theta>\alpha_j$ and in the second inequality we use $d(c_j,z) \ge 1+\Theta$.
This concludes the proof of \eqref{eq letsdothisnow}.

On the other hand, we know from \cref{lm H}, with $\epsilon:=\Theta/2 \in (0,1)$, that 
\begin{align}
\label{eq betterusethistoo}
\begin{split}
    \int_{\ball_{1}(c_i)} (\alpha_i-d(c_i,x))d\mu_i(x)- \int_{\ball_{\alpha_i}(z) \cap \ball_{1}(c_i)} (\alpha_i-d(z,x)) d\mu_i(x) & 
    = H^{(\alpha_i,\mu'_i,m)}(z-c_i) \\
    & \hspace{-1in} \ge \frac{\Theta^2}{32}-\sqrt{\frac{\pi m}{2}}\pare{1-\frac{\Theta^2}{64}}^{\frac{m-2}{2}}.
\end{split}
\end{align}

From \eqref{eq certificate 2}, \eqref{eq letsdothisnow}, and \eqref{eq betterusethistoo}, we obtain
\begin{align}
\label{eq G alpha 2}
\begin{split}
    G^\alpha(z)-G^\alpha(c_i) 
    & \ge \frac{\Theta^2}{32}-\sqrt{\frac{\pi m}{2}}\pare{1-\frac{\Theta^2}{64}}^{\frac{m-2}{2}}- k\frac{\sqrt{\pi}}{2}\sqrt{\frac{m}{2}}\pare{1-\frac{\Theta}{2(1+\Theta)}}^{m-2} \\
    \ & \ge \frac{\Theta^2}{32}-\sqrt{\frac{\pi m}{2}}\exp\pare{-\frac{(m-2)\Theta^2}{128}}- k\frac{\sqrt{\pi}}{2}\sqrt{\frac{m}{2}}\exp\pare{-\frac{(m-2)\Theta}{2(1+\Theta)}}.
\end{split}
\end{align}
Notice that the lower bound on $G^\alpha(z)-G^\alpha(c_i)$ obtained in \eqref{eq G alpha 2} is an increasing function in $\Theta$. 
We next show that $G^\alpha(z)-G^\alpha(c_i)$ is positive when $\Theta>C\sqrt{k \log m /m}$, where $C$ is a large constant. 
We have 
\begin{align*}
    & G^\alpha(z)-G^\alpha(c_i)\\
 \ge \ &   \frac{\Theta^2}{32}-\sqrt{\frac{\pi m}{2}}\exp\pare{-\frac{(m-2)\Theta^2}{128}}- k\frac{\sqrt{\pi}}{2}\sqrt{\frac{m}{2}}\exp\pare{-\frac{(m-2)\Theta}{2(1+\Theta)}} \\ 
 > \ &  k\pare{\frac{C^2 \log m}{32m}-\frac 1k \sqrt{\frac{\pi m}{2}}\exp\pare{-\frac{C^2(m-2)k\log m}{128m}}- \frac{\sqrt{\pi}}{2}\sqrt{\frac{m}{2}}\exp\pare{-\frac{(m-2)C\sqrt{k \log m /m}}{2(1+C\sqrt{k \log m /m})}}} \\
 \ge \ & k\pare{\frac{C^2 \log m}{32m}-\sqrt{\frac{\pi m}{2}}\exp\pare{-\frac{C^2(m-2)\log m}{128m}}- \frac{\sqrt{\pi}}{2}\sqrt{\frac{m}{2}}\exp\pare{-\frac{(m-2)C\sqrt{ \log m /m}}{2(1+C\sqrt{ \log m /m})}}} \\
 = \ & k \frac{\log m}{m} \pare{\frac{C^2}{32}-F(C,m)},
\end{align*}
where, to simplify the notation, we let
\begin{align*}
    F(C,m):= \frac{m}{\log m}\sqrt{\frac{\pi m}{2}}\exp\pare{-\frac{C^2(m-2)\log m}{128m}}+\frac{m}{\log m} \frac{\sqrt{\pi}}{2}\sqrt{\frac{m}{2}}\exp\pare{-\frac{(m-2)C\sqrt{ \log m /m}}{2(1+C\sqrt{ \log m /m})}}.
\end{align*}
It then suffices to show that, for every $m \ge 2$, we have $C^2/32 > F(C,m)$ for some constant $C$ large enough.
It can be checked that for every $m \ge 2$, $F(C,m)$ is a decreasing function in $C$. 
Also it can be checked that there is some threshold $C'>0$  such that if $C \ge C'$ then $\lim_{m \to \infty}F(C,m)=0$. 
This implies that $\sup \{F(C,m) \mid C \ge C',m\ge 2\} = \sup \{F(C',m) \mid m\ge 2\}<\infty$. 
Therefore it suffices to choose $C> C'$ large enough so that $C^2/32>\sup \{F(C,m) \mid C \ge C',m\ge 2\}$.
\qed

\subsection{Proof of \cref{th main5}}

For every $i \in [k]$, let $\mu_i := \mu + c_i$.
We first show that for every $i \in [k]$ and for every $z \in \bigcup_{j \in [k]} \ball_1(c_j) \setminus \{c_j\}_{j \in [k]}$ such that $\ball_1(z) \cap \ball_1(c_i)$ has positive measure, we have 
\begin{align}
\label{eq first target}
    \int_{\ball_1(z) \cap \ball_1(c_i)} (d(z,x)-d(c_i,x)) d\mu_i(x)>0.
\end{align}
Let $H$ be the unique hyperplane that contains $\sphere_1(z) \cap \sphere_1(c_i)$ (see \cref{fig: tech}).
We obtain that the balls $\ball_1(z)$ and $\ball_1(c_i)$ are the reflection of each other with respect to $H$. 
\begin{figure}[htbp]
    \centering
    \includegraphics[scale=0.11]{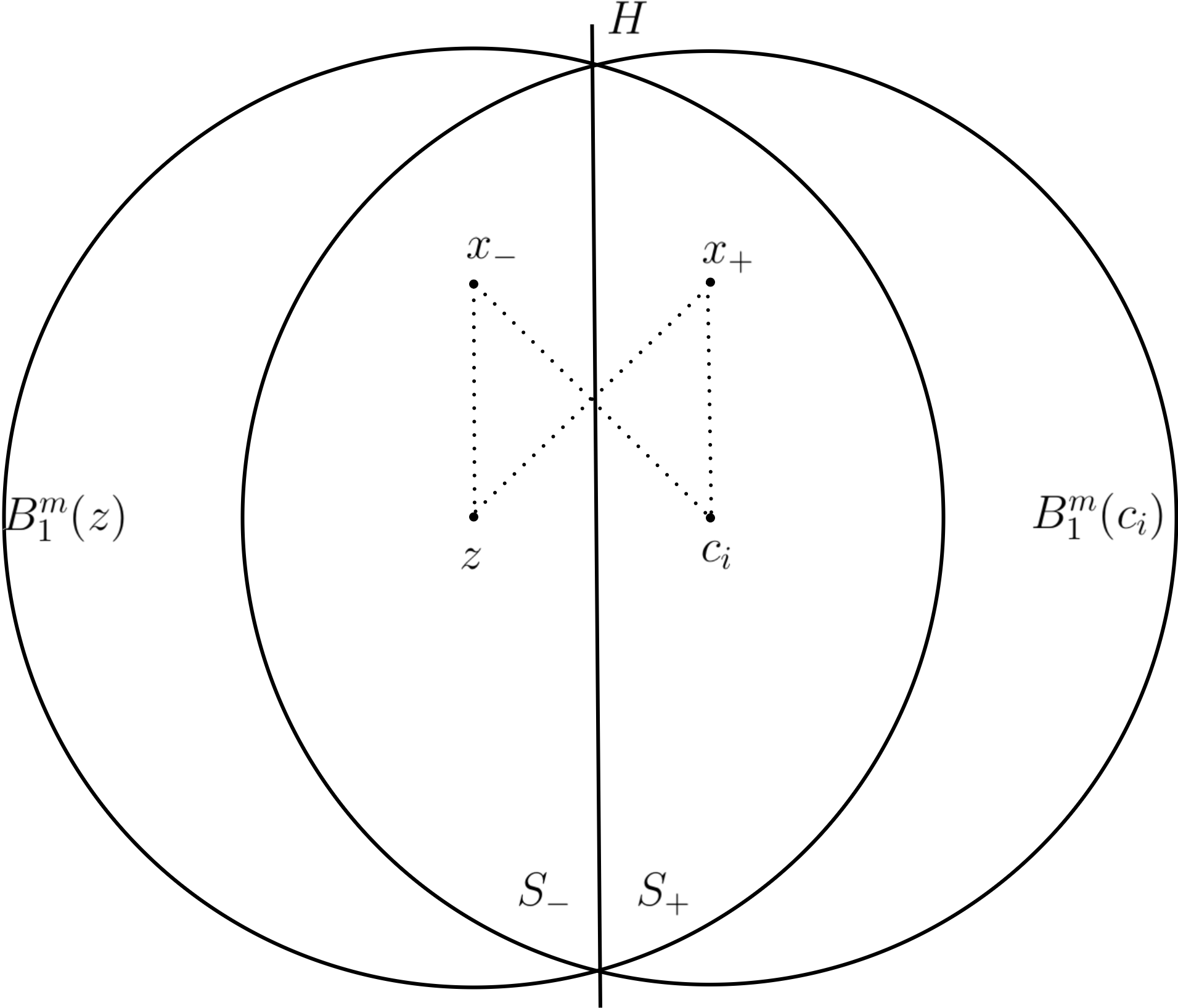}
    \caption{The balls $\ball_1(z)$ and $\ball_1(c_i)$ and the hyperplane $H$ in the proof of \cref{th main5}.}
    \label{fig: tech}
\end{figure}
Let $f(x): \ball_1(z) \cap \ball_1(c_i) \to \ball_1(z) \cap \ball_1(c_i)$ be the function that reflects $x$ with respect to $H$.
Let $S_+:= \{x \in \ball_1(z) \cap \ball_1(c_i) \mid d(z,x)-d(c_i,x) > 0\}$ and $S_-:=\{x \in \ball_1(z) \cap \ball_1(c_i) \mid d(z,x)-d(c_i,x) < 0\}$.
Let $x_+ \in S_+$ and let $x_-:=f(x)$.
We then have $x_- \in S_-$ since $d(z,x_-) - d(c_i,x_-)=d(x_+,c_i)-d(x_+,z)<0$. 
Let $p_i(x)$ be the density function of $\mu_i(x)$.
Since $d(x_+,c_i)<d(c_i,x_-)$, the assumption of the theorem on $p(x)$ implies that we have $p_i(x_+)>p_i(x_-)$. 
We obtain
\begin{align*}
    & \int_{\ball_1(z) \cap \ball_1(c_i)} (d(z,x)-d(c_i,x)) d\mu_i(x) \\
    = \ & \int_{S_+} \pare{d(z,x)-d(c_i,x)}p_i(x)dx+\int_{S_-} \pare{d(z,x)-d(c_i,x)}p_i(x)dx  \\
    = \ & \int_{S_+} \pare{d(z,x)-d(c_i,x)}(p_i(x)-p_i(f(x)))dx>0.
\end{align*}
This concludes the proof of \eqref{eq first target}.

Next we show that for every $i \in [k]$ and for every $z \in \bigcup_{j \in [k]} \ball_1(c_j) \setminus \{c_j\}_{j \in [k]}$ such that $\ball_1(z) \cap \ball_1(c_i)$ has positive measure, we have
\begin{align}
\label{eq second target}
\int_{\ball_{\Tilde{\alpha}}(z) \cap \ball_1(c_i)} (d(z,x)-d(c_i,x)) d\mu_i(x)>0 \qquad \forall \Tilde{\alpha}>1.
\end{align}
Let $\Tilde{S}_+:= \{x \in \ball_{\Tilde{\alpha}}(z) \cap \ball_1(c_i) \mid d(z,x)-d(c_i,x) > 0\}$ and $\Tilde{S}_-:= \{x \in \ball_{\Tilde{\alpha}}(z) \cap \ball_1(c_i) \mid d(z,x)-d(c_i,x) < 0\}$.
Furthermore, let $S_+$ and $S_-$ be defined as in the proof of \eqref{eq first target}.
Clearly we have $S_+ \subseteq \Tilde{S}_+$. 
We observe that $S_-=\Tilde{S}_-$. 
This is because $S_- \subseteq \Tilde{S}_-$ and for every $x \in \ball_1(c_i)$ with $d(z,x)>1$ we must have $d(z,x)-d(c_i,x)>0$, which implies $x \not \in \Tilde{S}_-$.
Then we have 
\begin{align*}
     & \int_{\ball_{\Tilde{\alpha}}(z) \cap \ball_1(c_i)} (d(z,x)-d(c_i,x)) d\mu_i(x) \\
     = \ & \int_{\Tilde{S}_+} \pare{d(z,x)-d(c_i,x)}p_i(x)dx+\int_{S_-} \pare{d(z,x)-d(c_i,x)}p_i(x)dx \\
     \ge \ & \int_{S_+} \pare{d(z,x)-d(c_i,x)}p_i(x)dx+\int_{S_-} \pare{d(z,x)-d(c_i,x)}p_i(x)dx \\
     = \ & \int_{\ball_1(z) \cap \ball_1(c_i)} (d(z,x)-d(c_i,x)) d\mu_i(x)>0,
\end{align*}
where the last inequality holds by \eqref{eq first target}.
This concludes the proof of \eqref{eq second target}.

Next, we claim that there exists $\epsilon \in (0, \min_{i \neq j}d(c_i,c_j)-2)$ such that, for every $i \in [k]$ and for every $z \in \ball_1(c_i)$, there exists a set $D_j$, for each $j \in [k]$, obtained from $\ball_{1+\epsilon}(z) \cap \ball_1(c_j)$ via a rotation centered in $c_j$ followed by the translation $c_j - c_i$, such that the sets $D_j$, for $j \in [k]$, do not intersect.
We now prove our claim.
Let $i \in [k]$ and $z \in \ball_1(c_i)$.
Note that, since the balls $\ball_1(c_j)$, for $j \in [k]$, do not intersect, we have that the sets $\ball_1(z) \cap \ball_1(c_j)$, for $j \in [k]$, do not intersect.
Let $H$ be the unique hyperplane that contains $\sphere_1(z) \cap \sphere_1(c_i)$.
It follows that also the reflections with respect to $H$ of the sets $\ball_1(z) \cap \ball_1(c_j)$, for $j \in [k]$, do not intersect.
Note that the reflection with respect to $H$ of each set $\ball_1(z) \cap \ball_1(c_j)$ can be seen as the set obtained from $\ball_1(z) \cap \ball_1(c_j)$ by first applying a rotation centered in $c_j$ and then the translation $c_j - c_i$.
Hence, we have shown that there exists a set $D_j$, for each $j \in [k]$, obtained from $\ball_1(z) \cap \ball_1(c_j)$ via a rotation centered in $c_j$ followed by the translation $c_j - c_i$, such that the sets $D_j$, for $j \in [k]$, do not intersect.
By continuity, for every $i \in [k]$ and for every $z \in \ball_1(c_i)$, there exists $\epsilon_{i,z}>0$ small enough such that there exists a set $D_j$, for each $j \in [k]$, obtained from $\ball_{1+\epsilon_{i,z}}(z) \cap \ball_1(c_j)$ via a rotation centered in $c_j$ followed by the translation $c_j - c_i$, such that the sets $D_j$, for $j \in [k]$, do not intersect.
Since $\cup_{i \in [k]} \ball_1(c_i)$ is a compact set, we can define $\epsilon:= \min \{\epsilon_{i,z} \mid i \in [k], z \in \ball_1(c_i)\} >0$. 
By eventually decreasing $\epsilon$, we can also assume $\epsilon < \min_{i \neq j}d(c_i,c_j)-2$, and this concludes the proof of our claim.

Let $\alpha':=1+\epsilon<\min_{i \neq j}d(c_i,c_j)-1$ and define $\alpha_i:=\alpha'$ for every $i \in [k]$. 
In order to apply \cref{cor prob}, it remains to show that for every $i \in [k]$, $c_i$ is the unique point that achieves $\max \{G^\alpha(z) \mid z \in \ball_{1}(c_i)\}$.  
We now fix $i \in [k]$ and $z \in \ball_1(c_i) \setminus \{c_i\}$.
For every $j \in [k]$, let $D_j$ be the set obtained from $\ball_{\alpha'}(z) \cap \ball_1(c_j)$ as stated in the previous claim.
Note that $D_j \subseteq \ball_1(c_i)$. 
Since $\mu_i$ is a translation of $\mu_j$, we know that 
\begin{align}
\label{eq lastonetouse}
    \int_{\ball_{\alpha'}(z) \cap \ball_1(c_j)} d(c_j,x) d\mu_j(x)= \int_{D_j} d(c_i,x) d\mu_i(x).
\end{align}
We obtain
\begin{align}
\label{eq trulylastone}
\begin{split}
    G^\alpha(c_i)&
    = \sum_{j \in [k]} \int_{\ball_{\alpha'}(c_i) \cap \ball_1(c_j)} (\alpha'-d(c_i,x)) d\mu_j(x)
    =\int_{\ball
    _1(c_i)} (\alpha'-d(c_i,x)) d\mu_i(x) \\
    & > \sum_{j \in [k]} \int_{D_j} (\alpha'- d(c_i,x)) d\mu_i(x)
    = \sum_{j \in [k]} \int_{\ball_{\alpha'}(z) \cap \ball_1(c_j)} (\alpha'-d(c_j,x)) d\mu_j(x).
\end{split}
\end{align}
In the first equality we use \cref{obs G}, in the second equality \cref{lm geometry} (with $a_i=b_i=\alpha'$ and $z=c_i$), in the inequality we use the fact that the sets and $D_j$, for $j \in [k]$ are disjoint subsets of $\ball_1(c_i)$, and in the last equality we use \eqref{eq lastonetouse}.
Therefore from \eqref{eq trulylastone} and \cref{obs G} we obtain
\begin{align*}
    G^\alpha(c_i)-G^\alpha(z)
    >\sum_{j \in [k]} \int_{\ball_{\alpha'}(z) \cap \ball_1(c_j)} (d(z,x)-d(c_j,x)) d\mu_j(x)>0,
\end{align*}
where the inequality follows from \eqref{eq second target}.
\qed

	\section{Numerical experiments}
	\label{sec num}
	In this section, we perform two sets of numerical experiments to illustrate the empirical performance of \eqref{pr LP} under the SBM and the ESBM.

	In the first set of experiments, we consider the ESBM. Our goal is to show under the ESBM, even if balls have different radii and different probability measures, exact recovery can still happen.
	We draw $N=20$ data points $\{x^{(1)}_i\}_{i=1}^N$ uniformly from $B^2_1(0)$ and $N$ data points $\{x^{(2)}_i\}_{i=1}^N$ uniformly from $B^2_R(\Delta,0)$. We take $P=\{x^{(i)}_j \mid i \in \{1,2\}, \ j \in [N]\}$ as the set of data points and perform an experiment by solving the corresponding \eqref{pr LP}. We say that an experiment succeeds, if the \eqref{pr LP} achieves exact recovery. In our experiments, $\Delta \in [2, 4]$ and $R \in [1, 3]$. For each fixed pair of parameters $(\Delta,R)$, we perform $10$ independent experiments and compute the empirical probability of success.
	\begin{figure}[htbp]
		\centering
		\includegraphics[scale=0.7]{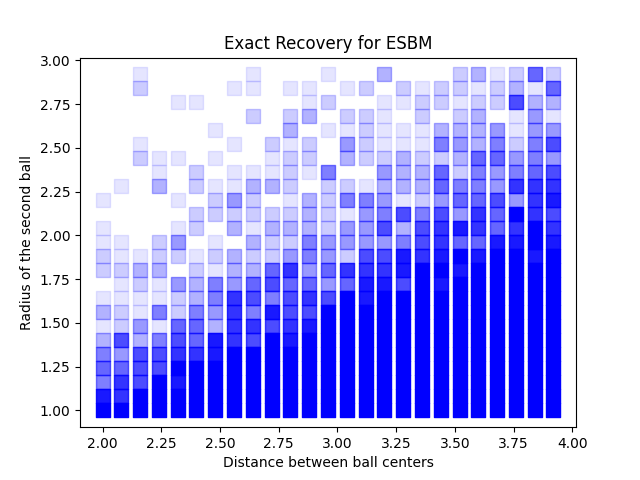}
		\caption{Numerical experiments for \cref{th main2}. We plot the empirical probability of success of \eqref{pr LP} under the ESBM for two balls, with parameters $N=20$, $\Delta \in [2, 4]$, and $R \in [1, 3]$. Deeper color represents higher probability. }
		\label{fig Experiment 2}
	\end{figure}
	\cref{fig Experiment 2} shows the change of empirical probability of success according to the parameter pairs $(\Delta,R)$.
	In \cref{fig Experiment 2}, each point below the diagonal of the figure corresponds to a pair of parameters $(\Delta,R)$ such that the two balls $B^2_1(0)$ and $B^2_R(\Delta,0)$ are separated. In such case, it is clear from \cref{fig Experiment 2} that \eqref{pr LP} achieves exact recovery with high probability. Also,
	when we fix the parameter $\Delta$, the probability of success becomes lower as $R$ increases, which implies that a larger separation is needed when the radii of the two balls are significantly different.

	In the second set of experiments, we consider the SBM. Our goal is to show that, when the balls are significantly separated, \eqref{pr LP} achieves exact recovery with high probability; furthermore, such a significant separation is also necessary.  We construct a probability measure $\mu$ over $B^2_{1}(0)$, that is invariant under rotations centered in $0$. We let $x$ be a point drawn according to $\mu$. With probability $0.9$, $\norm{x}$ is uniformly distributed over $[0.99,1]$, and with probability $0.1$, $\norm{x}$ is uniformly distributed over $[0,0.99]$. We first draw $N$ points $\{x^{(0)}_i\}_{i=1}^N$ according to $\mu$. For $i \in [N]$, we take $x^{(1)}_i:=x^{(0)}_i+(\Delta,0)$ as an input data point in $B^2_1(\Delta,0)$. Next, we draw $N$ more input points $\{x^{(2)}_i\}_{i=1}^N$ according to $\mu$. We take $P=\{x^{(i)}_j \mid i \in \{1,2\}, \ j \in [N]\}$ as the set of input data points and we perform an experiment by solving the corresponding \eqref{pr LP} as we discussed above. Let $\Delta \in [2, 4]$ and $N \in [5, 35]$. For each fixed pair of parameters $(\Delta,N)$, we perform $10$ independent experiments and we compute the empirical probability of success. 
	\begin{figure}[htbp]
		\centering
		\includegraphics[scale=0.7]{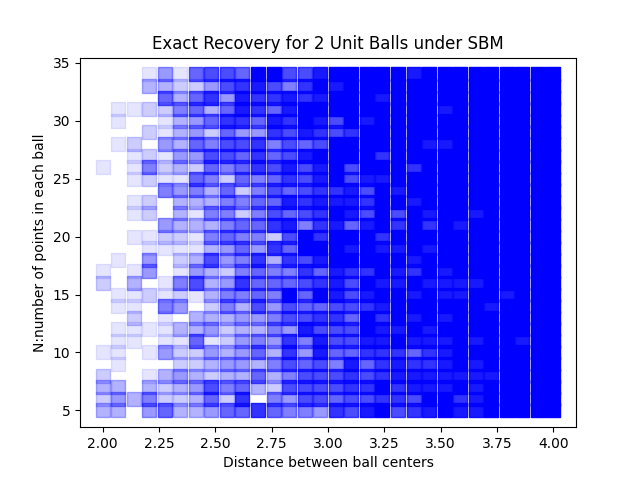}
		\caption{Numerical experiments for \cref{th main3}. We plot the empirical probability of success of \eqref{pr LP} under the SBM for two unit balls, with parameters $\Delta \in [2, 4]$ and $N \in [5, 35]$. Deeper color represents higher probability. }
		\label{fig Experiment 1}
	\end{figure}
	We can see from \cref{fig Experiment 1} that, as stated in \cref{th main3}, when $\Delta>3.29$, \eqref{pr LP} achieves exact recovery with high probability and almost all experiments succeed. However, exact recovery does not happen very often if two balls are not separated significantly. In fact, in \cref{fig Experiment 1}, a phase transition happens around $\Delta\approx 2.25$. When $\Delta>2.25$, as $N$ increases, the probability of success becomes higher. On the contrary, when $\Delta<2.25$, as $N$ increases, the probability of success becomes even lower. In particular, we can see that no experiment succeeds when $\Delta<2.25$ and $N>30$. This implies that for \eqref{pr LP} to succeed with high probability, a significant separation between ball centers is not only sufficient but also necessary.

\ifthenelse {\boolean{MPA}}
{
\bibliographystyle{spmpsci}
}
{
\bibliographystyle{plain}
}

\bibliography{biblio}

\newpage

\begin{appendices}
\crefalias{section}{appendix}

\noindent
{\LARGE\bfseries Online Supplemental Material}

\section{On the assumption $n_i = \beta_i n$ in the ESBM}
\label{app same order}

In this section, we present an example which justifies the assumption that in the ESBM the number $n_i$ of points drawn from each ball $i \in [k]$ satisfies $n_i = \beta_i n$.
In fact, \cref{ex diff order} shows that if in the SBM we allow to draw different numbers $n_i$ of data points from different balls, and the $n_i$ are of different orders, then with high probability \eqref{pr LP} does not achieve exact recovery, no matter how distant the balls are.
In the following example we denote by $e_1, \dots, e_m$ the vectors of the standard basis of $\R^m$.

\begin{example}
\label{ex diff order}
Consider the SBM with $k=2$.
Let $c_1:=0$ and let $c_2:=d e_1$ where $d>2$. 
Let $\mu$ be the uniform probability measure on $\ball_1(0)$. 
For each $i \in [2]$, we draw $n_i(n)$ random vectors instead of $n$ as in the definition of the SBM, and we 
assume that $\lim_{n\to \infty}n_1/n_2=\infty.$ 
Then with high probability every feasible solution to \eqref{pr IP} that assigns each point to the ball from which it is drawn is not optimal to \eqref{pr IP}.
\end{example}

\begin{prf}
For every $i \in [k]$, we denote by $x^{(i)}_*$ the median of $\{x^{(i)}_\ell\}_{\ell \in [n_i]}$.
Among all the feasible solutions $(y,z)$ to \eqref{pr IP} that assign each point to the ball from which it is drawn, the ones with the smaller objective function have the property that, for every $i \in [k]$, the component of the vecror $y^*$ corresponding to $x^{(i)}_*$ is equal to one.
Let $(y^*,z^*)$ be such a solution.
It then suffices to show that with high probability $(y^*,z^*)$ is not optimal to \eqref{pr IP}.

We first evaluate the objective value $\obj^*$ of $(y^*,z^*).$ 
We have 
\begin{align*}
    \frac{\obj^*}{n_1+n_2} 
    & =\frac{\sum_{\ell \in [n_1]}d(x^{(1)}_\ell,x^{(1)}_*)}{n_1}\frac{n_1}{n_1+n_2}+\frac{\sum_{\ell \in [n_2]}d(x^{(2)}_\ell,x^{(2)}_*)}{n_2}\frac{n_2}{n_1+n_2} \\
    & >\frac{\sum_{\ell \in [n_1]}d(x^{(1)}_\ell,x^{(1)}_*)}{n_1}\frac{n_1}{n_1+n_2}.
\end{align*}
Let $x$ be a random vector drawn according to $\mu$. 
Since $\mu$ is the uniform probability measure, we know that $\E \norm{x} \in (0,1)$. 
Let $\epsilon \in (0,1)$ be a small number.
From \cref{lm opt}, we know that with high probability we have   $\abs{\sum_{\ell \in [n_1]}d(x^{(1)}_\ell,x^{(1)}_*)/n_1-\E \norm{x}}<\epsilon$. 
Since $\lim_{n \to \infty}n_1/n_2=\infty$, we know that when $n$ is large enough, with high probability, we have 
\begin{align*}
    \frac{\obj^*}{n_1+n_2}>\frac{\sum_{\ell \in [n_1]}d(x^{(1)}_\ell,x^{(1)}_*)}{n_1}\frac{n_1}{n_1+n_2}>\E \norm{x}-2\epsilon.
\end{align*}

Consider now the point $s:=-e_1/2 \in \ball_1(0)$, and define the sets $S_1:=\{x \in \ball_1(0) \mid x_1 \le -\frac{1}{2}\}$ and $S_2:=\{x \in \ball_1(0) \mid x_1 > -\frac{1}{2}\}$. 
Let $x$ be a random vector drawn according to $\mu$. 
We know that when $x \in S_1$, we have $d(s,x)<\norm{x}$. 
To simplify the notation, let $\xi:= \min\{\norm{x}-d(s,x) \mid x \in S_1\}>0$. 
Then we have 
\begin{align}
\label{eq expect}
\begin{split}
    \E \min\{d(s,x), \norm{x}\}\ & =\int_{\ball_1(0)}\min\{d(s,x),\norm{x}\} d\mu(x) \le \int_{S_1}d(s,x)d\mu(x)+ \int_{S_2}\norm{x}d\mu(x) \\
    \ & \le \int_{\ball_1(0)}\norm{x}d\mu(x)-\xi \Pr(x \in S_1)= \E \norm{x}-\xi \Pr(x \in S_1).
\end{split}
\end{align}
Note that $\min\{d(x^{(1)}_\ell,s),\norms{x^{(1)}_\ell}\}$, for $\ell \in [n_1]$, are independent random variables bounded by the interval $[0,1]$.
From Hoeffding's inequality, with high probability we have 
\begin{align*}
    \abs{\frac{\sum_{\ell \in [n_1]}\min\{d(x^{(1)}_\ell,s),\norms{x^{(1)}_\ell}\}}{n_1}-\E \min\{d(s,x), \norm{x}\}}<\epsilon.
\end{align*}
So with high probability, we obtain 
\begin{align}
\label{eq to be used later}
    \frac{\sum_{\ell \in [n_1]}\min\{d(x^{(1)}_\ell,s),\norms{x^{(1)}_\ell}\}}{n_1}
    <\E \min\{d(s,x), \norm{x}\}+\epsilon
    <\E \norm{x}-\xi \Pr(x \in S_1)+\epsilon,
\end{align}
where the last inequality follows from \eqref{eq expect}.
Since $\mu$ is the uniform probability measure on $\ball_1(0)$, we know that with high probability there is some point $\Tilde{x} \in \ball_{\epsilon}(s)\cap \{x^{(1)}_\ell\}_{\ell \in [n_1]}$ and there is some point $\Bar{x} \in \ball_{\epsilon}(0)\cap \{x^{(1)}_\ell\}_{\ell \in [n_1]}$. 
Now we construct a feasible solution $(y',z')$ to \eqref{pr IP} with objective value $\obj'$ such that $\obj'<\obj^*$. 
We choose $\Tilde{x}$ and $\Bar{x}$ as the centers of the two clusters.
We then assign each point $x^{(i)}_\ell$, for $i \in [2]$ and  $\ell \in [n_i]$, to the cluster with the closest center.
Let $(y',z')$ be the feasible solution to \eqref{pr IP}  corresponding to this choice.
Then we have
\begin{align*}
    \frac{\obj'}{n_1+n_2}\ & =\frac{\sum_{\ell \in [n_1]}\min\{d(x^{(1)}_\ell,\Tilde{x}),d(x^{(1)}_\ell,\Bar{x})\}}{n_1}\frac{n_1}{n_1+n_2}+\frac{\sum_{\ell \in [n_2]}\min\{d(x^{(2)}_\ell,\Tilde{x}),d(x^{(2)}_\ell,\Bar{x})\}}{n_2}\frac{n_2}{n_1+n_2} \\
    \ & \le \frac{\sum_{\ell \in [n_1]}\min\{d(x^{(1)}_\ell,\Tilde{x}),d(x^{(1)}_\ell,\Bar{x})\}}{n_1}\frac{n_1}{n_1+n_2}+(d+2)\frac{n_2}{n_1+n_2}.
\end{align*}
Here, the inequality follows because $\min\{d(x^{(2)}_\ell,\Tilde{x}),d(x^{(2)}_\ell,\Bar{x})\} \le d+2$ for every $\ell \in [n_2]$. 
Next, we use the fact that when $n$ is large enough $(d+2)n_2/(n_1+n_2)$ can be arbitrarily small and thus can be bounded by $\epsilon$.
So with high probability we obtain
\begin{align*}
    \frac{\obj'}{n_1+n_2}\ & \le \frac{\sum_{\ell \in [n_1]}\min\{d(x^{(1)}_\ell,\Tilde{x}),d(x^{(1)}_\ell,\Bar{x})\}}{n_1}+ \epsilon
    \le \frac{\sum_{\ell \in [n_1]}\min\{d(x^{(1)}_\ell,s),\norms{x^{(1)}_\ell}\}+\epsilon}{n_1}+ \epsilon \\
     \ & \le \E \norm{x}-\xi \Pr(x \in S_1)+3\epsilon,
\end{align*}
where the second inequality follows by the triangle inequality, and in the last inequality we use \eqref{eq to be used later}.

Notice that when $\epsilon<\xi\Pr(x \in S_1)/5$, we have $\obj'<\obj^*$, which implies that with high probability, $(y^*,z^*)$ is not an optimal solution to \eqref{pr IP}.
\end{prf}

\section{Counterexample to Theorem~7 in \cite{awasthi2015relax}}
\label{app counter}

In this section we present an example which shows that Theorem~7 in \cite{awasthi2015relax} is false. 
In \cref{ex counter}, we construct a probability measure that satisfies all assumptions in the statement of Theorem~7 in \cite{awasthi2015relax} and $\min_{i \neq j}d(c_i,c_j)=2.2$.
We then show that with high probability \eqref{pr LP} does not achieve exact recovery. 
The key problem in the proof of Theorem~7 in \cite{awasthi2015relax} is discussed in \cref{sec problem}.

\begin{example}
\label{ex counter}
There is an instance of the SBM with $m=2$, $k=7$, $\min_{i \neq j} d(c_i,c_j) = 2 + 0.2$, where $\mu$ has a continuous density function and the probability space $(\mu,\ball_1(0))$ satisfies \ref{ass rotation}, \ref{ass open}, such that with high probability \eqref{pr LP} does not achieve exact recovery.
\end{example}


\begin{prf}
Let $\mu$ be a probability measure that has a continuous density function and the probability space $(\mu,\ball_1(0))$ satisfies \ref{ass rotation}, \ref{ass open}.
Let $\epsilon \in (0,1)$ be a small number. 
We further assume that $\mu$ and $\epsilon$ satisfy
\begin{align}
\label{eq ex2 assumption}
(0.292-8\epsilon) \Pr(\norm{x} \ge 1-\epsilon)>0.279+6\epsilon+(3+2\epsilon)\Pr(\norm{x}<1-\epsilon).
\end{align}
Note that assumption \eqref{eq ex2 assumption} can be fulfilled as long as $\Pr(\norm{x}<1-\epsilon)$ and $\epsilon$ are small enough.
We define $c_1:=0$ and, using polar coordinates, $c_i:=(2.2,-(i-2)\pi/3)$ for every $i \in [7] \setminus \{1\}$ (see \cref{fig new counter example 2}). 
\begin{figure}[htbp]
    \centering
    \includegraphics[scale=0.13]{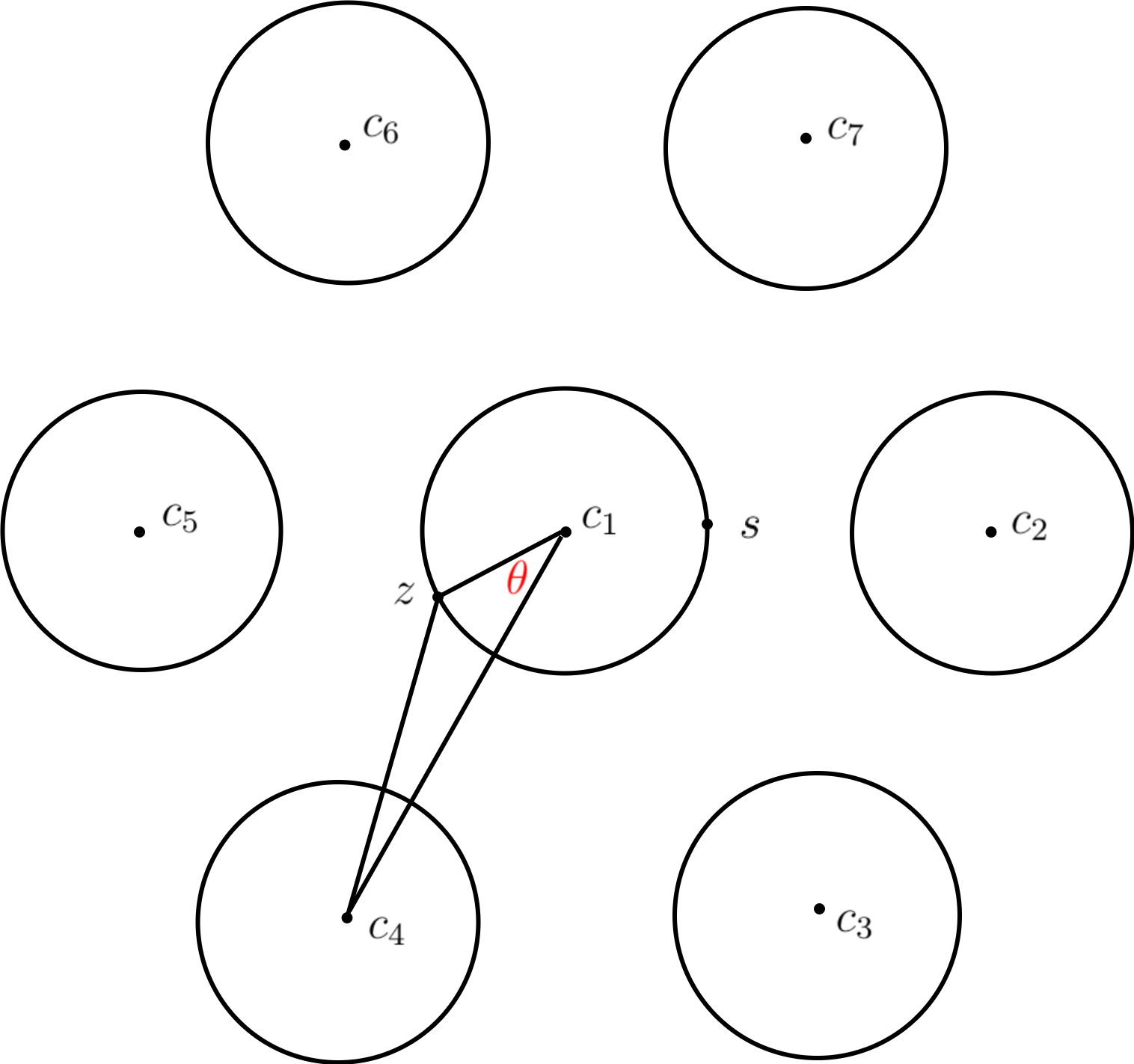}
    \caption{Instance of the SBM considered in \cref{ex counter}.}
    \label{fig new counter example 2}
\end{figure}
In particular, for every $i, j \in [7]$ with $i \neq j$, we have $d(c_i,c_j)=2.2$. 
This concludes the description of the instance of the SBM that we consider.
In the remainder of the example we show that with high probability \eqref{pr LP} does not achieve exact recovery.

For every $z \in \ball_1(0)$, let $c(z)$ be a point among $c_2,\dots,c_7$ that is closest to $z$ and we define $f(z):=d(z,c(z))$.
For every $z \in \ball_1(0)$, let $\theta(z)$ be the angle between the vectors $z$ and $c(z)$. 
Clearly, for every $z \in \ball_1(0)$, we have $\theta(z) \in [0,\pi/6]$. 
Let $z$ be a random vector drawn according to $\mu$.
Since $\mu$ satisfies \ref{ass rotation}, we know that the random variable $\theta(z)$ is uniform on $[0,\pi/6]$, thus its density function is constant on $[0,\pi/6]$ and equal to $6/\pi$. 

Next, we show the upper bound
\begin{align}
\label{eq ex2 ub}
\int_{\ball_1(0)} (f(x)+2\epsilon-\norm{x}) d\mu(x)
< 0.279+4\epsilon+(3+2\epsilon)\Pr(\norm{x}<1-\epsilon).
\end{align}
Let $z \in \sphere_1(0)$, then $f(z)=\sqrt{(2.2)^2+1-4.4\cos\theta}.$ Let $L:=\ball_1(0) \setminus \ball_{1-\epsilon}(0)$. 
The triangle inequality implies that for every $z \in L$ we have
\begin{align}\label{eq point in L}
    f(z)<\sqrt{(2.2)^2+1-4.4\cos\theta(z)}+\epsilon, 
    \qquad d(z,0) \ge 1-\epsilon.
\end{align}
So we obtain
\begin{align*}
    \int_{\ball_1(0)} & (f(x)+2\epsilon-\norm{x}) d\mu(x)
    =  \int_{\ball_{1-\epsilon}(0)}(f(x)+2\epsilon-\norm{x})d\mu(x)+ \int_{L}(f(x)+2\epsilon-\norm{x})d\mu(x) \\ 
    & \le (3+2\epsilon)\Pr(\norm{x}<1-\epsilon)+ \int_{L}(f(x)+2\epsilon-\norm{x})d\mu(x) \\
    & \le (3+2\epsilon)\Pr(\norm{x}<1-\epsilon)+  \int_L\pare{\sqrt{(2.2)^2+1-4.4\cos\theta(x)}+\epsilon+2\epsilon-(1-\epsilon)}d\mu(x) \\
    & \le  (3+2\epsilon)\Pr(\norm{x}<1-\epsilon)+\frac{6}{\pi}  \int_0^{\frac{\pi}{6}}\pare{\sqrt{(2.2)^2+1-4.4\cos\theta}+\epsilon+2\epsilon-(1-\epsilon)}d\theta \\
    & = (3+2\epsilon)\Pr(\norm{x}<1-\epsilon)+\frac{6}{\pi}\int_0^{\frac{\pi}{6}}\pare{\sqrt{(2.2)^2+1-4.4\cos\theta}-1} d\theta+4\epsilon \\
    & < 0.279+4\epsilon+(3+2\epsilon)\Pr(\norm{x}<1-\epsilon).
\end{align*}
Here, the first inequality uses the fact that $f(x) \le 3$ for every $x \in \ball_{1-\epsilon}(0)$ and the second inequality holds because of \eqref{eq point in L}. 
The third inequality follows by the fact that $\theta(x)$ does not depend on $\norm{x}$ and has a density function $\pi/6$.
In the last inequality, we use the fact that 
\begin{align*}
    \frac{6}{\pi}\int_0^{\frac{\pi}{6}}\pare{\sqrt{(2.2)^2+1-4.4\cos\theta}-1}d\theta < 0.279.
\end{align*}
This concludes the proof of \eqref{eq ex2 ub}.

Let $s:=e_1$, where $e_1$ is the first vector of the standard basis of $\R^m$, and let $\mu_i := \mu + c_i$ for every $i \in [k]$.
Next, we prove the lower bound 
\begin{align}
\label{eq ex2 lb}
\sum_{i=1}^7\int_{\ball_1(c_i)}(d(x,c_i)-2\epsilon-d(x,s))_+d\mu_i(x)
>  (0.292-8\epsilon) \Pr(\norm{x} \ge 1-\epsilon). 
\end{align}
For ease of notations we give the following definitions.
For every $x \in \ball_1(0)$, let $\psi(x)$ be the angle between the vectors $x$ and $s$. 
For every $x \in \ball_1(c_2)$, let $\phi(x)$ be the angle between $x-c_2$ and $s-c_2$.
We also define $L_1:=\{x \in \ball_1(0) \mid \psi(x) \le \pi/3, \norm{x} \ge 1-\epsilon\}$ and $L_2:=\{x \in \ball_1(c_2) \mid \phi(x) \le \theta', d(x,c_2) \ge 1-\epsilon\},$ where $\theta':=\arccos{0.6}$.

Notice that for every $x \in \sphere_1(0)$, we have $d(x,s)=\sqrt{2-2\cos\psi(x)}$ and for every $x \in \sphere_1(c_2)$, we have $d(x,s)=\sqrt{(1.2)^2+1-2.4\cos\phi(x)}$. 
Using the triangle inequality, we obtain
\begin{align}
\label{eq L11}
    d(x,s) & \le \sqrt{2-2\cos\psi(x)}+\epsilon & \forall x \in L_1, \\
\label{eq L21}
    d(x,s) & \le \sqrt{(1.2)^2+1-2.4\cos\phi(x)}+\epsilon & \forall x \in L_2.
\end{align}
We obtain
\begin{align*}
& \sum_{i=1}^7\int_{\ball_1(c_i)} (d(x,c_i) -2\epsilon-d(x,s))_+d \mu_i(x) \\
\ge \ & \int_{L_1} (\norm{x}-2\epsilon-d(x,s))_+  d\mu(x) + \int_{L_2}(d(x,c_2)-2\epsilon-d(x,s))_+  d\mu_2(x)\\
\ge \ & \int_{L_1} (1-3\epsilon-d(x,s))  d\mu(x) + \int_{L_2}(1-3\epsilon-d(x,s)  )d\mu_2(x) \\
\ge \ & \int_{L_1}\pare{1-4\epsilon-\sqrt{2-2\cos\psi(x)}} d\mu(x) + \int_{L_2}\pare{1-4\epsilon-\sqrt{(1.2)^2+1-2.4\cos\phi(x)}} d\mu_2(x) \\
= \ & \Pr(\norm{x} \ge 1-\epsilon)\frac{1}{\pi}\int_0^{\frac{\pi}{3}}\pare{1-4\epsilon-\sqrt{2-2\cos\psi}} d\psi \\
& \qquad + \Pr(\norm{x} \ge 1-\epsilon) \frac{1}{\pi}\int_0^{\theta'}\pare{1-4\epsilon-\sqrt{(1.2)^2+1-2.4\cos\phi} }d\phi \\
> \ & (0.292-8\epsilon) \Pr(\norm{x} \ge 1-\epsilon).
\end{align*}
Here, the second inequality follows from the definition of $L_1$ and $L_2$. The third inequality follows by \eqref{eq L11} and \eqref{eq L21}. 
The equality holds because $\psi(x)$ does not depend on $\norm{x}$ and $\phi(x)$ does not depend on $d(c_2,x)$. 
The last inequality holds because
\begin{align*}
\frac{1}{\pi}\int_0^{\frac{\pi}{3}}\pare{1-\sqrt{2-2\cos\psi}}d\psi + \frac{1}{\pi}\int_0^{\theta'}\pare{1-\sqrt{(1.2)^2+1-2.4\cos\phi}}d\phi
>0.292.
\end{align*}
This completes the proof of \eqref{eq ex2 lb}.

Using Hoeffding's inequality, with high probability we have
\begin{align*}
    \frac{1}{n}\sum_{\ell\in[n]}(f(x_\ell^{(1)})+2\epsilon-\norms{x_\ell^{(1)}}) - \int_{\ball_1(0)} \pare{f(x)+2\epsilon-\norm{x}} d\mu(x)
    < \epsilon,
\end{align*}
and using \eqref{eq ex2 ub} with high probability we have
\begin{align}
\label{eq ex2 H1}
\begin{split}
    \frac{1}{n}\sum_{\ell\in[n]}(f(x_\ell^{(1)})+2\epsilon-\norms{x_\ell^{(1)}})
    & < \int_{\ball_1(0)} \pare{f(x)+2\epsilon-\norm{x}} d\mu(x)+\epsilon \\
    & < 0.279+5\epsilon+(3+2\epsilon)\Pr(\norm{x}<1-\epsilon).
\end{split}
\end{align}
Using Hoeffding's inequality, with high probability we have 
\begin{align*}    
    \sum_{i=1}^7\int_{\ball_1(c_i)}(d(x,c_i)-2\epsilon-d(x,s))_+d\mu_i(x)
    - \frac{1}{n}\sum_{i \in [7]}\sum_{\ell\in[n]}(d(x^{(i)}_\ell,c_i)-2\epsilon-d(x^{(i)}_\ell,s))_+ 
    <\epsilon,
\end{align*}
and using \eqref{eq ex2 lb} with high probability we have
\begin{align}
\label{eq ex2 H2}
\begin{split}
    \frac{1}{n}\sum_{i \in [7]}\sum_{\ell\in[n]}(d(x^{(i)}_\ell,c_i)-2\epsilon-d(x^{(i)}_\ell,s))_+ 
    & > \sum_{i=1}^7\int_{\ball_1(c_i)}(d(x,c_i)-2\epsilon-d(x,s))_+d\mu_i(x)-\epsilon \\
    & > (0.292-8\epsilon)\Pr(\norm{x} \ge 1-\epsilon)-\epsilon.
\end{split}
\end{align}

For every $i \in [k]$, we denote by $x^{(i)}_*$ the median of $\{x^{(i)}_\ell\}_{\ell \in [n_i]}$.
Among all the feasible solutions $(y,z)$ to \eqref{pr IP} that assign each point to the ball from which it is drawn, the ones with the smaller objective function have the property that, for every $i \in [k]$, the component of the vecror $y^*$ corresponding to $x^{(i)}_*$ is equal to one.
Let $(y^*,z^*)$ be such a solution.
It then suffices to show that with high probability $(y^*,z^*)$ is not optimal to \eqref{pr LP}.

We know from \cref{lm improve center in one ball} that with high probability, for every $i \in [7]$, we have $d(x^{(i)}_*,c_i)<\epsilon$.
Next, we show that we can use \cref{th deterministic} to prove that $(y^*,z^*)$ is not optimal to \eqref{pr LP} with high probability.
To do so, we just need to show that there is no $\alpha$ that satisfies conditions \eqref{eq Th1 a}--\eqref{eq Th1 d}.

For ease of notation we denote by $\alpha_\ell^{(i)}$ the component of $\alpha$ corresponding to the point $x_\ell^{(i)}$.
Suppose that $\alpha$ satisfies \eqref{eq Th1 c} and \eqref{eq Th1 d}, thus $d(x^{(i)}_*,x^{(i)}_\ell) \le \alpha_\ell^{(i)} \le d(x^{(j)}_*,x^{(i)}_\ell)$ for every $i,j \in [k]$ with $i \neq j$ and for every $\ell \in [n]$. 
Then we have 
\begin{align}
\label{eq ex2 C1}
\begin{split}
    \frac{1}{n}C^\alpha(x^{(1)}_*)
    = \frac{1}{n}\sum_{\ell \in [n]} \pares{\alpha^{(1)}_\ell-d(x^{(1)}_*,x^{(1)}_\ell)} \ & \le \frac{1}{n}\sum_{\ell\in[n]}\pares{f(x_\ell^{(1)})+2\epsilon-\norms{x_\ell^{(1)}}} \\
    \ & < 0.279+5\epsilon+(3+2\epsilon)\Pr(\norm{x}<1-\epsilon),
\end{split}
\end{align}
where in the first inequality we use the fact that $\alpha_\ell^{(1)} \le d(x^{(j)}_*,x^{(1)}_\ell) \le d(c_j,x^{(1)}_\ell)+\epsilon$ for every $j \in [7] \setminus \{1\}$  and $d(x^{(1)}_*,x^{(1)}_\ell) \ge \norms{x_\ell^{(1)}}-\epsilon$, and the second inequality follows from \eqref{eq ex2 H1}.

Let $N:=\ball_{\epsilon}(s) \cap L$ and note that the assumption \eqref{eq ex2 assumption} on $\mu$ imply that with high probability there exists a point $x' \in N \cap \{x^{(i)}_\ell\}_{\ell \in [n_i]} $.
We have 
\begin{align}
\label{eq ex2 C2}
\begin{split}
    \frac{1}{n}C^\alpha(x')= \frac{1}{n}\sum_{i \in [7]}\sum_{\ell \in [n]} (\alpha^{(i)}_\ell-d(x',x^{(i)}_\ell))_+ 
    & \ge \frac{1}{n}\sum_{i \in [7]}\sum_{\ell\in[n]}(d(c_i,x_\ell^{(i)})-2\epsilon-d(s,x^{(i)}_\ell))_+ \\
    & > (0.292-8\epsilon)\Pr(\norm{x} \ge 1-\epsilon)-\epsilon,
\end{split}
\end{align}
where in the first inequality we use that for every $i \in [7]$ we have $\alpha^{(i)}_\ell \ge d(x^{(i)}_*,x^{(i)}_\ell) \ge d(c_i,x_\ell^{(i)})-\epsilon$ and $d(x',x_\ell^{(i)}) \ge d(s,x_\ell^{(i)})-\epsilon$, and the second inequality follows from \eqref{eq ex2 H2}.
The inequalities \eqref{eq ex2 C1} and \eqref{eq ex2 C2} imply $C^\alpha(x')>C^\alpha(x^{(1)}_*)$ due to assumption \eqref{eq ex2 assumption}.
This implies that conditions \eqref{eq Th1 a},\eqref{eq Th1 b} cannot hold.
Thus, according to \cref{th deterministic}, with high probability $\eqref{pr LP}$ does not achieve exact recovery.
\end{prf}

\section{Problem in the proof of Theorem~7 in \cite{awasthi2015relax}}
\label{sec problem}

In this section we point out the key problem in the proof of Theorem~7 in \cite{awasthi2015relax}.
To prove this theorem, the authors introduce two conditions: the separation condition and the central dominance condition.
When the two conditions happen together, then \eqref{pr LP} achieves exact recovery.
We refer the reader to \cite{awasthi2015relax} for more details about these two conditions. 
In the proof of Theorem~7 the authors show that the separation condition happens with high probability according to the law of large number, while the central dominance condition happens in expectation and thus happens with high probability. 
Formally, the authors prove the following lemma about the central dominance condition.
\begin{lemma}
[Lemma 13 in \cite{awasthi2015relax}]
\label{lem wrong}
In the hypothesis of Theorem 7, there exists $\alpha>1$ such that for all $j \in [k]$, $\E P^{(\alpha,\dots,\alpha)}(z)$ restricted to $z \in \ball_1(c_j)$ attains its maximum in $z=c_j$.
\end{lemma}

In the proof of \cref{lem wrong}, the goal of the authors is to obtain some $\alpha>1$ such that $c_i$ achieves $\max \{\E P^{(\alpha,\dots,\alpha)}(z) \mid z \in \ball_1(c_i)\}$ for every $i \in [k]$, where 
\begin{align*}
    \E P^{(\alpha,\dots,\alpha)}(z)=\sum_{i \in [k]} \int_{x \in \ball_{1}(c_i)}(\alpha-d(z,x))_+d\mu_i(x).
\end{align*}
In order to do so, they select some $\alpha>1$ such that for every $i \in [k]$ and for every $z \in \ball_1(c_i)$, the sets $\ball_\alpha(z) \cap \cup_{j \neq i} \ball_1(c_i)$, for $j \in [k] \setminus \{i\}$, can be copied isometrically inside $\ball_1(c_i)$ along the boundary without intersecting each other. 
Their goal is to use the fact that $\ball_1(c_i)$ contains all these copies to show that $\E P^{(\alpha,\dots,\alpha)}(c_i) > \E P^{(\alpha,\dots,\alpha)}(z)$. 
The problem is that, although the copies have the same area of the original sets, the density function may differ from a point $x \in \ball_\alpha(z) \cap \cup_{j \neq i} \ball_1(c_i)$ to the corresponding point $x' \in \ball_1(c_i)$ with $d(z,x)=d(c_i,x')$. 
If the probability measure is anti-concentrated, which means that the area near the boundary of each ball has a very large probability, then the choice of $\alpha$ given by the authors may cause $\E P^{(\alpha,\dots,\alpha)}(c_i) < \E P^{(\alpha,\dots,\alpha)}(z)$ for some $z \in \ball_1(c_i) \setminus\{c_i\}$. 

We also remark that there is also a requirement omitted in the statement of \cref{lem wrong}.
In fact, in the statement the authors require $\alpha>1$.
However, in order to satisfy the central dominance condition, $\alpha$ cannot be chosen too large.
In particular, the requirement $\alpha<1+\Theta$, where $\Theta= \min_{j\neq i}d(c_i,c_j)-2$, should be added to the lemma.

\end{appendices}

\nocite{*}

\end{document}